\documentclass[12pt]{article}
\usepackage{amssymb,amsfonts,amsmath,latexsym
,amsthm,
}
\usepackage[T2A]{fontenc}
\usepackage[cp1251]{inputenc}
\usepackage[english]{babel}

\numberwithin{equation}{section}

\newtheorem{predlll}{Corollary}[section]
\newtheorem{pred}{Theorem}[section]
\newtheorem{predl}{Lemma}[section]
\newtheorem{predllll}{Proposition}[section]
\newtheorem{predll}{Proposition}[section]

\newtheorem*{predd}{Teoрема}

\newtheorem{pr*}{Theorem 1.}

\begin{document}

\sloppy

\thispagestyle{empty}
\begin{center}
{\bf \large O.G. Nakonechnyi \\[3pt] Yu.K. Podlipenko \\[3pt] Yu. V. Shestopalov \\}
\end{center}
\bigskip\bigskip \bigskip

\begin{center}
{\large \bf Estimation of parameters \\[3pt]
of boundary value problems
\\[3pt] for linear ordinary differential equations \\[5pt]
with uncertain data}
\end{center}
\bigskip


\renewcommand{\contentsname}%
{\begin{center}Contents\end{center}}

\newpage

\addtocontents{toc}{\large}

\tableofcontents

\newpage

\begin{abstract}


In this paper we construct optimal, in certain sense, estimates of
values of linear functionals on solutions to two-point boundary
value problems (BVPs) for systems of linear first-order ordinary
differential equations from observations which are linear
transformations of the same solutions perturbed by additive random
noises. It is assumed here that right-hand sides of equations and
boundary data as well as statistical characteristics of random
noises in observations are not known and belong to certain given
sets in corresponding functional spaces.
This leads to the necessity of introducing minimax statement of an
estimation problem when optimal estimates are defined as linear,
with respect to observations, estimates for which the maximum of mean
square error of estimation taken over the above-mentioned sets
attains minimal value. Such estimates are called minimax
estimates.

We establish that the minimax estimates are expressed via
solutions of some systems of differential equations of  special
type.

Similar estimation problems for solutions of BVPs for linear
differential equations of order $n$ with general boundary
conditions are considered.

We also elaborate minimax estimation methods under incomplete data
of unknown right-hand sides of equations and boundary data and
obtain representations for the corresponding minimax estimates.


In all the cases estimation errors are determined.
\end{abstract}

\section*{}

\addcontentsline{toc}{section}{Introduction}

\begin{center}
{\bf \large {Introduction}}
\end{center}

Minimax estimation is studied in a big number of works;
one may refer e.g. to \cite{BIBL5a}--\cite{BIBL13}
and the bibliography therein.

Let us formulate a general approach to the problem. If a state of a system is described by a linear ordinary differential equation
$$
\frac{dx(t)}{dt}=Ax(t)+Bv_1(t),\quad x(t_0)=x_0,
$$
and a function $y(t)$ is observed in a time interval $[t_0,T]$,
where $y(t)=Hx(t)+v_2(t),$ $x(t) \in \mathbb R^n$, $v_2 \in
\mathbb R^m,$ $y\in \mathbb R^m,$ and $A,\,B,\, H$ are known
matrices, the minimax estimation problem consists in the most
accurate determination of a function $x(t)$ at the "worst"
realization of unknown quantities $(x_0,v_1(\cdot),v_2(\cdot))$
taken from a certain set. N.N. Krasovskii was the first who stated
this problem in \cite{BIBL12}. Under different constraints imposed
on function $v_2(t)$ and for known function $v_1(t)$ he proposed
various methods of estimating inner products $(a,x(T))$ in the
class of operations linear with respect to observations that
minimize the maximal error. Later these estimates were called
minimax a priori estimates (see \cite{BIBL12}, \cite{BIBL13}).

Fundamental results concerning estimation under uncertainties were obtained by A. B. Kurzhanskii (see \cite{BIBL13}, \cite{BIBL14}).

 The duality principle elaborated in \cite{BIBL12},
 \cite{BIBL13}, and \cite{BIBL5a} proved its efficiency for the determination of minimax estimates \cite{BIBL5a}.
 According to this principle, finding minimax a priori estimates can be reduced to a certain problem of
 optimal control of a system; this approach enabled one to obtain, under certain restrictions, recurrent equations, namely,
 the minimax Kalman--Bucy filter (see \cite{BIBL5a}).

In this work we consider the problems of minimax estimation of solutions to
two-point boundary value problems (BVPs) for systems of linear
first-order ordinary differential equations.
We find general form of the minimax estimates of solutions
from observations on an interval and determine estimation errors.

In the second part of the work (section 6 and 7) we formulate and
solve the problems of estimation under incomplete data of the
values of linear functionals
from solutions
and right-hand sides of linear  differential equations of order
$n$ with general boundary conditions. Additional difficulties that
arise in the course of the analysis of these estimation problems
are connected with (i) the necessity of imposing certain
solvability conditions on the data (right-hand sides of the
equations and boundary conditions) and (ii) that their solutions
are determined
up to solutions
of the corresponding
homogeneous problems.

\section{Preliminaries and auxiliary results}
Assume that it is given a vector-function $f(t)=(f_1(t),f_2(t)\ldots f_n(t))^T$ with the components belonging to space $L^2(0,T)$ and vectors
$f_0=(f_1^{(0)},f_2^{(0)},\ldots,f_m^{(0)})^T\in \mathbb R^m$ and $f_1=(f_1^{(1)},f_2^{(1)},\ldots,f_{n-m}^{(1)})^T\in \mathbb
R^{n-m}$.
Consider the following BVP: find a vector-function $\varphi(t)=(\varphi_1(t),\varphi_2(t)\ldots,\varphi_n(t))^T\in H^1(0,T)^n$ that satisfies a system of linear first-order ordinary differential equations
\begin{equation}
\label{dy1} \frac{d\varphi(t)}{dt}+A\varphi(t) = f(t),\quad t\in
(0,T),
\end{equation}
almost everywhere on an interval $(0,T)$ and the boundary conditions
\begin{equation}\label{dy2}
B_{0}\varphi(0)=f_0, \quad B_1\varphi(T)=f_1
\end{equation}
at the points $0$ and $T$. Here $A=A(t)$ is an $n\times n$ matrix
with the entries $a_{ij}=a_{ij}(t)$ continuous on $[0,T]$,
$\frac{d\varphi(t)}{dt}=(\frac{d\varphi_1(t)} {dt},\,
\frac{d\varphi_2(t)}{dt}\,
\ldots,\,\,\frac{d\varphi_n(t)}{dt})^T,$ $B_0=\{b^{(0)}_{rs}\},$
$r=\overline{1,m},$ $s=\overline{1,n},$ and $B_1=
\{b^{(1)}_{rs}\},$ $r=\overline{1,n-m},$ $s=\overline{1,n},$ are
$m\times n$ and $(n-m)\times n$ matrices of rank $m$ and $n-m$,
respectively, $T$ denotes transposition, and
 $H^1(a,b)$ is a space of
functions absolutely continuous on $[a,b]$
for which the derivative that exists almost everywhere on $(a,b)$
belongs to space
$L^2(a,b).$

The problem of finding a function $\varphi(t)$ that satisfies on $(0,T)$ the equation
\begin{equation}\label{dy10}
\frac{d\varphi(t)}{dt}+A\varphi(t) = 0,
\end{equation}
and the boundary conditions
\begin{equation}\label{dy20}
 B_{0}\varphi(0)=0, \quad
B_1\varphi(T)=0
\end{equation}
will be called the homogeneous BVP corresponding to BVP (\ref{dy1}), (\ref{dy2}).

The solution $\varphi(t)\equiv 0$ to homogeneous BVP
(\ref{dy10}), (\ref{dy20}) is called the trivial solution.

BVP (\ref{dy1}), (\ref{dy2}) can be written in a scalar form:
\begin{gather}\label{xa}
\begin{array}{ccc}
\varphi'_1(t)+a_{11}\varphi_1(t)+a_{12}\varphi_2(t)
+\cdots+a_{1n}\varphi_n(t)&=&f_1(t),\\
\varphi'_2(t)+a_{21}\varphi_1(t)+a_{22}\varphi_2(t)
+\cdots+a_{2n}\varphi_n(t)&=&f_1(t),\\ \cdot  \cdot\cdot \cdot
\cdot\cdot\cdot  \cdot\cdot\cdot \cdot\cdot \cdot  \cdot\cdot
\cdot  \cdot\cdot\cdot \cdot \cdot\cdot \cdot
\cdot\cdot\cdot\cdot\cdot\cdot
\cdot\cdot\cdot\cdot\cdot\cdot\cdot\cdot&\cdots&\cdots,\\
\varphi'_n(t)+a_{n1}\varphi_1(t)+a_{n2}\varphi_2(t)
+\cdots+a_{nn}\varphi_n(t)&=&f_n(t),
\end{array}\\
\begin{split}
\label{xa1} U_i(\varphi):=\sum_{q=1}^n
b^{(0)}_{iq}\varphi_q(0)=f_i^{(0)}, \quad i=\overline{1,m},\\
 U_{m+i}(\varphi):=\sum_{q=1}^n
b^{(1)}_{iq}\varphi_q(T)=f_i^{(1)}, \quad i=\overline{1,n-m}.
\end{split}
\end{gather}
Let
\begin{equation}\label{0x}
\varphi^{(i)}(t)=(\varphi^{(i)}_1(t),\varphi^{(i)}_2(t)
\ldots,\varphi^{(i)}_n(t))^T,\quad i=\overline{1,n},
\end{equation}
be a fundamental system of solutions to (\ref{dy10})
 (for the definition, see e.g. \cite{BIBLFED} p. 179). Then the solutions to
(\ref{dy10}), (\ref{dy20}) have the form
$$
\varphi(t)=c_1\varphi^{(1)}(t)+c_2\varphi^{(2)}(t)+\cdots+
c_n\varphi^{(n)}(t),
$$
where, by virtue of (\ref{dy20}), constants $c_1,\, c_2,\ldots,c_n$
must be such that
\begin{equation}\label{xa2x}
\begin{array}{ccc}
c_1U_1(\varphi^{(1)})+c_2U_1(\varphi^{(2)})
+\cdots+c_nU_1(\varphi^{(n)})&=&0,\\
c_1U_2(\varphi^{(1)})+c_2U_2(\varphi^{(2)})
+\cdots+c_nU_2(\varphi^{(n)})&=&0,\\ \cdot \cdot\cdot \cdot
\cdot\cdot\cdot  \cdot\cdot\cdot \cdot\cdot \cdot \cdot\cdot \cdot
\cdot\cdot\cdot \cdot \cdot\cdot \cdot
\cdot\cdot\cdot\cdot\cdot\cdot
\cdot\cdot\cdot\cdot\cdot\cdot\cdot\cdot&\cdots&\cdot\cdot,\\
c_1U_n(\varphi^{(1)})+c_2U_n(\varphi^{(2)})
+\cdots+c_nU_n(\varphi^{(n)})&=&0.
\end{array}
\end{equation}
Thus, if the matrix
\begin{equation}\label{xa2}
 \left(
\begin{array}{cccc}
U_1(\varphi^{(1)})&U_1(\varphi^{(2)})& \cdots&U_1(\varphi^{(n)})\\
U_2(\varphi^{(1)})&U_2(\varphi^{(2)})& \cdots&U_2(\varphi^{(n)})\\
\cdots &\cdots
 & \cdot\cdot\cdot&
\cdots
\\U_n(\varphi^{(1)})&U_n(\varphi^{(2)})& \cdots&U_n(\varphi^{(n)})
\end{array}
\right)
\end{equation}
has rank $n,$ the homogeneous BVP has only the trivial solution. The inverse statement is also valid:  if the homogeneous BVP has only the trivial solution then the rank of matrix (\ref{xa2})
equals $n.$ \label{pl} Indeed, following the reasoning that can be found e.g. in \cite{BIBLNaym}, assume that the rank of this matrix is $r<n$; then system (\ref{xa2x}) would have
$n-r$ linearly independent solutions
$c^{(i)}=(c_1^{(i)},\ldots,c_n^{(i)})^T,$ $i=\overline{1,n-r}$
(see e.g. \cite{BIBLKUR}, p. 85). Let us show that if this assumption holds then the functions
\begin{equation}\label{xa2x1}
\tilde
\varphi^{(i)}(x)=c_1^{(i)}\varphi^{(1)}(x)+\cdots+c_n^{(i)}\varphi^{(n)}(x)
\quad i=\overline{1,n-r},
\end{equation}
satisfying conditions (\ref{dy10}), (\ref{dy20}) will be linearly independent; that is, the equality
\begin{equation}\label{xa2x2}
\sum_{i=1}^{n-r}\alpha_i\tilde \varphi^{(i)}(x)=0
\end{equation}
is fulfilled only at $\alpha_i=0,$ $i=\overline{1,n-r}.$
Substituting  (\ref{xa2x1}) into (\ref{xa2x2}), we have
$$
\sum_{i=1}^{n-r}\alpha_i\sum_{k=1}^nc_k^{(i)}\varphi^{(k)}(x)=
\sum_{k=1}^n\varphi^{(k)}(x)\sum_{i=1}^{n-r}\alpha_ic_k^{(i)}
=\sum_{k=1}^n\beta_k\varphi^{(k)}(x) =0,
$$
where $\beta_k=\sum_{i=1}^{n-r}\alpha_ic_k^{(i)}.$ However, vector-functions
$\varphi^{(k)}(x),$ $k=\overline{1,n},$ are linearly independent; therefore, $\beta_k=0,$ $k=\overline{1,n},$ or
$\sum_{i=1}^{n-r}\alpha_ic_k^{(i)}=0,$ $k=\overline{1,n}.$ Then all $\alpha_i=0,$ $i=\overline{1,n-r},$ because vectors
$c^{(i)}=(c_1^{(i)},\ldots,c_n^{(i)})^T,$ $i=\overline{1,n-r},$
are linearly independent. Next, linear independence of functions
(\ref{xa2x1}) satisfying (\ref{dy10}), (\ref{dy20}),
contradicts the assumption that BVP
(\ref{dy10}), (\ref{dy20}) has only the trivial solution which means that the rank of matrix  (\ref{xa2}) is $n.$

Assume in what follows that homogeneous BVP
(\ref{dy10}), (\ref{dy20}) corresponding to BVP
(\ref{dy1}), (\ref{dy2}) has only the trivial solution. Show that under this assumption, initial BVP (\ref{dy1}),
(\ref{dy2}) is uniquely solvable at any right-hand sides
$f(t)=(f_1(t),f_2(t)\ldots f_n(t))^T,$
$f_0=(f_1^{(0)},f_2^{(0)},\ldots,f_m^{(0)})^T\in \mathbb R^m$, and
$f_1=(f_1^{(1)},f_2^{(1)},\ldots,f_{n-m}^{(1)})^T\in \mathbb
R^{n-m}.$

Indeed, let (\ref{0x}) be a fundamental system of solutions to homogeneous system  (\ref{dy10}) and
$\varphi^{(0)}(t)=(\varphi^{(0)}_1(t),\varphi^{(0)}_2(t)
\ldots,\varphi^{(0)}_n(t))^T$ a particular solution to
(\ref{dy1}). Then the general solution to system (\ref{dy1}) or to equivalent system (\ref{xa}) has the form
$$
\varphi(t)=c_1\varphi^{(1)}(t)+c_2\varphi^{(2)}(t)+\cdots+
c_n\varphi^{(n)}(t)+\varphi^{(0)}(t),
$$
where $c_i=\mbox{const}.$ This solution satisfies conditions
(\ref{dy2}) or equivalent conditions (\ref{xa1}) if coefficients $c_i,\,\, i=\overline{1,n},$ satisfy the system of linear algebraic equations
\begin{equation}\label{xa3}
\begin{array}{rcl}
c_1U_1(\varphi^{(1)})+c_2U_1(\varphi^{(2)})
+\cdots+c_nU_1(\varphi^{(n)})&=&f_1^{(0)}-U_1(\varphi^{(0)}),\\
c_1U_2(\varphi^{(1)})+c_2U_2(\varphi^{(2)})
+\cdots+c_nU_2(\varphi^{(n)})&=&f_2^{(0)}-U_2(\varphi^{(0)}),\\
\cdots \cdots\cdots\cdots\cdots \cdots\cdots\cdots\cdots
\cdots\cdots\cdots\cdots&\cdots&\cdots\cdots\cdots\cdots\cdot\cdot,\\
c_1U_m(\varphi^{(1)})+c_2U_m(\varphi^{(2)})
+\cdots+c_nU_m(\varphi^{(n)})&=&f_m^{(0)}-U_m(\varphi^{(0)}),\\
c_1U_{m+1}(\varphi^{(1)})+c_2U_{m+1}(\varphi^{(2)})
+\cdots+c_nU_{m+1}(\varphi^{(n)})&=&f_1^{(1)}-U_{m+1}(\varphi^{(0)}),\\
c_1U_{m+2}(\varphi^{(1)})+c_2U_{m+2}(\varphi^{(2)})
+\cdots+c_nU_{m+2}(\varphi^{(n)})&=&f_2^{(1)}-U_{m+2}(\varphi^{(0)}),\\
\cdots \cdots\cdots\cdots\cdots \cdots\cdots\cdots\cdots
\cdots\cdots\cdots\cdots&\cdots&\cdots\cdots\cdots\cdots\cdot\cdot,\\
c_1U_n(\varphi^{(1)})+c_2U_n(\varphi^{(2)})
+\cdots+c_nU_n(\varphi^{(n)})&=&f_{n-m}^{(1)}-U_n(\varphi^{(0)}).
\end{array}
\end{equation}
The rank of matrix (\ref{xa2}) of this system is $n $ because
homogeneous BVP (\ref{dy10}), (\ref{dy20}) has only the trivial
solution. Therefore, system (\ref{xa3}) and, consequently, BVP
(\ref{dy1}), (\ref{dy2}), have the unique solution. We have proved
the following
\begin{pred}
Inhomogeneous BVP (\ref{dy1}),
(\ref{dy2}) is uniquely solvable if and only if the corresponding homogeneous BVP (\ref{dy10}),
(\ref{dy20}) has only the trivial solution.
\end{pred}

Formulate the notion of a BVP conjugate to
(\ref{dy1}), (\ref{dy2}). To this end, introduce the following designations: $E_k$ is the $k\times k $ unit matrix;  $O_{k,r}$ is the $k\times r $ null matrix; $B_{01}=\{b^{(0)}_{ri_k}\},$
$r=\overline{1,m},$ $k=\overline{1,m},$ is a square nondegenerate $m\times m$ submatrix of the matrix $B_0=
\{b^{(0)}_{rs}\},$ $r=\overline{1,m},$ $s=\overline{1,n}$;
$B_{02}=\{b^{(0)}_{rj_l}\},$ $r=\overline{1,m},$
$l=\overline{1,n-m},$ is an $m\times (n-m)$
submatrix of $B_0$ obtained as a result of deleting in $B_0$ all columns of matrix $B_{01}$ (so that
$\{j_1,\ldots,j_{n-m}\}=\{1,\ldots,n\}\setminus
\{i_1,\ldots,i_m\}$); $\hat
B_{0}=(-B_{02}^T(B_{01}^T)^{-1},E_{n-m})$ is\label{ox}
an $(n-m)\times n$ matrix such that its $i_k$th column equals $k$th
column of matrix $-B_{02}^T(B_{01}^T)^{-1}$ (its size is
$(n-m)\times m$), $k=\overline{1,m},$ and $j_l$th column equals
$l$th column of matrix $E_{n-m},$ $l=\overline{1,n-m};$
$\bar B_{0}=((B_{01}^T)^{-1},O_{m,n-m})$ is an
$m\times n$ matrix such that its $i_k$th column equals $k-$th column of matrix
$(B_{01}^T)^{-1},$ $k=\overline{1,m},$ and $j_l$th column equals
$l$th column of matrix $O_{m,n-m},$ $l=\overline{1,n-m};$
$\tilde B_{0}=(O_{n-m,m},E_{n-m})$ is an
$(n-m)\times n$ matrix such that its $i_k$th column equals $k$th column of matrix $O_{n-m,m},$ $k=\overline{1,m},$ and $j_l$th column equals
$l$th column of matrix $E_{n-m},$ $l=\overline{1,n-m}.$

Introduce more similar notations: $B_{11}=\{b^{(1)}_{ri'_{k}}\},$ $
r=\overline{1,n-m},$ $k=\overline{1,n-m},$ is a square nondegenerate $(n-m)\times (n-m)$  submatrix  of the matrix
$B_1= \{b^{(1)}_{rs}\},$ $r=\overline{1,n-m},$ $s=\overline{1,n};$
$B_{12}=\{b^{(1)}_{rj'_{l}}\},$ $ r=\overline{1,n-m},$
$l=\overline{1,m},$ is a $(n-m)\times m$  submatrix  of the matrix $B_1$ obtained as a result of deleting in $B_1$ all columns of matrix $B_{11}$ (so that
$\{j'_1,\ldots,j'_m\}=\{1,\ldots,n\}\setminus
\{i'_1,\ldots,i'_{n-m}\}$); $\hat
B_{1}=(-B_{12}^T(B_{11}^T)^{-1},E_m)$ is an
$m\times n$ matrix such that its $i'_{k}$th column equals $k$th column of matrix  $-B_{12}^T(B_{11}^T)^{-1}$ (the size of the latter is $m\times
(n-m)$), $k=\overline{1,n-m},$ and $j'_{l}$th column equals $l$th
column of matrix $E_m,$ $l=\overline{1,m};$  $\bar
B_{1}=((B_{11}^T)^{-1},O_{n-m,m})$ is\label{ox1} an $(n-m)\times n$ matrix such that its  $i'_{k}$th column equals
$k$th column of matrix $(B_{11}^T)^{-1},$ $k=\overline{1,n-m},$ and
$j'_{l}$th column equals $l$th column of matrix $O_{n-m,m},$
$l=\overline{1,m};$ $\tilde B_{1}=(O_{m,n-m},E_m)$ is an $m\times n$ matrix such that its $i'_k$th column equals
$k$th column of matrix $O_{m,n-m},$ $k=\overline{1,n-m},$ and
$j'_l$th column equals $l$th column of matrix $E_m,$
$l=\overline{1,m}.$

By
$$
(u,v)_N=\sum_{i=1}^Nu_iv_i
$$
we will denote the inner product of vectors $u=(u_1,\ldots,u_N)^T$ and
$v=(v_1,\ldots,v_N)^T$ in the Euclidean space $\mathbb R^N.$
Set
$$
L=\frac{d}{dt}+A,
$$
then
$$
L\varphi(t)=\frac{d\varphi(t)}{dt}+A\varphi(t).
$$
Calculate the inner product of both sides of the latter equality and the vector-function
$\psi(t)=(\psi_1(t),\ldots,\psi_n(t))^T$ and integrate the result from $0$ to $T$ to obtain
$$
\int_0^T(L\varphi(t),\psi(t))_ndt=
\int_0^T\left(\frac{d\varphi(t)}{dt}+A\varphi(t),
\psi(t)\right)_ndt
$$
$$
=\int_0^T\left(\frac{d\varphi(t)}{dt},
\psi(t)\right)_ndt+\int_0^T\left(A\varphi(t), \psi(t)\right)_ndt
$$
$$
=\int_0^T\sum_{i=1}^n\frac{d\varphi_i(t)}{dt}\psi_i(t)dt+
\int_0^T\sum_{i=1}^n\sum_{j=1}^na_{ij}\varphi_j(t)\psi_i(t)dt
$$
$$
=\sum_{i=1}^n\varphi_i(t)\psi_i(t)\left|_0^T\right.-
\int_0^T\sum_{i=1}^n\frac{d\psi_i(t)}{dt}\varphi_i(t)dt +
\int_0^T\sum_{j=1}^n\left(\sum_{i=1}^na_{ij}\psi_i(t)
\varphi_j(t)\right)dt
$$
$$
\!\!\!\!=\!(\varphi(T),\psi(T))_n\!-(\varphi(0),\psi(0))_n\!-\!
\int_0^T\!\!\left(\frac{d\psi(t)}{dt}, \varphi(t)\right)_n\!dt+\!
\int_0^T\!\!\sum_{i=1}^n\!\!\left(\sum_{j=1}^n\!a_{ij}\psi_j(t)
\varphi_i(t)\right)\!dt
$$
$$
\!\!\!\!=(\varphi(T),\psi(T))_n-(\varphi(0),\psi(0))_n+
\int_0^T\left(-\frac{d\psi(t)}{dt}, \varphi(t)\right)_ndt+
\int_0^T\left(\varphi(t),A^T\psi(t)\right)_ndt
$$
\begin{equation}\label{xa5}
=(\varphi(T),\psi(T))_n-(\varphi(0),\psi(0))_n+
\int_0^T\left(\varphi(t),L^*\psi(t)\right)_n,
\end{equation}
where the differential operator
$$
L^*=-\frac{d}{dt}+A^T
$$
will be called formally conjugate to operator $L.$

Let us show that the integrands in (\ref{xa5})
can be represented as
$$
(\psi(T),\varphi(T))_n-(\psi(0),\varphi(0))_n=(\bar
B_1\psi(T),B_1\varphi(T))_{n-m}+(\hat{B}_1\psi(T),\tilde B_1
\varphi(T))_m
$$
\begin{equation}\label{xa6}
-(\bar B_0\psi(0),B_0\varphi(0))_m-(\hat{B}_0\psi(0),\tilde B_0
\varphi(0))_{n-m}.
\end{equation}
Note first that
$$
B_0\varphi(0)=\left(
\begin{array}{c}
\sum_{q=1}^nb_{1q}^{(0)}\varphi_q(0)\\ \vdots\\
\sum_{q=1}^nb_{nq}^{(0)}\varphi_q(0)
\end{array}\right)=
\left(
\begin{array}{c}
\sum_{k=1}^nb_{1i_k}^{(0)}\varphi_{i_k}(0)
+\sum_{l=1}^{n-m}b_{1j_l}^{(0)}\varphi_{j_l}(0)\\ \vdots\\
\sum_{k=1}^nb_{mi_k}^{(0)}\varphi_{i_k}(0)
+\sum_{l=1}^{n-m}b_{mj_l}^{(0)}\varphi_{j_l}(0)
\end{array}\right)
$$
$$
=B_{01}\varphi^{(0)}_1(0)+B_{02}\varphi^{(0)}_2(0),
$$
where
$$
\varphi^{(0)}_1(0):=\left(
\begin{array}{c}
\varphi_{i_1}(0)\\ \vdots\\ \varphi_{i_m}(0)\end{array}\right),
\quad \varphi^{(0)}_2(0):=\left(
\begin{array}{c}
\varphi_{j_1}(0)\\ \vdots\\
\varphi_{j_{n-m}}(0)\end{array}\right).
$$
Then \label{pt} $\varphi^{(0)}_1(0)=B_{01}^{-1}B_0\varphi(0)
-B_{01}^{-1}B_{02}\varphi^{(0)}_2(0),$ and
$$
(\psi(0),\varphi(0))_n=(\psi^{(0)}_1(0),\varphi^{(0)}_1(0))_m
+(\psi^{(0)}_2(0),\varphi^{(0)}_2(0))_{n-m}
$$
$$
=
(\psi^{(0)}_1(0),B_{01}^{-1}B_0\varphi(0))_m
-(\psi^{(0)}_1(0),B_{01}^{-1}B_{02}
\varphi^{(0)}_2(0))_m+(\psi^{(0)}_2(0),\varphi^{(0)}_2(0))_{n-m}
$$
$$
=(\bar B_0\psi(0),B_0\varphi(0))_m
-(B_{02}^T(B_{01}^T)^{-1}\psi^{(0)}_1(0),
\varphi^{(0)}_2(0))_{n-m}+(\psi^{(0)}_2(0),
\varphi^{(0)}_2(0))_{n-m}
$$
$$
=(\bar B_0\psi(0),B_0\varphi(0))_m+
((-B_{02}^T(B_{01}^T)^{-1},0)\psi(0),\varphi^{(0)}_2(0) )_{n-m}
$$
$$
+\bigl((0,E_{n-m})\psi(0),\varphi^{(0)}_2(0) \bigr)_{n-m},
$$
where $\psi^{(0)}_1(0)$ and $\psi^{(0)}_2(0)$ are vectors composed of components of vector $\psi(0)$ with the numbers equal to the numbers of components of
vectors $\varphi^{(0)}_1(0)$ and $\varphi^{(0)}_2(0)$, respectively.
Taking into account that
$$
(-B_{02}^T(B_{01}^T)^{-1},O_{n-m,n-m})+(O_{n-m,m},E_{n-m})
=(-B_{02}^T(B_{01}^T)^{-1},E_{n-m})=\hat B_0,
$$
we have
$$
(\psi(0),\varphi(0))_n=(\bar
B_0\psi(0),B_0\varphi(0))_m+(\hat{B}_0\psi(0),\tilde B_0
\varphi(0))_{n-m}.
$$
Thus
$$(\psi(T),\varphi(T))_n=(\bar
B_1\psi(T),B_1\varphi(T))_{n-m}+(\hat{B}_1\psi(T),\tilde B_1
\varphi(T))_m.
$$
These two equalities yield representation  (\ref{xa6});
using the latter and (\ref{xa5}), we obtain
$$
\int_0^T(L\varphi(t),\psi(t))_ndt=(\bar
B_1\psi(T),B_1\varphi(T))_{n-m}+(\hat{B}_1\psi(T),\tilde B_1
\varphi(T))_m
$$
\begin{equation}\label{xa7}
-(\bar B_0\psi(0),B_0\varphi(0))_m-(\hat{B}_0\psi(0),\tilde B_0
\varphi(0))_{n-m}+ \int_0^T\left(\varphi(t),L^*\psi(t)\right)_n.
\end{equation}
In order to write the sum of the first four terms on the
right-hand side of (\ref{xa7}) in a scalar form, introduce the
following notations:
\begin{equation}\label{xa8}
\left(
\begin{array}{c}
U_{n+1}(\varphi)\\ \vdots\\ U_{2n-m}(\varphi)
\end{array}\right):=\tilde B_0\varphi(0)=\left(
\begin{array}{c}
\sum_{q=1}^n \tilde b^{(0)}_{1q}\varphi_q(0)\\ \vdots\\
\sum_{q=1}^n \tilde b^{(0)}_{n-m,q}\varphi_q(0)
\end{array}\right),
\end{equation}
\begin{equation}\label{xa9}
\left(
\begin{array}{c}
U_{2n-m+1}(\varphi)\\ \vdots\\ U_{2n}(\varphi)
\end{array}\right):=\tilde B_1\varphi(T)=\left(
\begin{array}{c}
\sum_{q=1}^n \tilde b^{(1)}_{1q}\varphi_q(T)\\ \vdots\\
\sum_{q=1}^n \tilde b^{(1)}_{m,q}\varphi_q(T)
\end{array}\right),
\end{equation}
\begin{equation}\label{xa10}
\left(
\begin{array}{c}
V_{2n}(\psi)\\ \vdots\\ V_{2n-m+1}(\psi)
\end{array}\right):=\bar B_0\psi(0)=\left(
\begin{array}{c}
\sum_{q=1}^n \bar b^{(0)}_{1q}\psi_q(0)\\ \vdots\\ \sum_{q=1}^n
\bar b^{(0)}_{m,q}\psi_q(0)
\end{array}\right),
\end{equation}
\begin{equation}\label{xa11}
\left(
\begin{array}{c}
V_{2n-m}(\psi)\\ \vdots\\ V_{n+1}(\psi)
\end{array}\right):=\bar B_1\psi(T)=\left(
\begin{array}{c}
\sum_{q=1}^n \bar b^{(1)}_{1q}\psi_q(T)\\ \vdots\\ \sum_{q=1}^n
\bar b^{(1)}_{n-m,q}\psi_q(T)
\end{array}\right),
\end{equation}
\begin{equation}\label{xa12}
\left(
\begin{array}{c}
V_{n}(\psi)\\ \vdots\\ V_{m+1}(\psi)
\end{array}\right):=\hat B_0\psi(0)=\left(
\begin{array}{c}
\sum_{q=1}^n \hat b^{(0)}_{1q}\psi_q(0)\\ \vdots\\ \sum_{q=1}^n
\hat b^{(0)}_{n-m,q}\psi_q(0)
\end{array}\right),
\end{equation}
\begin{equation}\label{xa13}
\left(
\begin{array}{c}
V_{m}(\psi)\\ \vdots\\ V_{1}(\psi)
\end{array}\right):=\hat B_1\psi(T)=\left(
\begin{array}{c}
\sum_{q=1}^n \hat b^{(1)}_{1q}\psi_q(T)\\ \vdots\\ \sum_{q=1}^n
\hat b^{(1)}_{m,q}\psi_q(T)
\end{array}\right).
\end{equation}
The equality (\ref{xa7}) can be written as
$$
\int_0^T(L\varphi(t),\psi(t))_ndt-
\int_0^T\left(\varphi(t),L^*\psi(t)\right)_n
$$
$$
=-U_1(\varphi)V_{2n}(\psi) -U_2(\varphi)V_{2n-1}(\psi)-\cdots-
U_m(\varphi)V_{2n-m+1}(\psi)
$$
$$
+U_{m+1}(\varphi)V_{2n-m}(\psi)
+U_{m+2}(\varphi)V_{2n-m-1}(\psi)+\cdots+
U_n(\varphi)V_{n+1}(\psi)
$$
$$
-U_{n+1}(\varphi)V_{n}(\psi)
-U_{n+2}(\varphi)V_{n-1}(\psi)-\cdots-
U_{2n-m}(\varphi)V_{m+1}(\psi)
$$
\begin{equation}\label{xa77}
 + U_{2n-m+1}(\varphi)V_{m}(\psi)
+U_{2n-m+2}(\varphi)V_{m-1}(\psi)+\cdots+
U_{2n}(\varphi)V_{1}(\psi).
\end{equation}

Now we can introduce the notion of the conjugate BVP.
\begin{predlll}
The homogeneous BVP
\begin{gather}\label{xax3}
L^*\psi(t)=0, \quad t \in (0,T),\\ \label{xax4} \hat B_0
\psi(0)=0, \quad \hat B_1 \psi(T)=0,
\end{gather}
is called conjugate to homogeneous BVP (\ref{dy10}),
(\ref{dy20}).
\end{predlll}
\begin{predlll}
The inhomogeneous BVP
\begin{gather}\label{xax5}
L^*\psi(t)=\tilde f(t), \quad t\in (0,T),\\ \label{xax6} \hat
B_0\psi(0)=\tilde f_0, \quad \hat B_1\psi(T)=\tilde f_1,
\end{gather}
is called conjugate to inhomogeneous BVP (\ref{dy1}),
(\ref{dy2}).
\end{predlll}
Using designations (\ref{xa1}) and (\ref{xa8})--(\ref{xa13}),
we can write BVP (\ref{xax3}), (\ref{xax4}) conjugate to BVP
(\ref{dy10}), (\ref{dy20}) in the scalar form:
\begin{gather}\label{xax7}
\begin{array}{ccc}
-\psi'_1(t)+a_{11}\psi_1(t)+a_{21}\psi_2(t)
+\cdots+a_{n1}\psi_n(t)&=&0,\\
-\psi'_2(t)+a_{12}\psi_1(t)+a_{22}\psi_2(t)
+\cdots+a_{n2}\psi_n(t)&=&0,\\ \cdot  \cdot\cdot \cdot
\cdot\cdot\cdot \cdot\cdot\cdot \cdot\cdot \cdot  \cdot\cdot \cdot
\cdot\cdot\cdot \cdot \cdot\cdot \cdot
\cdot\cdot\cdot\cdot\cdot\cdot
\cdot\cdot\cdot\cdot\cdot\cdot\cdot\cdot&\cdots&\cdots,\\
-\psi'_n(t)+a_{1n}\psi_1(t)+a_{2n}\psi_2(t)
+\cdots+a_{nn}\psi_n(t)&=&0,
\end{array}\\
\label{xax8} V_i(\psi):=0, \quad i=\overline{1,n}.
\end{gather}
Let $z^{(1)}(t),$ $z^{(2)}(t),$ $\ldots,$
$z^{(n)}(t)$ is a fundamental system of solutions to the homogeneous system
$L^*\psi(t)=0$. Show that the rank of matrix
\begin{equation}\label{xa14}
 \left(
\begin{array}{cccc}
V_1(z^{(1)})&V_1(z^{(2)})& \cdots&V_1(z^{(n)})\\
V_2(z^{(1)})&V_2(z^{(2)})& \cdots&V_2(z^{(n)})\\ \cdots &\cdots
 & \cdot\cdot\cdot&
\cdots
\\V_n(z^{(1)})&V_n(z^{(2)})& \cdots&V_n(z^{(n)})
\end{array}
\right)
\end{equation}
equals $n.$
Assume that it is wrong and the rank of matrix (\ref{xa14})
is $r<n.$ Every solution of the equation $L^*\psi(t)=0$
and, in particular, of homogeneous  BVP (\ref{xax7}),
(\ref{xax8}) has the form
$$
\psi(t)=c_1z^{(1)}(t)+\cdots+c_nz^{(n)}(t),
$$
where $c_k=\mbox{const},$ $k=\overline{1,n}.$ Substituting the latter into (\ref{xax8}), we obtain a homogeneous linear equation system
\begin{equation}\label{xxa1}
\begin{array}{ccc}
V_1(\psi)=c_1V_1(z^{(1)})+\cdots+c_nV_1(z^{(n)})&=&0,\\
\cdots\cdots\cdots\cdots\cdots\cdots\cdots\cdots\cdots\cdots\cdots
&\cdots&\cdot,
 \\V_n(\psi)=c_1V_n(z^{(1)})+\cdots+c_nV_n(z^{(n)})&=&0
\end{array}
\end{equation}
with respect to constants $c_k,$ $k=\overline{1,n}.$
Since the rank of the system matrix equals $r$ and
$r<n,$ system (\ref{xxa1}) has $n-r$ linearly independent solutions $c^{(i)}=(c_1^{(i)},\ldots,c_n^{(i)})^T,$
$i=\overline{1,n-r}$;
therefore, the functions
$$
\psi^{(i)}(t)=c_1^{(i)}z^{(1)}(t)+\cdots+c_nz^{(n)}(t),
$$
which solve  conjugate homogeneous BVP (\ref{xax7}), (\ref{xax8})
will be linearly independent (in line with the reasoning on p.
\pageref{pl}).

If now $\psi(t)$ is a solution to homogeneous BVP (\ref{xax7}),
(\ref{xax8}), then, if we set $\varphi(t)=\varphi^{(i)}(t),$ where
$\varphi^{(i)}(t),$ $i=\overline{1,n}$ is the fundamental system
of solutions to the homogeneous equation $L\varphi(t)=0,$ the
integrals in (\ref{xa77}) vanish. Also,
$V_1(\psi)=V_2(\psi)=\ldots=V_n(\psi)=0,$ and therefore,
(\ref{xa77}) takes the form
\begin{equation}\label{xa15}
\begin{array}{rcl}
\!U_1(\varphi^{(1)})(-V_{2n}(\psi))
\!+\!\cdots\!+\!U_m(\varphi^{(1)})(-V_{2n-m+1}(\psi))
\!+\!U_{m+1}(\varphi^{(1)})V_{2n-m}(\psi)\!+\!\cdots
\!+\!U_n(\varphi^{(1)})V_{n+1}(\psi)\! &=&\!0,\\ \!\cdots
\cdots\cdots\cdots\cdots\cdots \cdots\cdots\cdots\cdots
\cdots\cdots\cdots\cdots\cdots\cdots\cdots\cdot\cdot
\cdots\cdots\cdots\cdots\cdots\cdots\cdots\cdot\cdot
\cdots\cdots\cdots\cdots &\cdots& \cdot,\\
\!U_1(\varphi^{(n)})(-V_{2n}(\psi))
\!+\!\cdots\!+\!U_m(\varphi^{(n)})(-V_{2n-m+1}(\psi))
\!+\!U_{m+1}(\varphi^{(n)})V_{2n-m}(\psi)\!+\!\cdots
\!+\!U_n(\varphi^{(n)})V_{n+1}(\psi)\! &=&\!0.
\end{array}
\end{equation}
This system has $n-r$ linearly independent solutions
\begin{equation}\label{xxa15}
-V_{2n}(\psi^{(i)}),\,\cdots,\,-V_{2n-m+1}(\psi^{(i)}),\,
V_{2n-m}(\psi^{(i)}),\,\cdots,\,V_{n+1}(\psi^{(i)}) \quad
(i=1,2,\ldots,n-r),
\end{equation}
Indeed, if we assume their linear dependence, then the rank of
matrix
\begin{equation}\label{xa16}
 \left(
\begin{array}{cccccc}
V_{n+1}(\psi^{(1)})&\cdots&V_{2n-m}(\psi^{(1)})&
-V_{2n-m+1}(\psi^{(1)})&\cdots&-V_{2n}(\psi^{(1)})\\
\cdots\cdots\cdots &\cdots&\cdots\cdots\cdots&
 \cdots\cdots\cdots&\cdots&\cdots\cdots\cdots\\
 V_{n+1}(\psi^{(n-r)})&\cdots&V_{2n-m}(\psi^{((n-r))})&
-V_{2n-m+1}(\psi^{((n-r))})&\cdots&-V_{2n}(\psi^{((n-r))})
\end{array}
\right)
\end{equation}
will be less than $n-r.$ However, since the rank of (\ref{xa16})
equals the maximum number of its linearly independent  rows,
there exist numbers $a_1,$ $\ldots,$ $a_{n-r},$ such that at least
one of them is nonzero and
$$
a_1V_i(\psi^{(1)})+a_2V_i(\psi^{(2)})+\cdots+
a_{n-r}V_i(\psi^{(n-r)})=0, \quad i=\overline{n+1,2n}
$$
or
$$
V_i\left(a_1\psi^{(1)}+a_2\psi^{(2)}+\cdots+
a_{n-r}\psi^{(n-r)}\right)=0, \quad i=\overline{n+1,2n}.
$$
Thus, setting
\begin{equation}\label{xaxx}
\psi(t)=a_1\psi^{(1)}(t)+a_2\psi^{(2)}(t)+\cdots+
a_{n-r}\psi^{(n-r)}(t),
\end{equation}
we have
$$
V_i(\psi)=0\quad (i=1,2,\ldots,n,n+1,\ldots,2n);
$$
in a more detailed form
\begin{gather}
\hat b_{m1}^{(1)}\psi_1(T)+\cdots+\hat
b_{mn}^{(1)}\psi_n(T)=0,\notag\\ \cdots \cdots\cdots \cdots \cdots
\cdots \cdots \cdots\cdots \cdots,\notag\\ \hat
b_{11}^{(1)}\psi_1(T)+\cdots+\hat b_{1n}^{(1)}\psi_n(T)=0,\notag\\
\hat b_{n-m,1}^{(0)}\psi_1(0)+\cdots+\hat
b_{n-m,n}^{(0)}\psi_n(0)=0,\notag\\ \cdots \cdots\cdots \cdots
\cdots \cdots \cdots \cdots\cdots \cdots,\notag\\ \hat
b_{11}^{(0)}\psi_1(0)+\cdots+\hat b_{1n}^{(0)}\psi_n(0)=0,\notag\\
\bar b_{n-m,1}^{(1)}\psi_1(T)+\cdots+\bar
b_{n-m,n}^{(1)}\psi_n(T)=0,\notag\\ \cdots \cdots\cdots \cdots
\cdots \cdots \cdots \cdots\cdots \cdots,\notag\\ \bar
b_{11}^{(1)}\psi_1(T)+\cdots+\bar b_{1n}^{(1)}\psi_n(T)=0,\notag\\
\bar b_{m,1}^{(0)}\psi_1(0)+\cdots+\bar
b_{m,n}^{(0)}\psi_n(0)=0,\notag\\ \cdots \cdots\cdots \cdots
\cdots \cdots \cdots \cdots\cdots \cdots,\notag\\ \label{xa17}
\bar b_{11}^{(0)}\psi_1(0)+\cdots+\bar b_{1n}^{(0)}\psi_n(0)=0.
\end{gather}
Show that the determinant of the $2n\times 2n$ matrix
\begin{equation}\label{xa18}
 \left(
\begin{array}{cccccc}
0&\cdots&0& \hat b_{m1}^{(1)}&\cdots&\hat b_{mn}^{(1)}\\ \cdots
&\cdots&\cdots& \cdots&\cdots&\cdots\\ 0&\cdots&0& \hat
b_{11}^{(1)}&\cdots&\hat b_{1n}^{(1)}\\
 \hat
b_{n-m,1}^{(0)}&\cdots&\hat b_{n-m,n}^{(0)}& 0&\cdots&0 \\ \cdots
&\cdots&\cdots& \cdots&\cdots&\cdots\\ \hat
b_{11}^{(0)}&\cdots&\hat b_{1n}^{(0)}&0&\cdots&0 \\0&\cdots&0&
\bar b_{n-m,1}^{(1)}&\cdots&\bar b_{n-m,n}^{(1)}\\ \cdots
&\cdots&\cdots& \cdots&\cdots&\cdots\\ 0&\cdots&0& \bar
b_{11}^{(1)}&\cdots&\bar b_{1n}^{(1)}\\
 \bar
b_{m,1}^{(0)}&\cdots&\bar b_{m,n}^{(0)}& 0&\cdots&0 \\ \cdots
&\cdots&\cdots& \cdots&\cdots&\cdots\\ \bar
b_{11}^{(0)}&\cdots&\bar b_{1n}^{(0)}&0&\cdots&0
\end{array}
\right):=G
\end{equation}
of system (\ref{xa17}) with respect to
$\psi_1(0),\ldots,\psi_n(0),\psi_1(T),\ldots,\psi_n(T)$ is not
equal to  zero. By virtue of the Laplace theorem  (see, e.g.
\cite{BIBLKUR}, p. 51),
the sum of all $n$th order minors in the last $n$ rows of matrix
$G$ multiplied by their algebraic complements equals the matrix
determinant. However, in the rows of matrix $G$ that have numbers
$n+1,$ $\ldots,$ $2n$ there is only one $n$th order minor: its
elements are in the columns with numbers $i_{1}$, $\ldots,$
$i_{m}$ and rows with numbers $2n-m+1,$ $\ldots$ $2n$ that form
matrix  $(B_{01}^T)^{-1},$ and also in the columns $n+i'_{1},$
$\ldots,$ $n+i'_{m}$ and rows $n+1,$ $\ldots,$ $2n-m$ that form
matrix $(B_{11}^T)^{-1}.$ Its complement equals the determinant
situated in in the rows  $1,$ $\ldots,$ $m$ and columns
$n+j'_{1},$ $\ldots,$ $n+j'_{l}$ forming a matrix $E_m,$ and also
in the rows $m+1,$ $\ldots,$ $n$ and columns $j_{1},$ $\ldots,$
$j_{l}$ forming a matrix  $E_{n-m}.$ Calculating these minors with
the help of the Laplace theorem, we obtain
$$
\mbox{det}\,G=(-1)^{i_1+\cdots+i_m+2n-m+1+\cdots+2n
+n+i'_1+\cdots+n+i'_{n-m}+(n+1)+\cdots+2n-m} \times
$$
$$
\times
(-1)^{i_1+\cdots+i_m+2n-m+1+\cdots+2n}\mbox{det}\,(B_{01}^T)^{-1}
\mbox{det}\,(
B_{11}^T)^{-1}(-1)^{n+j'_1+\cdots+n+j'_m+1+\cdots+m}=
$$
$$
=(-1)^{i'_1+\cdots+i'_{n-m}+j'_1+\cdots+j'_m}(-1)^{n^2}
(-1)^{1+\cdots+m+(n+1)+\cdots+2n-m}\mbox{det}\,(B_{01}^T)^{-1}
\mbox{det}\,( B_{11}^T)^{-1}=
$$
$$
=(-1)^{m^2}\mbox{det}\,(B_{01}^T)^{-1} \mbox{det}\,(
B_{11}^T)^{-1}=\frac{(-1)^m}{\mbox{det}\,B_{01} \mbox{det}\,
B_{11}}\neq 0.
$$
Thus, linear equation system (\ref{xa17}) has only the trivial
solution
$\psi_1(0)=\cdots=\psi_n(0)=\psi_1(T)=\cdots=\psi_n(T)=0$;
therefore,  $\psi(t)\equiv 0$ which contradicts the linear
independence of functions $\psi^{(1)}(t),$ $\ldots,$
$\psi^{(n-r)}(t)$ (see equality (\ref{xaxx})). Finally, linear
equation system (\ref{xa15}) has $n-r$ linearly independent
solutions (\ref{xxa15}) so that the rank of matrix (\ref{xa2}) of
system (\ref{xa15}) does not exceed $r$ which is impossible
because this rank equals $n$ according to the assumption.

We see that our initial assumption that the rank of matrix
(\ref{xa14}) is less than  $n$ leads to contradiction.
Consequently, the rank of this matrix equals $n$ and homogeneous
BVP (\ref{dy10}), (\ref{dy20}) has only the trivial solution.

The reasoning above shows that the following statement is valid.
\begin{pred}
If homogeneous BVP (\ref{dy10}), (\ref{dy20}) has only the trivial
solution, then the corresponding conjugate BVP (\ref{xax3}),
(\ref{xax4}) also has only the trivial solution.
\end{pred}
\begin{pred}
Under the conditions of Theorem 2 inhomogeneous BVP  (\ref{xax5}),
(\ref{xax6}) has one and only one solution.
\end{pred}
\begin{proof}
According to Theorem 2,  BVP (\ref{xax3}), (\ref{xax4}) conjugate to (\ref{dy10}), (\ref{dy20}) has only the trivial
solution. Literally repeating the proof of Theorem  1,
we obtain the required result.
\end{proof}
\section{Statement of the minimax estimation
problem and its reduction to an optimal control problem}
Let a vector-function
\begin{equation}
\label{d1} y(t)=H(t)\varphi(t)+\xi(t),
\end{equation}
with the values form the space $R^l$ be observed on an interval $(\alpha, \beta)\subseteq (0,T)$;
here $H(t)$ is an $l\times n$ matrix with the entries that are continuous functions on $[\alpha, \beta ],$
$\xi(t)$ is a random vector process with
zero expectation $M\xi(t)$ and unknown $l\times l$ correlation matrix
$R(t,s)=M\xi(t)\xi^{T}(s)$. Let a vector-function $\varphi(t)$ be a solution to BVP (\ref{dy1}), (\ref{dy2}).

Denote by $V$ the set of random vector processes $\tilde \xi(t)$ with zero expectation $M\tilde \xi(t)$ and second moments
$M\tilde \xi(t)^2$ integrable on $(\alpha,\beta)$
such that their correlation matrix $\tilde R(t,s)$ belong to the space
\begin{equation}
\label{d6} \left\{\tilde  R:\int_{\alpha}^{\beta} Sp\,[Q(t)\tilde R(t,t)]dt\leq
1\right\}.
\end{equation}
Set
\begin{multline}\label{d7}
\!\!G=\left\{\tilde F:=(\tilde f_0,\tilde f_1,\tilde f(\cdot)):(Q_0(\tilde f_0-f_0^{(0)}),\tilde f_0-f_0^{(0)})_m+(Q_1(\tilde f_1-f_1^{(0)}),\tilde f_1-f_1^{(0)})_{n-m}\right.\\ \left.+
\int_{0}^{T}\!\!(Q_2(t)(\tilde f(t)-f^{(0)}),\tilde f(t)-f^{(0)})_ndt\leq1\right\},
\end{multline}
where $f_0^{(0)}\in\mathbb R^m,$ $f_1^{(0)}\in \mathbb
R^{n-m}$ are given vectors; $f^{(0)}(t)=(f_1^{(0)}(t),f_2^{(0)}(t)\ldots f_n^{(0)}(t))^T$
is a given vector-function with the components belonging to the space $L^2(0,T)$; $Q(t),$ $Q_0,$ $Q_1$, and $Q_2(t)$ are positive definite matrices of dimensions $l\times l,$ $m\times m,$ $(n-m)\times (n-m)$, and $n\times n$, respectively, the entries of $Q(t),\,
Q^{-1}(t)$, and $Q_2(t),\, Q_2^{-1}(t)$ are continuous on
$[\alpha,\beta]$ and $[0,T]$; and $\mbox{Sp}\,B = \sum_{i=1}^lb_{ii}$ denotes the trace of the matrix $B=\{b_{ij}\}_{i,j=1}^l$.

Assume that the right-hand sides $f(\cdot),$ $f_0,$ and $f_1$ of equation (\ref{dy1}) and boundary conditions (\ref{dy2}) are not known exactly and it is known only that the element  $F:=(f_0, f_1, f(\cdot))$ belongs to a set $G$ and, additionally, $\xi(t)\in V.$

We will look for an estimation of the inner product
\begin{equation}\label{skal}
(a,\varphi(s))_n,
\end{equation}
in the class of estimates  linear with respect to observations that have the form
\begin{equation}
\label{d3}
\widehat{(a,\varphi(s))_n}=\int_{\alpha}^s(u_1(t),y(t))_ldt+
\int_s^{\beta}(u_2(t),y(t))_ldt+c,
\end{equation}
where $s\in(\alpha,\beta)$ and $a$ are vectors belonging to  $\mathbb R^n,$ $u_i(t),$ $i=1,2$ are
vector-functions belonging, respectively, to $L^2(\alpha,s)$ and $L^2(s,\beta)$, and $c$ is certain constant. Set $u:=(u_1,u_2)\in H:=L^2(\alpha,s)\times L^2(s,\beta)=L^2(\alpha,\beta).$

An estimate
$$
\widehat{\widehat{(a,\varphi(s))_n}}=
\int_{\alpha}^s(\hat{u}_1(t),y(t))_ldt+
\int_s^{\beta}(\hat{u}_2(t),y(t))_ldt+\hat{c}
$$
for which vector-function $\hat u(t)=(\hat{u}_1(t),\hat{u}_2(t))$ and constant
$\hat{c}$
are determined from the condition
$$
\sigma(u,c):=\sup_{\tilde F\in
G, \tilde\xi \in
V}
M|(a,\tilde\varphi(s))_n-\widehat{(a,\tilde\varphi(s))_n}
|^2\to\inf_{u\in H, c \in \mathbb
R}:=\sigma^2,
$$
where $\tilde\varphi$ is a solution to BVP (\ref{dy1}),
(\ref{dy2}) at $f(t)=\tilde f(t),$ $f_0=\tilde f_0,$ $f_1=\tilde
f_1,$ and
$$
\widehat{(a,\tilde
\varphi(s))_n}=\int_{\alpha}^s(u_1(t),\tilde y(t))_ldt+
\int_s^{\beta}(u_2(t),\tilde y(t))_ldt+c,\quad \tilde y(t)=H(t)\tilde
\varphi(t)+\tilde \xi(t),
$$
will be called a minimax estimate of inner product $(a,\varphi(s))_n.$
The quantity
$$
\sigma=\{\sup_{\tilde F \in G,\,\tilde \xi \in V}M[(a,\tilde \varphi(s))_n
-\widehat{\widehat{(a,\tilde \varphi(s))_n}}]^2 \}^{1/2}
$$
will be called the minimax estimation error.

We see that the minimax mean square estimate of inner product
$(a,\varphi(s))_n$ is an estimate at which the maximum
mean square estimation error calculated for the worst realization of perturbations attains its minimum.

In this section, we will show that solution to the minimax estimation problem is reduced to the solution of a certain optimal control problem.

For every fixed $u:=(u_1,u_2)\in H$ introduce
vector-functions $z_1(\cdot;u)\in H^1(0,\alpha)^n,$ $
z_2(\cdot;u)\in H^1(\alpha,s)^n,$ $z_3(\cdot;u)\in
H^1(s,\beta)^n,$ and $z_4(\cdot;u)\in H^1(\beta,T)^n$ as solutions
to the following BVP:
$$
L^{\ast}z_1(t;u)=0,\quad 0<t<\alpha, \quad\hat{B}_0z_1(0;u)=0,
$$
$$
L^{\ast}z_2(t;u)=-H^{T}(t)u_1(t),\quad \alpha<t<s,\quad
z_2(\alpha;u)=z_1(\alpha;u),
$$
$$
L^{\ast}z_3(t;u)=-H^{T}(t)u_2(t),\quad s<t<\beta,\quad
z_3(s;u)=z_2(s;u)-a,
$$
\begin{equation}\label{d5} L^{\ast}z_4(t;u)=0,\quad \beta<t<T,\quad
z_4(\beta;u)=z_3(\beta;u), \quad \hat{B}_1z_4(T;u)=0.
\end{equation}
\begin{predl}\label{lem1}
Determination of the minimax estimate of inner product
$(a,\varphi(s))_n$ is equivalent to the problem of optimal control of the system described by BVP (\ref{d5}) with the cost function
$$
I(u) =(Q_0^{-1}\bar{B}_0 z_1(0;u),\bar{B}_0 z_1(0;u))_m +
(Q_1^{-1}\bar{B}_1z_4(T;u),\bar{B}_1z_4(T;u))_{n-m}
$$
$$
+ \int_0^{\alpha} (Q_2^{-1}(t)z_1(t;u), z_1(t;u))_n dt +
\int_{\alpha}^s (Q_2^{-1}(t) z_2(t;u), z_2(t;u))_n dt
$$
$$
+ \int_{s}^{\beta} (Q_2^{-1}(t) z_3(t;u), z_3(t;u))_n dt+
\int_{\beta}^T (Q_2^{-1}(t) z_4(t;u), z_4(t;u))_n dt
$$
\begin{equation}\label{N4}
+\int^s_{\alpha}\left(Q^{-1}(t)u_1(t),u_1(t)\right)_ldt+
\int_s^{\beta}\left(Q^{-1}(t)u_2(t),u_2(t)\right)_ldt.
\end{equation}
\end{predl}
\begin{proof}
Show first that BVP (\ref{d5}) is uniquely solvable under the condition that functions $u_1(t)$ and $u_2(t)$ belong, respectively, to the spaces $L^2(\alpha,s)$ and $L^2(s,\beta).$

Since homogeneous BVP (\ref{dy10}), (\ref{dy20})
has only the trivial solution, the BVP
\begin{equation}\label{d31}
L^*\psi(t)=g(t),\quad 0<t<T, \quad\hat{B}_0\psi(0)=0,
\quad\hat{B}_1\psi(T)=0
\end{equation}
has, in line with Theorem 3, the unique solution for any right-hand side, in particular, at
\begin{equation}\label{du31}
g(t)=g(t;u)=\left\{
\begin{array}{lc}
0,&0<t<\alpha;\\
-H^T(t)u_1(t),&\alpha<t<s;\\-H^T(t)u_2(t),&s<t<\beta;\\
0,&\beta<t<T.
\end{array}
\right.
\end{equation}

Denote this solution by $\bar z(t;u)$ and its reductions on intervals $(0,\alpha),$ $(\alpha,s),$ $(s,\beta)$, and $(\beta,T)$
by $\bar z_1(t;u),$ $\bar z_2(t;u),$ $\bar
z_3(t;u),$ and $\bar z_4(t;u)$, respectively. Note that function $\bar z(t;u)$ is absolutely continuous on  $[0,T]$ (see \cite{BIBLAtk}).

Let us show that the problem
$$
L^{\ast}\Bar{\Bar z}^{(1)}(t)=0,\quad 0<t<\alpha,
\quad\hat{B}_0\Bar{\Bar z}^{(1)}(0)=0,
$$
$$
L^{\ast}\Bar{\Bar z}^{(2)}(t)=0,\quad \alpha<t<s,\quad \Bar{\Bar
z}^{(2)}(\alpha;u)=\Bar{\Bar z}^{(1)}(\alpha;u),
$$
$$
L^{\ast}\Bar{\Bar z}^{(3)}(t)=0,\quad s<t<\beta,\quad \Bar{\Bar
z}^{(3)}(s;u)=\Bar{\Bar z}^{(2)}(s)-a,
$$
\begin{equation}\label{d33} L^{\ast}\Bar{\Bar z}^{(4)}(t)=0,
\quad \beta<t<T,\quad \Bar{\Bar z}^{(4)}(\beta)=\Bar{\Bar
z}^{(3)}(\beta), \quad \hat{B}_1\Bar{\Bar z}^{(4)}(T)=0
\end{equation}
has one and only one solution at any vector $a \in \mathbb R^n.$

Denote by $\Bar{\Bar z}_i^{(j)}(t),$ $i=\overline{1,n},$
$j=\overline{1,4},$ the coordinates of vector-function $\Bar{\Bar z}
^{(j)}(t),$ $j=\overline{1,4}.$ Let $y_{ik}(t),$
$i,k=\overline{1,n}$ is the fundamental system of solutions
of the equation system $L^*z(t)=0$ on  $[0,T].$ The we can represent functions
$\Bar{\Bar z}_i^{(j)}(t),$ $i=\overline{1,n},$ $j=\overline{1,4},$
in the form
$$
\Bar{\Bar z}_i^{(j)}(t)=\sum_{k=1}^n c_k^{(j)}y_{ik}(t),
$$
where $c_k^{(j)}$ are constants.  Taking into account the boundary conditions at the points $t=0,\,T$ and conjugation conditions at
$t=\alpha,\,s,\,\beta$ in (\ref{d33}), we see that the solution to
BVP (\ref{d33}) is equivalent to the solution of the following linear equation system with $4n$ unknowns $c_k^{(j)},$
$k=\overline{1,n},$ $j=\overline{1,4}:$
\begin{eqnarray}
\sum_{k=1}^n a_{ik}^0 c_k^{(1)}&=& 0,\quad i
=\overline{1,n-m},\label{d35}\\ \sum_{k=1}^n y_{ik}(\alpha)
(c_k^{(1)}-c_k^{(2)})&=& 0,\quad i =\overline{1,n},\label{d36}\\
\sum_{k=1}^n y_{ik}(s) (c_k^{(2)}-c_k^{(3)})&=& a_i,\quad i
=\overline{1,n},\label{d37}\\ \sum_{k=1}^n y_{ik}(\beta)
(c_k^{(3)}-c_k^{(4)})&=& 0,\quad i =\overline{1,n},\label{d38}\\
\sum_{k=1}^n a_{ik}^1 c_k^{(4)}&=& 0,\quad i
=\overline{1,m},\label{d39}
\end{eqnarray}
where
$$
a_{ik}^0=\sum_{r=1}^n \hat b_{ir}^{(0)} y_{rk}(0),\quad
i=\overline{1,n-m},\quad k =\overline{1,n},
$$
$$
a_{ik}^1=\sum_{s=1}^n \hat b_{is}^{(1)} y_{sk}(T),\quad i
=\overline{1,m},\quad k =\overline{1,n},
$$
$a_i,$ $i=\overline{1,n},$ denote the coordinates of vector
$a,$ and  $\hat b_{ir}^{(0)},$ $i=\overline{1,n-m},$
$r=\overline{1,n},$ and $\hat b_{is}^{(1)},$ $i=\overline{1,m},$
$s=\overline{1,n},$ denote the entries of matrices $\hat B_0$ and $\hat B_1$, respectively.

Show that system (\ref{d35})--(\ref{d39}) is uniquely solvable at any vector $a \in \mathbb R^n.$ To this end, note that homogeneous system (\ref{d35})--(\ref{d39}) (at $a=0$) has only the trivial solution.

Indeed, setting $a=0$ in equations (\ref{d36}) and (\ref{d37}), taking into account (\ref{d38}) and the fact that
$\det\{y_{ik}(\alpha)\}_{i,k=1}^n\neq 0,$
$\det\{y_{ik}(s)\}_{i,k=1}^n\neq 0$, and
$\det\{y_{ik}(\beta)\}_{i,k=1}^n\neq 0$ because
$y_{ik}(t),$ $i,k=\overline{1,n}$ is the fundamental system of
solutions of the equation system $L^*z(t)=0$ on $[0,T],$
we obtain
$$
c_k^{(1)}=c_k^{(2)}=c_k^{(3)}=c_k^{(4)}=:c_k.
$$
Coefficients $c_k$ satisfy equations
(\ref{d35}) and (\ref{d39}); therefore vector-function $\psi(t)$ with the components $\psi_i(t)=\sum_{k=1}^n c_ky_{ik}(t),$
$i=\overline{1,n}$ is a solution to conjugate BVP
(\ref{xax3}), (\ref{xax4}) which has only the trivial
solution $\psi(t)\equiv 0$ on $[0,T].$ This implies $c_k=0$, so the homogeneous linear equation system (\ref{d35})--(\ref{d39}) (at $a=0$) has only the trivial solution. Consequently, system
(\ref{d35})--(\ref{d39}) and therefore BVP
(\ref{d33}) which is equivalent to this system are uniquely solvable at any vector $a \in
\mathbb R^n.$ Then vector-functions $z_i(t;u)=\bar z_i(t;u)+
\Bar{\Bar z}^{(i)}(t),$ $i=\overline{1,4},$ form the unique solution to BVP (\ref{d5}).

Show next that the determination of the minimax estimate of inner product
$(a,\varphi(s))_n$ is equivalent to the problem of optimal control of the system described by BVP (\ref{d5}) with the cost function (\ref{N4}).

Using the second and third equations in (\ref{d5}) and the fact that
$\tilde\varphi$ is a solution to BVP (\ref{dy1}),
(\ref{dy2}) at $f(t)=\tilde f(t),$ $f_0=\tilde f_0,$ and $f_1=\tilde
f_1,$ we easily obtain the relationships
$$
-\int^s_{\alpha}(H^T(t)u_1(t),\tilde \varphi(t))_ndt
=\left(z_2(\alpha;u),
\tilde \varphi(\alpha)\right)_n-\left(z_2(s;u),\tilde \varphi(s)\right)_n
+\int^s_{\alpha}(z_2(t;u),\tilde f(t))_ndt,
$$
$$
-\int_s^{\beta}(H^T(t)u_2(t),\tilde \varphi(t))_ndt=\left(z_3(s;u),
\tilde \varphi(s)\right)_n-\left(z_3(\beta;u),\tilde \varphi(\beta)\right)_n
+\int_s^{\beta}(z_3(t;u),\tilde f(t))_ndt.
$$
Taking into account the equalities
$$
z_2(\alpha;u)=z_1(\alpha;u),\quad z_2(s;u)-z_3(s;u)=a, \quad
z_3(\beta;u)=z_4(\beta;u),
$$
$$
\left(z_1(\alpha;u),\tilde \varphi(\alpha)\right)_n=\int_0^{\alpha}
d(z_1(t;u),\tilde \varphi(t))_n+(z_1(0;u),\tilde \varphi(0))_n,
$$
$$
\left(z_4(\beta;u),\tilde \varphi(\beta)\right)_n=-\int_{\beta}^T
d(z_4(t;u),\tilde \varphi(t))_n+(z_4(T;u),\tilde \varphi(T))_n,
$$
and that (we refer to the reasoning on p. \pageref{pt})
$$
\left(z_1(0;u),\tilde \varphi(0)\right)_n=\left(\bar{B}_0z_1(0;u)
,B_0\tilde \varphi(0)\right)_m +\left(\hat{B}_0z_1(0;u),\tilde
B_0\tilde \varphi(0)\right)_{n-m}=\left(\bar{B}_0z_1(0;u),\tilde f_0\right)_m,
$$
$$
\left(z_4(T;u),\tilde \varphi(T)\right)_n
=\left(\bar{B}_1z_4(T;u),B_1\tilde \varphi(T)\right)_{n-m}
+\left(\hat{B}_1z_4(T;u),\tilde
B_1\tilde \varphi(T)\right)_m=\left(\bar{B}_1z_4(T;u),\tilde f_1\right)_{n-m},
$$
we obtain
$$
(a,\tilde \varphi(s))_n-(\widehat{a,\tilde \varphi(s)})_n=
(z_2(s;u),\tilde \varphi(s))_n-(z_3(s;u),\tilde \varphi(s))_n
-(\widehat{a,\tilde \varphi(s)})_n=
$$
$$
=\left(z_2(\alpha;u),\tilde \varphi(\alpha)\right)_n
+\int^s_{\alpha}(H^T(t)u_1(t),\tilde \varphi(t))_ndt
+\int^s_{\alpha}(z_2(t;u),\tilde f(t))_ndt
$$
$$
-\left(z_3(\beta;u),\tilde \varphi(\beta)\right)_n
+\int_s^{\beta}(H^T(t)u_2(t),\tilde \varphi(t))_ndt
+\int_s^{\beta}(z_3(t;u),\tilde f(t))_ndt
$$
$$
-\int_{\alpha}^s(u_1(t),\tilde y(t))_ldt-
\int_s^{\beta}(u_2(t),\tilde y(t))_ldt-c
$$
$$
=\int_0^{\alpha} d(z_1(t;u),\tilde \varphi(t))_n+(z_1(0;u),\tilde \varphi(0))_n
+\int^s_{\alpha}(H^T(t)u_1(t),\tilde \varphi(t))_ndt
$$
$$
+\int^s_{\alpha}(z_2(t;u),\tilde f(t))_ndt +\int_{\beta}^T
d(z_4(t;u),\tilde \varphi(t))_n-(z_4(T;u),\tilde \varphi(T))_n
$$
$$
+\int_s^{\beta}(H^T(t)u_2(t),\tilde \varphi(t))_ndt+
\int_s^{\beta}(z_3(t;u),\tilde f(t))_ndt
-\int^s_{\alpha}(H^T(t)u_1(t),\tilde \varphi(t))_ndt
$$
$$
-\int^s_{\alpha}(u_1(t),\tilde \xi(t)_ldt
-\int_s^{\beta}(H^T(t)u_2(t),\tilde \varphi(t))_ndt
-\int_s^{\beta}(u_2(t),\tilde \xi(t)_ldt-c
$$
$$
=\int_0^{\alpha} \left(\frac
{dz_1(t;u)}{dt},\tilde \varphi(t)\right)_ndt+\int_0^{\alpha} \left(\frac
{d\tilde \varphi(t)}{dt},z_1(t;u)\right)_ndt+(\bar B_0z_1(0;u),\tilde f_0)_m
$$
$$
+\int^s_{\alpha}(z_2(t;u),\tilde f(t))_ndt +\int^T_{\beta} \left(\frac
{dz_4(t;u)}{dt},\tilde \varphi(t)\right)_ndt+\int^T_{\beta} \left(\frac
{d\tilde \varphi(t)}{dt},z_4(t;u)\right)_ndt
$$
$$
\!\!-(\bar B_1z_4(T;u),\tilde f_1)_{n-m}
+\int_s^{\beta}\!(z_3(t;u),\tilde f(t))_ndt
-\int^s_{\alpha}\!(u_1(t),\tilde \xi(t)_ldt
-\int_s^{\beta}\!(u_2(t),\tilde \xi(t)_ldt-c.
$$
Taking into notice that
$$
\frac {dz_1(t;u)}{dt}=A^T(t)z_1(t;u)\quad \mbox{на}\quad
(0,\alpha), \quad  \frac {dz_4(t;u)}{dt}=A^T(t)z_4(t;u)\quad
\mbox{на}\quad (\beta,T)
$$
$$
\mbox{и}\quad \frac {d\tilde \varphi(t)}{dt}=\tilde f(t)-A(t)\tilde \varphi(t)\quad
\mbox{на}\quad (0,T),
$$
and therefore
$$
\int_0^{\alpha} \left(\frac
{dz_1(t;u)}{dt},\tilde \varphi(t)\right)_ndt+\int_0^{\alpha} \left(\frac
{d\tilde \varphi(t)}{dt},z_1(t;u)\right)_ndt
$$
$$
=\int_0^{\alpha}(A^T(t)z_1(t;u),\tilde \varphi(t))_ndt+\int_0^{\alpha}(z_1(t;u),
\tilde f(t)-A(t)\tilde \varphi(t))_ndt=\int_0^{\alpha}(z_1(t;u), \tilde f(t))_ndt,
$$
$$
\int^T_{\beta} \left(\frac
{dz_4(t;u)}{dt},\tilde \varphi(t)\right)_ndt+\int^T_{\beta} \left(\frac
{d\tilde \varphi(t)}{dt},z_4(t;u)\right)_ndt=\int^T_{\beta}(z_4(t;u),
\tilde f(t))_ndt,
$$
we use the last equality to obtain
$$
(a,\tilde \varphi(s))_n-(\widehat{a,\tilde \varphi(s)})_n=
\left(\bar{B}_0z_1(0;u),\tilde f_0\right)_m+ \int^T_0(\tilde
z(t;u),\tilde f(t))_ndt
$$
\begin{equation}\label{xq}
- \left(\bar{B}_1z_4(T;u),\tilde f_1\right)_{n-m}
-\int^s_{\alpha}(u_1(t),\tilde \xi(t))_ldt-
\int_s^{\beta}(u_2(t),\tilde \xi(t))_ldt-c=:\eta,
\end{equation}
where
$$
\tilde z(t;u)=\left\{
\begin{array}{lc}
z_1(t;u),&0<t<\alpha;\\
z_2(t;u),&\alpha<t<s;\\z_3(t;u),&s<t<\beta;\\ z_4(t;u),&\beta<t<T.
\end{array}
\right.
$$
Recalling that $\tilde \xi(t)$ is a vector process with zero expectation, we use condition (\ref{d6}) and the known relationship
$D\eta=M\eta^2-(M\eta)^2$ that couples the dispersion
$D\eta:=M[\eta-M\eta]^2$ of random quantity $\eta$ with its expectation $M\eta,$ to obtain
$$
M\eta=\left(\bar{B}_0z_1(0;u),\tilde f_0\right)_m+ \int^T_0(\tilde
z(t;u),\tilde f(t))_ndt- \left(\bar{B}_1z_4(T;u),\tilde f_1\right)_{n-m}-c,
$$
$$
\eta-M\eta=-\int^s_{\alpha}(u_1(t),\tilde \xi(t))_ldt-
\int_s^{\beta}(u_2(t),\tilde \xi(t))_ldt,
$$
$$
D\eta=M[\eta-M\eta]^2=M\left[\int^s_{\alpha}(u_1(t),\tilde \xi(t))_ldt+
\int_s^{\beta}(u_2(t),\tilde \xi(t))_ldt\right]^2,
$$
$$
M\eta^2=D\eta+(M\eta)^2=M\left[\int^s_{\alpha}(u_1(t),\tilde \xi(t))_ldt+
\int_s^{\beta}(u_2(t),\tilde \xi(t))_ldt\right]^2
$$
$$
+\left[\left(\bar{B}_0z_1(0;u),\tilde f_0\right)_m+ \int^T_0(\tilde
z(t;u,\tilde f(t))_ndt-
\left(\bar{B}_1z_4(T;u),\tilde f_1\right)_{n-m}-c\right]^2,
$$
which yields
$$
\inf_{c \in \mathbb R^1} \sup_{\tilde F \in G,\, \tilde \xi \in V}
M[(a,\tilde \varphi(s))_n-(\widehat{a,\tilde \varphi(s)})_n]^2=
$$
$$
=\inf_{c \in \mathbb R^1}\sup_{\tilde F \in
G}\left[(\bar{B}_0z_1(0;u),\tilde f_0)_m+ \int^T_0(\tilde
z(t;u),\tilde f(t))_ndt- (\bar{B}_1z_4(T;u),\tilde f_1)_{n-m}-c \right]^2
$$
\begin{equation}\label{xq1}
+\sup_{\tilde \xi \in V} M\left[\int^s_{\alpha}(u_1(t),\tilde \xi(t))_ldt+
\int_s^{\beta}(u_2(t),\tilde \xi(t))_ldt\right]^2.
\end{equation}
In order to calculate the supremum on the right-hand side of
(\ref{xq1}) we apply the generalized Cauchy$-$Bunyakovsky
inequality \cite{BIBL366}. Let us write this inequality in the
form convenient for further analysis.

{\bf Lemma.} {\it For any $f_0^{(1)}, f_0^{(2)} \in R^m, \,\,
f_1^{(1)}, f_1^{(2)} \in R^{n-m},\,\, f_1,f_2 \in
\left(L^2(0,T)\right)^n$, the generalized Cauchy$-$Bunyakovsky inequality holds
$$
\left| (f_0^{(1)}, f_0^{(2)})_m+ (f_1^{(1)}, f_1^{(2)})_{n-m} +
\int_0^T(f_1(t),f_2(t))_n\,dt \right| \leq
$$
$$
\leq \left\{ (Q_0^{-1}f_0^{(1)}, f_0^{(1)})_m +
(Q_1^{-1}f_1^{(1)}, f_1^{(1)})_{n-m} + \int_0^T(Q_2^{-1}(t)f_1(t),
f_1(t))_n\,dt \right\}^\frac{1}{2} \times
$$
$$
\times \left\{ (Q_0f_0^{(2)}, f_0^{(2)})_m + (Q_1f_1^{(2)},
f_1^{(2)})_{n-m} + \int_0^T(Q_2(t)f_2(t), f_2(t))_n\,dt
\right\}^\frac{1}{2},
$$
in which the equality is attained at
$$
f_0^{(2)}=\lambda Q_0^{-1}f_0^{(1)}, \quad f_1^{(2)}=\lambda
Q_1^{-1}f_1^{(1)}, \quad f_2(t) =\lambda Q_2^{-1}f_1(t).
$$
} Setting in the generalized Cauchy$-$Bunyakovsky inequality
$$
f_0^{(1)}=\bar{B}_0z_1(0;u), \quad f_1^{(1)}=-\bar{B}_1z_4(T;u),
\quad f_1(t)=\tilde{z}(t;u),
$$
$$
f_0^{(2)}=\tilde f_0-f_0^{(0)}, \quad f_1^{(2)}=\tilde f_1-f_1^{(0)}, \quad f_2(t)=\tilde f(t)-f^{(0)}(t),
$$
and denoting
$$
Y:= (\bar{B}_0z_1(0;u),\tilde f_0-f_0^{(0)})_m - (\bar{B}_1z_4(T;u),\tilde f_1-f_1^{(0)})_{n-m}+
\int_0^T (\tilde{z}(t;u),\tilde f(t)-f^{(0)}(t))_n\,dt
$$
we obtain, in line with (2.7), the inequality
$$
|Y| \leq \left\{ \vphantom{\int_0^T}
(Q_0^{-1}\bar{B}_0z_1(0;u),\bar{B}_0z_1(0;u))_m +
(Q_1^{-1}\bar{B}_1z_4(T;u),\bar{B}_1z_4(T;u))_{n-m} + \right.$$
$$
\left.+\int_0^T(Q_2^{-1}(t)\tilde{z}(t;u),\tilde{z}(t;u))_n dt
\right\}^\frac{1}{2} \times \biggl\{ (Q_0(\tilde f_0-f_0^{(0)}),\tilde f_0-f_0^{(0)})_m
+(Q_1(\tilde f_1-f_1^{(0)}),\tilde f_1-f_1^{(0)})_{n-m}+ \biggr.
$$
$$
\left.+ \int_0^T(Q_2(t)(\tilde f(t)-f^{(0)}(t)),
\tilde f(t)-f^{(0)}(t))_n dt \right\}^\frac{1}{2} \leq
\left\{ \vphantom{\int_0^T}
(Q_0^{-1}\bar{B}_0z_1(0;u),\bar{B}_0z_1(0;u))_m  \right.
$$
$$
\left. + (Q_1^{-1}\bar{B}_1z_4(T;u),\bar{B}_1z_4(T;u))_{n-m} +
\int_0^T(Q_2^{-1}(t)\tilde{z}(t;u),\tilde{z}(t;u))_n dt
\right\}^\frac{1}{2}:= q,
$$
where the equality is attained at
$$
\tilde f_0=\pm \frac{1}{q}\, Q_0^{-1}\bar{B}_0z_1(0;u)+f_0^{(0)}, \quad \tilde f_1 = \mp
\frac{1}{q}\, Q_1^{-1}\bar{B}_1z_4(T;u)+f_1^{(0)}, \quad \tilde f(t)=\pm
\frac{1}{q}\, Q_2(t)\tilde{z}(t;u)+f^{(0)}(t).
$$
Thus,
$$
\inf_{c \in \mathbb R^1}\sup_{\tilde F \in
G}\left[(\bar{B}_0z_1(0;u),\tilde f_0)_m+ \int^T_0(\tilde
z(t;u),\tilde f(t))_ndt- (\bar{B}_1z_4(T;u),\tilde f_1)_{n-m}-c \right]^2
$$
$$
=\inf_{c \in R^1} \sup_{\tilde{F} \in G_0} \left[
(\bar{B}_0z_1(0;u),\tilde f_0-f_0^{(0)})_m - (\bar{B}_1z_4(T;u),\tilde f_1-f_1^{(0)})_{n-m}+
\int_0^T (\tilde{z}(t;u),\tilde f(t)-f^{(0)}(t))_n\,dt \right. =
$$
$$
\left.+(\bar{B}_0z_1(0;u),f_0^{(0)})_m - (\bar{B}_1z_4(T;u),f_1^{(0)})_{n-m}+
\int_0^T (\tilde{z}(t;u),f^{(0)}(t))_n\,dt- c \right]^2
$$
$$
=\inf_{c \in R^1} \sup_{|Y|\leq q} \left[Y+(\bar{B}_0z_1(0;u),f_0^{(0)})_m - (\bar{B}_1z_4(T;u),f_1^{(0)})_{n-m}+
\int_0^T (\tilde{z}(t;u),f^{(0)}(t))_n\,dt-c\right]^2= q^2
$$
$$
= (Q_0^{-1}\bar{B}_0z_1(0;u),\bar{B}_0z_1(0;u))_m +
(Q_1^{-1}\bar{B}_1z_4(T;u),\bar{B}_1z_4(T;u))_{n-m}
$$
\begin{equation}\label{xq0}
 + \int_0^T (Q_2^{-1}(t)\tilde{z}(t;u),\tilde{z}(t;u))_n dt
\end{equation}
at $c=(\bar{B}_0z_1(0;u),f_0^{(0)})_m - (\bar{B}_1z_4(T;u),f_1^{(0)})_{n-m}+
\int_0^T (\tilde{z}(t;u),f^{(0)}(t))_n\,dt.$

Calculate the second term on the right-hand side of (\ref{xq1}).
Setting
$$
\tilde u(t)=\left\{
\begin{array}{lc}
u_1(t),&\alpha< t<s,\\ u_2(t),&s<t< \beta,
\end{array}
\right.
$$
and applying the generalized Cauchy$-$Bunyakovsky inequality, we have
$$
M\left[\int^s_{\alpha}(u_1(t),\tilde \xi(t))_ldt+
\int_s^{\beta}(u_2(t),\tilde \xi(t))_ldt\right]^2
=M\left[\int^{\beta}_{\alpha}(\tilde u(t),\tilde \xi(t))_ldt\right]^2
$$
$$ \leq
M\left[\int^{\beta}_{\alpha}(Q^{-1}(t)\tilde u(t),\tilde u(t))_ldt
\cdot \int^{\beta}_{\alpha}(Q(t)\xi(t),\xi(t))_ldt\right]
$$
\begin{equation}\label{xq2}
=\int^{\beta}_{\alpha}(Q^{-1}(t)\tilde u(t),\tilde u(t))_ldt \cdot
\int^{\beta}_{\alpha}M(Q(t)\tilde \xi(t),\tilde \xi(t))_ldt.
\end{equation}
Here $M$ can be placed under the integral sign according to the
Fubini theorem because we assume that $\tilde \xi(t)$ is a random
process of the integrable second moment. Transform the last factor
on the right-hand side of (\ref{xq2}):
$$
\int^{\beta}_{\alpha}M(Q(t)\tilde \xi(t),\tilde \xi(t))_ldt
=\int^{\beta}_{\alpha} M(\sum_{i=1}^l(Q(t)\tilde \xi(t))_i\tilde \xi_i(t))dt
$$
$$
=\int^{\beta}_{\alpha}
M(\sum_{i=1}^l\sum_{k=1}^lq_{ik}\tilde \xi_k(t)\tilde \xi_i(t))dt
=\int^{\beta}_{\alpha}\sum_{i=1}^l\sum_{k=1}^lq_{ik}M(\tilde \xi_k(t)\tilde \xi_i(t))dt
$$
$$
=\int^{\beta}_{\alpha}Sp\,[Q(t)\tilde R(t,t)]dt.
$$
Taking into account that (\ref{d6}) holds, we see that (\ref{xq2})
yields
\begin{multline}\label{xq3}
\sup_{\tilde \xi \in V}M\left[\int^s_{\alpha}(u_1(t),\tilde \xi(t))_ldt+
\int_s^{\beta}(u_2(t),\tilde \xi(t))_ldt\right]^2\leq \\
\leq\int^s_{\alpha}\left(Q^{-1}(t)u_1(t),u_1(t)\right)_ldt+
\int_s^{\beta}\left(Q^{-1}(t)u_2(t),u_2(t)\right)_ldt.
\end{multline}
It is not difficult to check that here, the equality sign is attained at the element
$$
\tilde \xi(t)=\left\{
\begin{array}{lc}
\frac{\displaystyle \eta Q^{-1}(t)u_1(t)}{\displaystyle
\left[\int^s_{\alpha}\left(Q^{-1}(t)u_1(t),u_1(t)\right)_ldt+
\int_s^{\beta}\left(Q^{-1}(t)u_2(t),u_2(t)\right)_ldt\right]^{1/2}}\,,
&\alpha\leq t<s;\\[4pt] \frac{\displaystyle \eta
Q^{-1}(t)u_2(t)}{\displaystyle
\left[\int^s_{\alpha}\left(Q^{-1}(t)u_1(t),u_1(t)\right)_ldt+
\int_s^{\beta}\left(Q^{-1}(t)u_2(t),u_2(t)\right)_ldt\right]^{1/2}}\,,&s<t\leq
\beta,
\end{array}
\right.
$$
where $\eta$ is a random quantity such that $M\eta=0$ and
$M\eta^2=1.$ We conclude that statement of the lemma follows now from (\ref{xq1}), (\ref{xq0}), and
(\ref{xq3}).
\end{proof}

\section{Representations for minimax estimates of functionals of solutions
to two-point boundary value problems and estimation errors}

In this section we prove the theorem concerning general form of
minimax mean square estimates. Solving optimal control problem
(\ref{d5}), (\ref{N4}), we arrive at the following result.
\begin{pred}\label{tm1}
The minimax estimate of expression $(a,\varphi(s))$ has the form
$$
\widehat{\widehat{(a,\varphi(s))}}=
\int^s_{\alpha}(\hat{u}_1(t),y(t))_ldt+
\int_s^{\beta}(\hat{u}_2(t),y(t))_ldt
$$
where
\begin{equation}\label{ddd7}
\hat{u}_1(t)=Q(t)H(t)p_2(t),\quad \hat{u}_2(t)=Q(t)H(t)p_3(t),
\end{equation}
$$
\hat c=(\bar{B}_0z_1(0),f_0^{(0)})_m - (\bar{B}_1z_4(T),f_1^{(0)})_{n-m}+
\int_0^T (\tilde{z}(t),f^{(0)}(t))_n\,dt,
$$
$$
\tilde z(t)=\left\{
\begin{array}{lc}
z_1(t),&0<t<\alpha;\\
z_2(t),&\alpha<t<s;\\z_3(t),&s<t<\beta;\\ z_4(t),&\beta<t<T,
\end{array}
\right.
$$
and vector-functions $p_i(t)$ and $z_i(t)$, $i=\overline{1,4},$
are determined from the solution to the equation systems
$$
L^*z_1(t)=0,\quad  0<t<\alpha, \quad \hat{B}_0z_1(0)=0,
$$
$$
L^{\ast}z_2(t)=-H^{T}(t)Q(t)H(t)p_2(t),\quad  \alpha<t<s,\quad
z_2(\alpha)=z_1(\alpha),
$$
$$
L^{\ast}z_3(t)=-H^{T}(t)Q(t)H(t)p_3(t),\quad  s<t<\beta,\quad
z_3(s)=z_2(s)-a,
$$
\begin{equation} \label{d8}
L^{\ast}z_4(t)=0,\quad  \beta<t<T,\quad
z_4(\beta)=z_3(\beta),\quad \hat{B}_1z_4(T)=0,
\end{equation}
$$
Lp_1(t)=Q_2^{-1}(t)z_1(t),\quad  0<t<\alpha, \quad
B_0p_1(0)=Q_0^{-1}\bar B_0z_1(0),
$$
$$
Lp_2(t)=Q_2^{-1}(t)z_2(t),\quad  \alpha<t<s, \quad  p_2(\alpha)=
p_1(\alpha),
$$
$$
Lp_3(t)=Q_2^{-1}(t)z_3(t), \quad  s<t<\beta,\quad  p_3(s)=p_2(s),
$$
$$
Lp_4(t)=Q_2^{-1}(t)z_4(t),\quad  \beta<t<T,\quad
p_4(\beta)=p_3(\beta), \quad B_1p_4(T)=-Q_1^{-1}\bar B_1z_4(T).
$$
Here $z_1,p_1\in H^1(0,\alpha)^n,$ $ z_2,p_2\in H^1(\alpha,s)^n,$
$z_3,p_3\in H^1(s,\beta)^n,$ and $z_4,p_4\in H^1(\beta,T)^n.$ The
minimax estimation error
\begin{equation} \label{ddd8}
\sigma=(a,p_2(s))_n^{1/2}.
\end{equation}
System (\ref{d8}) is uniquely solvable.
\end{pred}
\begin{proof}
We will solve optimal control problem  (\ref{d5}), (\ref{N4}).
Represent solutions $z_i(x;u),$ $i=\overline{1,4},$ of problem
(\ref{d5}) as $z_i(t;u)=\bar z_i(t;u)+\Bar{\Bar z}_i(t),$ where
$\bar z_1(t;u),$ $\bar z_2(t;u),$ $\bar z_3(t;u),$ $\bar z_4(t;u)$
and $\Bar{\Bar z}_1(t),$ $\Bar{\Bar z}_2(t),$ $\Bar{\Bar z}_3(t),$
$\Bar{\Bar z}_4(t)$ \label{BB} denote the solutions to this
problem at $a=0$ and $u_1\equiv 0,$ $u_2\equiv 0,$ respectively.
Then function (\ref{N4}) can be represented in the form \label{D}
$$
I(u)=\tilde I(u)+2L(u) +A,
$$
where
$$
\tilde I(u) =(Q_0^{-1}\bar{B}_0\bar z_1(0;u),\bar{B}_0\bar
z_1(0;u))_m + (Q_1^{-1}\bar{B}_1\bar z_4(T;u),\bar{B}_1\bar
z_4(T;u))_{n-m}
$$
$$
+ \int_0^{\alpha} (Q_2^{-1}(t)\bar z_1(t;u),\bar z_1(t;u))_n dt +
\int_{\alpha}^s (Q_2^{-1}(t)\bar z_2(t;u),\bar z_2(t;u))_n dt
$$
$$
+ \int_{s}^{\beta} (Q_2^{-1}(t)\bar z_3(t;u),\bar z_3(t;u))_n dt+
\int_{\beta}^T (Q_2^{-1}(t)\bar z_4(t;u),\bar z_4(t;u))_n dt
$$
$$
+\int^s_{\alpha}\left(Q^{-1}(t)u_1(t),u_1(t)\right)_ldt+
\int_s^{\beta}\left(Q^{-1}(t)u_2(t),u_2(t)\right)_ldt,
$$
$$
L(u)=(Q_0^{-1}\bar{B}_0\bar z_1(0;u),\bar{B}_0 \Bar{\Bar
z}_1(0))_m + (Q_1^{-1}\bar{B}_1\bar z_4(T;u),\bar{B}_1\Bar{\Bar
z}_4(T))_{n-m}
$$
$$
+ \int_0^{\alpha} (Q_2^{-1}(t)\bar z_1(t;u),\Bar{\Bar z}_1(t))_n
dt + \int_{\alpha}^s (Q_2^{-1}(t)\bar z_2(t;u),\Bar{\Bar
z}_2(t))_n dt
$$
$$
+ \int_{s}^{\beta} (Q_2^{-1}(t)\bar z_3(t;u),\Bar{\Bar z}_3(t))_n
dt+ \int_{\beta}^T (Q_2^{-1}(t)\bar z_4(t;u),\Bar{\Bar z}_4(t))_n
dt,
$$
$$
A=(Q_0^{-1}\bar{B}_0\Bar{\Bar z}_1(0),\bar{B}_0\Bar{\Bar
z}_1(0))_m + (Q_1^{-1}\bar{B}_1\Bar{\Bar
z}_4(T),\bar{B}_1\Bar{\Bar z}_4(T))_{n-m}
$$
$$
+ \int_0^{\alpha} (Q_2^{-1}(t)\Bar{\Bar z}_1(t),\Bar{\Bar
z}_1(t))_n dt + \int_{\alpha}^s (Q_2^{-1}(t)\Bar{\Bar
z}_2(t),\Bar{\Bar z}_2(t))_n dt
$$
$$
+ \int_{s}^{\beta} (Q_2^{-1}(t)\Bar{\Bar z}_3(t),\Bar{\Bar
z}_3(t))_n dt+ \int_{\beta}^T (Q_2^{-1}(t)\Bar{\Bar
z}_4(t),\Bar{\Bar z}_4(t))_n dt.
$$
Since solution $\bar z(t;u)$ of BVP (\ref{d31}) is
continuous\footnote{This continuous dependence follows from the
representation of function $\bar z(t;u)$ in terms of Green's
matrix $G^*(t,\xi)$ of BVP  (\ref{d31}) (see \cite{BIBLNaym}, p.
115):
$$
\bar z(t;u)=-\int_{\alpha}^{s}G^*(t,\xi)H^T(\xi)u_1(\xi)\,d\xi
-\int_{s}^{\beta}G^*(t,\xi)H^T(\xi)u_2(\xi)\,d\xi.
$$
} with respect to right-hand side $g(t;u)$ defined by
(\ref{du31}), the function $u \to \bar z(\cdot;u)$ is a linear
bounded operator mapping the space $H=L^2(\alpha,s)\times
L^2(s,\beta)$ to
$$
H_1:=H^1(0,\alpha)\times H^1(\alpha,s)\times H^1(s,\beta)\times
H^1(\beta,T).
$$
Thus, $\tilde I(u)$ is a continuous quadratic form corresponding
to a symmetric continuous bilinear form
$$
\pi(u,v):=(Q_0^{-1}\bar{B}_0\bar z_1(0;u),\bar{B}_0\bar
z_1(0;v))_m + (Q_1^{-1}\bar{B}_1\bar z_4(T;u),\bar{B}_1\bar
z_4(T;v))_{n-m}
$$
$$
+ \int_0^{\alpha} (Q_2^{-1}(t)\bar z_1(t;u),\bar z_1(t;v))_n dt +
\int_{\alpha}^s (Q_2^{-1}(t)\bar z_2(t;u),\bar z_2(t;v))_n dt
$$
$$
+ \int_{s}^{\beta} (Q_2^{-1}(t)\bar z_3(t;u),\bar z_3(t;v))_n dt+
\int_{\beta}^T (Q_2^{-1}(t)\bar z_4(t;u),\bar z_4(t;v))_n dt
$$
$$
+\int^s_{\alpha}\left(Q^{-1}(t)u_1(t),v_1(t)\right)_ldt+
\int_s^{\beta}\left(Q^{-1}(t)u_2(t),v_2(t)\right)_ldt,
$$
$L(u)$ is a linear continuous functional defined on $H,$ and $A$
is a constant independent of  $u.$ We have
\begin{equation*}
 \tilde
I(u)=\tilde I(u_1,u_2) \geq
\int^s_{\alpha}\left(Q^{-1}(t)u_1(t),u_1(t)\right)_ldt+
\int_s^{\beta}\left(Q^{-1}(t)u_2(t),u_2(t)\right)_ldt \geq
c\|u\|_H^2,\quad \mbox{c=const};
\end{equation*}
using Theorem 1.1 from \cite{BIBL20}, we conclude that there is
one and only one element  $\hat u=(\hat u_1,\hat u_2) \in H$ such
that
$$
I(\hat u)=I(\hat u_1,\hat u_2)= \inf_{( u_1,u_2) \in H} I(u_1,u_2) =
\inf_{u_1\in L^2(\alpha,s),u_2 \in L^2(s,\beta)}I(u_1, u_2).
$$
Therefore
$$
\frac {d}{d\tau} I(\hat u_1+\tau v_1, \hat u_2+\tau v_2)
\left.\right|_{\tau=0}=0, \quad \forall v=(v_1,v_2) \in H.
$$
Taking into consideration the latter equality,  (\ref{N4}), and
designations on p. \pageref{BB}, we obtain
$$
0 =\frac{1}{2} \frac{dI(\hat{u}+\tau
v)}{d\tau}\left.\right|_{\tau=0}=
$$
$$
=\int_{\alpha}^s(Q^{-1}(t)\hat{u}_1,v_1)_l dt +
\int_s^{\beta}(Q^{-1}(t)\hat{u}_2,v_2)_ldt+
$$
$$
+(Q_0^{-1}\bar{B}_0z_1(0;\hat{u}),\bar{B}_0\bar{z}_1(0;v))_m +
(Q_1^{-1}\bar{B}_1z_4(T;\hat{u}),\bar{B}_1\bar{z}_4(T;v))_{n-m}
$$
$$
+\int_0^{\alpha}(Q_2^{-1}(t)z_1(t;\hat{u}),\bar{z}_1(t;v))_n dt+
\int_{\alpha}^s(Q_2^{-1}(t)z_2(t;\hat{u}),\bar{z}_2(t;v))_n dt
$$
\begin{equation}\label{j1}
+\int_s^{\beta}(Q_2^{-1}(t)z_3(t;\hat{u}),\bar{z}_3(t;v))_n dt +
\int_{\beta}^T(Q_2^{-1}(t)z_4(t;\hat{u}),\bar{z}_4(t;v))_n dt
\end{equation}
Introduce functions $p_1\in H^1(0,\alpha)^n,$ $p_2\in
H^1(\alpha,s)^n,$ $p_3\in H^1(s,\beta)^n,$ and $p_4\in
H^1(\beta,T)^n$ as unique solutions to the following problem:
$$
Lp_1(t)=Q_2^{-1}(t)z_1(t;\hat{u}), \quad 0<t<\alpha, \quad
B_0p_1(0)=Q_0^{-1}\bar{B}_0z_1(0;\hat{u}),
$$
$$
Lp_2(t)=Q_2^{-1}(t)z_2(t;\hat{u}), \quad \alpha<t<s, \quad
p_2(\alpha)=p_1(\alpha),
$$
\begin{equation}\label{j2}
Lp_3(t)=Q_2^{-1}(t)z_3(t;\hat{u}), \quad s<t<\beta, \quad
p_3(s)=p_2(s),
\end{equation}
$$
Lp_4(t)=Q_2^{-1}(t)z_4(t;\hat{u}), \quad \beta<t<T,
$$
$$
p_4(\beta)=p_3(\beta), \quad
B_1p_4(T)=-Q_1^{-1}\bar{B}_1z_4(T;\hat{u}).
$$
Now transform the sum of the last for terms on the right-hand side
of (\ref{j1}) taking into notice that $(\bar
z_1(0;v),p_1(0))_n=(\bar{B}_0\bar z_1(0;v),B_0p_1(0))_m$ and
$(\bar z_4(T;v),p_4(T))_n=(\bar{B}_1\bar
z_4(T;v),B_1p_4(T))_{n-m}.$ We have
$$
\int_0^{\alpha}(Q_2^{-1}(t)z_1(t;\hat{u}),\bar{z}_1(t;v))_n dt+
\int_{\alpha}^s(Q_2^{-1}(t)z_2(t;\hat{u}),\bar{z}_2(t;v))_n dt
$$
$$
+\int_s^{\beta}(Q_2^{-1}(t)z_3(t;\hat{u}),\bar{z}_3(t;v))_n dt +
\int_{\beta}^T(Q_2^{-1}(t)z_4(t;\hat{u}),\bar{z}_4(t;v))_n dt
$$
$$
=\int_0^{\alpha}(Lp_1(t),\bar{z}_1(t;v))_n dt +
\int_{\alpha}^s(Lp_2(t),\bar{z}_2(t;v))_n dt
$$
$$
+\int_s^{\beta}(Lp_3(t),\bar{z}_3(t;v))_n dt + \int_{\beta}^T
(Lp_4(t),\bar{z}_4(t;v))_n dt
$$
$$
=\int_0^{\alpha}(p_1(t),L^{\ast}\bar z_1(t;v))_n dt + (\bar
z_1(\alpha;v),p_1(\alpha))_n-(\bar z_1(0;v),p_1(0))_n
$$
$$
+\int_{\alpha}^s(p_2(t),L^{\ast}\bar z_2(t;v))_n dt + (\bar
z_2(s;v),p_2(s))_n-(\bar z_2(\alpha;v),p_2(\alpha))_n
$$
$$
+\int_s^{\beta}(p_3(t),L^{\ast}\bar z_3(t;v))_n dt + (\bar
z_3(\beta;v),p_3(\beta))_n-(\bar z_3(s;v),p_3(s))_n
$$
$$
+\int_{\beta}^T (p_4(t),L^{\ast}\bar z_4(t;v))_n dt + (\bar
z_4(T;v),p_4(T))_n-(\bar z_4(\beta;v),p_4(\beta))_n
$$
$$
=-(\bar z_1(0;v),p_1(0))_n+(\bar z_4(T;v),p_4(T))_n
$$
$$
-\int_{\alpha}^s(p_2(t),H^T(t)v_1(t))_ndt -
\int_s^{\beta}(p_3(t),H^T(t)v_2(t))_ndt + (\bar z_2(s;v)-\bar
z_3(s;v),p_2(s))_n
$$
$$
=-(\bar{B}_0\bar z_1(0;v),B_0p_1(0))_m+ (\bar{B}_1\bar
z_4(T;v),B_1p_4(T))_{n-m}
$$
$$
-\int_{\alpha}^s (H(t)p_2(t),v_1(t))_ldt-
\int_s^{\beta}(H(t)p_3(t),v_2(t))_ldt
$$
$$
=-(\bar{B}_0\bar z_1(0;v),Q_0^{-1}\bar{B}_0 z_1(0;\hat{u}))_m-
(B_1\bar z_4(T;v),Q_1^{-1}\bar{B}_1z_4(T;\hat{u}))_{n-m}
$$
\begin{equation}\label{j3}
-\int_{\alpha}^s (H(t)p_2(t),v_1(t))_ldt-
\int_s^{\beta}(H(t)p_3(t),v_2(t))_ldt
\end{equation}
From equalities (\ref{j1})--(\ref{j3}) it follows that
$$
\int_{\alpha}^s (Q^{-1}(t)\hat{u}_1(t),v_1(t))_l dt +
\int_s^{\beta}(Q^{-1}(t)\hat{u}_2(t),v_2(t))_l dt
$$
$$
=\int_{\alpha}^s(H(t)p_2(t),v_1(t))_l dt
+\int_s^{\beta}(H(t)p_3(t),v_2(t))_l dt,
$$
so that
$$
Q^{-1}(t)\hat{u}_1(t)=H(t)p_2(t),\quad
Q^{-1}(t)\hat{u}_2(t)=H(t)p_3(t),
$$
\begin{equation}\label{j4}
\hat{u}_1(t)=Q(t)H(t)p_2(t), \quad \hat{u}_2(t)=Q(t)H(t)p_3(t).
\end{equation}
Functions $p_1(t),$ $p_2(t),$ $p_3(t),$ and $p_4(t)$ are
absolutely continuous on segments $[0,\alpha],$ $[\alpha,s],$
$[s,\beta]$, and $[\beta,T]$, respectively, as solutions to BVP
(\ref{j2}); therefore, functions $\hat u_1(t)$ and $\hat u_2(t)$
that perform optimal control are continuous on $[\alpha,s]$ and
$[s,\beta].$ Replacing in (\ref{d5}) functions $u_1(t)$ and $
u_2(t)$ by
 $\hat u_1(t)$ and $\hat u_2(t)$ defined by formulas
(\ref{j4}) and denoting $z(t)=z(t;\hat u)$ we arrive at problem
(\ref{d8}) and equalities (\ref{ddd7}).

Taking into consideration the way this problem was formulated we
can state that its unique solvability follows from the fact that
 functional (\ref{N4}) has one minimum point $\hat u$.

Now let us prove representation (\ref{ddd8}). Substituting into
formula  $ \sigma^2=I(\hat u)$ expressions (\ref{ddd7}) for $\hat
u_1(t)$ and $\hat u_2(t),$ we have
$$
\sigma^2=(Q_0^{-1}\bar{B}_0z_1(0),\bar{B}_0z_1(0))_m +
(Q_1^{-1}\bar{B}_1z_4(T),\bar{B}_1z_4(T))_{n-m}
$$
$$
+\int_0^{\alpha} (Q_2^{-1}(t) z_1(t), z_1(t))_n dt +
\int_{\alpha}^s (Q_2^{-1}(t) z_2(t), z_2(t))_n dt
$$
$$
+ \int_{s}^{\beta} (Q_2^{-1}(t) z_3(t), z_3(t))_n dt+
\int_{\beta}^T (Q_2^{-1}(t) z_4(t), z_4(t))_n dt
$$
\begin{equation}\label{k1}
+\int^s_{\alpha}\left(Q(t)H(t)p_2(t), H(t)p_2(t)\right)_ldt+
\int_s^{\beta}\left(Q(t)H(t)p_3(t), H(t)p_3(t)\right)_ldt.
\end{equation}
Next, we can apply the reasoning similar to that on p.
\pageref{pt} and use (\ref{d8}) to obtain
$$
(z_1(0),p_1(0))_n=(\bar B_0z_1(0),B_0p_1(0))_m=(\bar
B_0z_1(0),Q_0^{-1}\bar B_0z_1(0))_m,
$$
which yields
$$
\int_0^{\alpha} (Q_2^{-1}(t) z_1(t), z_1(t))_n
dt+\int^s_{\alpha}\left(Q(t)H(t)p_2(t),H(t)p_2(t)\right)_ldt
$$
$$
=\int_0^{\alpha} (Lp_1(t), z_1(t))_n
dt-\int^s_{\alpha}(L^{\ast}z_2(t),p_2(t))_ndt
$$
$$
=\int_0^{\alpha} (p_1(t), L^*z_1(t))_n dt
+(z_1(\alpha),p_1(\alpha))_n-(z_1(0),p_1(0))_n
$$
$$
-\int_{\alpha}^s(z_2(t),Lp_2(t))_ndt -(z_2(\alpha),p_2(\alpha))_n
+(z_2(s),p_2(s))_n
$$
$$
=-(z_1(0),p_1(0))_n-\int_{\alpha}^s(z_2(t),Q_2^{-1}(t)z_2(t))_ndt
+(z_2(s),p_2(s))_n
$$
\begin{equation}\label{k2}
 =-\int_{\alpha}^s(Q_2^{-1}(t)z_2(t),z_2(t))_ndt- (\bar
B_0z_1(0),Q_0^{-1}\bar B_0z_1(0))_m+(z_2(s),p_2(s))_n
\end{equation}
In a similar manner, using the equality
$$
(z_4(T),p_4(T))_n=(\bar B_1z_4(T),B_1p_4(T))_{n-m}=-(\bar
B_1z_4(T),Q_1^{-1}\bar B_1z_4(T))_{n-m},
$$
we obtain
$$
\int_{\beta}^T (Q_2^{-1}(t) z_4(t), z_4(t))_n
dt+\int_s^{\beta}\left(Q(t)H(t)p_3(t), H(t)p_3(t)\right)_ldt
$$
$$
=\int_{\beta}^T (Lp_4(t), z_4(t))_n
dt-\int_s^{\beta}(L^{\ast}z_3(t),p_3(t))_ndt
$$
$$
=\int_{\beta}^T (p_4(t), L^*z_4(t))_n dt
+(p_4(T),z_4(T))_n-(p_4(\beta),z_4(\beta))_n
$$
$$
-\int_s^{\beta}(z_3(t),Lp_3(t))_ndt +(z_3(\beta),p_3(\beta))_n
-(z_3(s),p_3(s))_n
$$
$$
=(p_4(T),z_4(T))_n-\int_s^{\beta}(z_3(t),Q_2^{-1}(t)z_3(t))_ndt
-(z_3(s),p_2(s))_n
$$
\begin{equation}\label{k3}
\!\!=\!-\int_s^{\beta}\!\!(Q_2^{-1}(t)z_3(t),z_3(t))_ndt- (\bar
B_1z_4(T),Q_1^{-1}\bar B_1z_4(T))_{n-m}-(z_3(s),p_2(s))_n.
\end{equation}
Relationships (\ref{k1})--(\ref{k3}) yield
$$
\sigma^2=(a,p_2(s)),
$$
which is to be proved.
\end{proof}

Obtain now another representation for the minimax estimate of
quantity $(a,\varphi(s))_n$ which is independent of vector $a.$ To
this end, introduce vector-functions
 $\hat{p}_1,\hat{\varphi}_1\in H^1(0,\alpha)^n,$
 $ \hat{p}_2,\hat{\varphi}_2\in H^1(\alpha,s)^n,$
 $\hat{p}_3,\hat{\varphi}_3\in H^1(s,\beta)^n,$
 and $\hat{p}_4,\hat{\varphi}_4\in H^1(\beta,T)^n$
as solutions to the equation system
$$
L^{\ast}\hat{p}_1(t)=0,\quad 0<t<\alpha, \quad
\hat{B}_0\hat{p}_1(0)=0,
$$
$$
L^*\hat{p}_2(t)=H^T(t)Q(t)(y(t)-H(t)\hat{\varphi}_2(t)),\quad
\alpha<t<s, \quad \hat{p}_2(\alpha)=\hat{p}_1(\alpha),
$$
$$
L^*\hat{p}_3(t)=H^T(t)Q(t)(y(t)-H(t)\hat{\varphi}_3(t)), \quad
s<t<\beta,\quad  \hat{p}_3(s)=\hat{p}_2(s),
$$
\begin{equation}\label{tk}
L^*\hat{p}_4(t)=0,\quad  \beta<t<T,\quad
\hat{p}_4(\beta)=\hat{p}_3(\beta),\quad \hat{B}_1\hat{p}_4(T)=0,
\end{equation}
$$
L\hat{\varphi}_1(t)=Q_2^{-1}(t)\hat{p}_1(t)+f^{(0)}(t),\quad  0<t<\alpha,
\quad B_0\hat{\varphi}_1(0)=Q_0^{-1}\bar{B}_0\hat{p}_1(0)+f_0^{(0)},
$$
$$
L\hat{\varphi}_2(t)=Q_2^{-1}(t)\hat{p}_2(t)+f^{(0)}(t),\quad  \alpha<t<s,
\quad \hat{\varphi}_2(\alpha)=\hat{\varphi}_1(\alpha),
$$
$$
L\hat{\varphi}_3(t)=Q_2^{-1}(t)\hat{p}_3(t)+f^{(0)}(t), \quad
 s<t<\beta,\quad  \hat{\varphi}_3(s)=\hat{\varphi}_2(s),
$$
$$
L\hat{\varphi}_4(t)=Q_2^{-1}(t)\hat{p}_4(t)+f^{(0)}(t),\quad  \beta<t<T,\quad
$$
$$
\hat{\varphi}_4(\beta)=\hat{\varphi}_3(\beta),\quad
B_1\hat{\varphi}_4(T)=-Q_1^{-1}\bar{B}_1\hat{p}_4(T)-f_1^{(0)}
$$
at realizations $y$ that belong with probability 1 to space
$L^2(\alpha,\beta).$

Note that unique solvability of problem  (\ref{tk}) can be proved
similarly to the case of (\ref{d8}). Namely, one can show that
solutions to the problem of optimal control of the system
$$
L^{\ast}\hat{p}_1(t;v)=0,\quad 0<t<\alpha, \quad
\hat{B}_0\hat{p}_1(0;v)=0,
$$
$$
L^{\ast}\hat{p}_2(t;v)=d(t)-H^T(t)v_1(t),\quad \alpha<t<s, \quad
\hat{p}_1(\alpha;v)=\hat{p}_2(\alpha;v),
$$
$$
L^{\ast}\hat{p}_3(t;v)=d(t)-H^T(t)v_2(t), \quad
 s<t<\beta,\quad  \hat{p}_2(s;v)=\hat{p}_3(s;v),$$
$$
L^{\ast}\hat{p}_4(t;v)=0,\quad  \beta<t<T,\quad \hat
p_4(\beta;v)=\hat p_3(\beta;v),\quad \hat{B}_1\hat{p}_4(T;v)=0
$$
with the cost  function
$$
J(v) =\left(Q_0^{-1}(\bar{B}_0 \hat p_1(0;v)+Q_0f_0^{(0)}),\bar{B}_0 \hat p_1(0;v)+Q_0f_0^{(0)}\right)_m
$$
$$
+ \left(Q_1^{-1}(\bar{B}_1\hat p_4(T;v)+Q_1f_1^{(0)}),\bar{B}_1\hat p_4(T;v)+Q_1f_1^{(0)}\right)_{n-m}
$$
$$
+ \int_0^{\alpha} (Q_2^{-1}(t)(\hat p_1(t;v)+Q_2(t)f^{(0)}(t)), \hat p_1(t;v)+Q_2(t)f^{(0)}(t))_n dt
$$
$$
+\int_{\alpha}^s (Q_2^{-1}(t) (\hat p_2(t;v)+Q_2(t)f^{(0)}(t)), \hat p_2(t;v)+Q_2(t)f^{(0)}(t))_n dt
$$
$$
+ \int_{s}^{\beta} (Q_2^{-1}(t)(\hat p_3(t;v)+Q_2(t)f^{(0)}(t)), \hat p_3(t;v)+Q_2(t)f^{(0)}(t))_n
dt
$$
$$
+ \int_{\beta}^T (Q_2^{-1}(t)(\hat p_4(t;v)+Q_2(t)f^{(0)}(t)), \hat p_4(t;v)+Q_2(t)f^{(0)}(t))_n dt
$$
$$
+\int^s_{\alpha}\left(Q^{-1}(t)v_1(t),v_1(t)\right)_ldt+
\int_s^{\beta}\left(Q^{-1}(t)v_2(t),v_2(t)\right)_ldt\to \min_{v=(v_1,v_2)\in H},
$$
$$
d(t)=H^T(t)Q(t)y(t),\quad \alpha<t<\beta,
$$
can be reduced to the solution of problem (\ref{tk}) where the
optimal control $\hat v=(\hat v_1,\hat v_2)$ is expressed in terms
of the solution to this problem as $\hat
v_1=Q(t)H(t)\hat\varphi_2(t),$ $\hat
v_2=Q(t)H(t)\hat\varphi_3(t)$; the unique solvability of the
problem follows from the existence of the unique minimum point
$\hat v$ of functional $J(v)$.

Considering system (\ref{tk}) at realizations $y$ it is easy to
see that its solution is continuous with respect to the right-hand
side. This property enables us to conclude, using the general
theory of linear continuous  transformations of random processes,
that these solutions, i.e. the functions $\hat
p_i(t)=\hat p_i(t,\omega),$ $\hat \varphi_1(t)=\hat
\varphi_1(t,\omega),$ $i=\overline{1,4},$ considered as random
fields have finite second moments.
\begin{pred}\label{tm2}
The following representation is valid
$$
\widehat{\widehat{(a,\varphi(s))_n}}=(a,\hat{\varphi}_2(s))_n.
$$
\end{pred}
\begin{proof}
By virtue of (\ref{ddd7}) and (\ref{tk}),
$$
\widehat{\widehat{(a,\varphi(s))_n}} =
\int_{\alpha}^s(\hat{u}_1(t),y(t))_ldt
+\int_s^{\beta}(\hat{u}_2(t),y(t))_ldt+\hat c
$$
$$
=\int_{\alpha}^s(Q(t)H(t)p_2(t),y(t))_ldt +
\int_{\alpha}^s(Q(t)H(t)p_3(t),y(t))_ldt+\hat c
$$
\begin{equation}\label{aj19}
=\int_{\alpha}^s(p_2(t),H^T(t)Q(t)y(t))_ndt +
\int_{\alpha}^s(p_3(t),H^T(t)Q(t)y(t))_ndt+\hat c .
\end{equation}
Next,
\begin{equation*}
\int_{\alpha}^s(p_2(t),H^T(t)Q(t)y(t))_l dt=
\end{equation*}
$$
=\int_{\alpha}^s(p_2(t),L^{\ast}\hat{p}_2(t))_n dt +
\int_{\alpha}^s(p_2(t),H^T(t)Q(t)H(t)\hat{\varphi}_2(t))_n dt
$$
$$
=\int_{\alpha}^s(Lp_2(t),\hat{p}_2(t))_ndt +
(p_2(\alpha),\hat{p}_2(\alpha))_n - (p_2(s),\hat{p}_2(s))_n
$$
$$
+\int_{\alpha}^s(H^T(t)Q(t)H(t)p_2(t),\hat{\varphi}_2(t))_n dt
$$
$$
=\int_{\alpha}^s(Q_2^{-1}(t)z_2(t),\hat{p}_2(t))_n dt  +
(p_2(\alpha),\hat{p}_2(\alpha))_n - (p_2(s),\hat{p}_2(s))_n
$$
$$
-\int_{\alpha}^s(L^{\ast}z_2(t),\hat{\varphi}_2(t))_ndt
$$
\begin{equation*}
=\int_{\alpha}^s(z_2(t),L\hat{\varphi}_2(t))_n dt  +
(p_2(\alpha),\hat{p}_2(\alpha))_n - (p_2(s),\hat{p}_2(s))_n -
\int_{\alpha}^s(z_2(t),L\hat{\varphi}_2(t))_n dt
\end{equation*}
\begin{equation*}
+(z_2(s),\hat{\varphi}_2(s))_n -
(z_2(\alpha),\hat{\varphi}_2(\alpha))_n-\int_{\alpha}^s(z_2(t),f^{(0)}(t))_ndt
\end{equation*}
\begin{equation}\label{aj20}
=(p_2(\alpha),\hat{p}_2(\alpha))_n - (p_2(s),\hat{p}_2(s))_n
+(z_2(s),\hat{\varphi}_2(s))_n -
(z_2(\alpha),\hat{\varphi}_2(\alpha))_n-\int_{\alpha}^s(z_2(t),f^{(0)}(t))_ndt.
\end{equation}
Similarly,
$$
\int_s^{\beta}(p_3(t),H^T(t)Q(t)y(t))_l dt=
$$
\begin{equation}\label{j19}
=(p_3(s),\hat{p}_3(s))_n - (p_3(\beta),\hat{p}_3(\beta))_n
+(z_3(\beta),\hat{\varphi}_3(\beta))_n -
(z_3(s),\hat{\varphi}_3(s))_n-\int_{s}^\beta(z_3(t),f^{(0)}(t))_ndt
\end{equation}
From (\ref{aj19}), (\ref{aj20}), and (\ref{j19}) it follows
$$
(\widehat{\widehat{a,\varphi(s)}} ) =
(p_1(\alpha),\hat{p}_1(\alpha))_n -\int_{\alpha}^s(z_2(t),f^{(0)}(t))_ndt-\int_{s}^\beta(z_3(t),f^{(0)}(t))_ndt
$$
\begin{equation}\label{jj19}
-(z_1(\alpha),\hat{\varphi}_1(\alpha))_n -
(p_4(\beta),\hat{p}_4(\beta))_n +
(z_4(\beta),\hat{\varphi}_4(\beta))_n + (a,\hat{\varphi}_2(s))_n+\hat c.
\end{equation}
However,
$$
0=\int_0^{\alpha}(\hat \varphi_1(t),L^*\hat z_1(t))_n dt=
\int_0^{\alpha}(L\hat \varphi_1(t),z_1(t))_n dt
$$
$$
-(z_1(\alpha),\hat \varphi_1(\alpha))_n+(z_1(0),\hat
\varphi_1(0))_n
$$
\begin{equation}\label{j20}
=\int_0^{\alpha}(z_1(t),Q_2(t)\hat p_1(t))_n dt+\int_0^{\alpha}(z_1(t),f^{(0)}(t))_n dt -(z_1(\alpha),\hat
\varphi_1(\alpha))_n+(z_1(0),\hat \varphi_1(0))_n,
\end{equation}
$$
0=\int_0^{\alpha}(L^*\hat p_1(t),p_1(t))_n dt=
\int_0^{\alpha}(\hat p_1(t),Lp_1(t))_n dt
$$
$$
-(\hat p_1(\alpha),p_1(\alpha))_n+(\hat p_1(0),p_1(0))_n
$$
\begin{equation}\label{j21}
=\int_0^{\alpha}(\hat p_1(t),Q_2(t) z_1(t))_n dt-(\hat
p_1(\alpha),p_1(\alpha))_n+(\hat p_1(0),p_1(0))_n.
\end{equation}
Subtracting from (\ref{j20}) equality (\ref{j21}), we obtain
$$
0=(\hat p_1(\alpha),p_1(\alpha))_n-(z_1(\alpha),\hat
\varphi_1(\alpha))_n -(\hat p_1(0),p_1(0))_n+(z_1(0),\hat
\varphi_1(0))_n+\int_0^{\alpha}(z_1(t),f^{(0)}(t))_n dt
$$
or
\begin{equation}\label{j22}
\int_0^{\alpha}(z_1(t),f^{(0)}(t))_n dt+(\hat p_1(\alpha),p_1(\alpha))_n-(z_1(\alpha),\hat
\varphi_1(\alpha))_n =(\hat p_1(0),p_1(0))_n-(z_1(0),\hat
\varphi_1(0))_n.
\end{equation}
Since
$$
(z_1(0),\hat \varphi_1(0))_n=(\bar B_0z_1(0),B_0\hat
\varphi_1(0))_m=(\bar B_0z_1(0),Q_0^{-1}\bar B_0\hat p_1(0))_m+(\bar B_0z_1(0),f_0^{(0)})_m,
$$
$$
(\hat p_1(0),p_1(0))_n=(\bar B_0\hat p_1(0),B_0p_1(0))_m=(\bar
B_0\hat p_1(0),Q_0^{-1}\bar B_0z_1(0))_m,
$$
we can use the latter equalities, (\ref{j22}), and the fact that
$Q_0^{-1}$ is a symmetric matrix to obtain
\begin{equation}\label{j23}
(\hat p_1(\alpha),p_1(\alpha))_n-(z_1(\alpha),\hat
\varphi_1(\alpha))_n =-\int_0^{\alpha}(z_1(t),f^{(0)}(t))_n dt-(\bar B_0z_1(0),f_0^{(0)})_m.
\end{equation}
Performing a similar analysis, one can prove that
\begin{equation}\label{j24}
-(p_4(\beta),\hat p_4(\beta))_n+(z_4(\beta),\hat
\varphi_4(\beta))_n =-\int_\beta^{T}(z_4(t),f^{(0)}(t))_n dt+(\bar B_1z_1(T),f_1^{(0)})_{n-m}.
\end{equation}
From (\ref{j23}), (\ref{j24}), and (\ref{jj19}) and the expression
for  $\hat c,$ it follows
$$
\widehat{\widehat{(a,\varphi(s))}}=(a,\hat{\varphi}_2(s)).
$$
The theorem is proved.
\end{proof}

{\bf Corollary.} {\it Function $ \hat{\varphi}_2(s) $ can be taken
as an estimate of solution $\varphi(s)$ of initial BVP
(\ref{dy1}), (\ref{dy2}).}

As an example, consider the case  when a  vector-function
$y(t)=H(t)\varphi(t)+\xi(t)$ is observed on an interval $(0,T)$,
where a vector-function $\varphi(t)$ with values in  $\mathbb R^n$
is a solution to the BVP
\begin{equation}\label{za}
L_1\varphi=f(t), \quad \varphi(0)=f_0, \quad  \varphi(T)=f_1,
\end{equation}
and operator $L_1$ is defined by the relation
$$
L_1\varphi(t)=-\varphi''(t)+q(t)\varphi(t),
$$
where $q(t)$ is a positive definite $n\times n$-matrix whose
entries are continuous functions on $[0,T].$

Note that this problem has the unique classical solution if $f(t)$
is continuous on $[0,T]$ and the unique generalized solution if
$f(t)\in L_2(0,T).$

Assume that, as well as in the previous case,
$H(t)$ is an $l\times n$ matrix with the entries that are
continuous functions on $[\alpha, \beta ]$ and  $\xi(t)$ is a
random vector process with zero expectation $M\xi(t)$ and unknown
$l\times l$ correlation matrix $R(t,s)=M\xi(t)\xi^{T}(s)$. Assume
also that domains $V$ and $G$ are given in the form (\ref{d6}) and
(\ref{d7}) where matrices $Q_0,$ $Q_1$, and $Q_2(t)$ entering
(\ref{d7}) have dimensions $n\times n,$ $f_0^{(0)}=0,$
$f_1^{(0)}=0$, and $f^{(0)}(x)=0.$

Write equation (\ref{za}) as a first-order system by setting
$\varphi_1(t)=\varphi'(t),$ $\varphi_2(t)=\varphi(t)$ and
introducing a vector-function
$$
\tilde\varphi(t)=\left(\begin{array} {c}\varphi_1(t)\\
\varphi_2(t)\end{array}\right),\quad
\frac{d\tilde\varphi(t)}{dt}=\left(\begin{array}
{c}\frac{d\varphi_1(t)}{dt}\\
\frac{d\varphi_2(t)}{dt}\end{array}\right),
$$
$$
\tilde f (t)=\left(\begin{array} {c}f(t)\\
0\end{array}\right),\quad \tilde f_0 =\left(\begin{array} {c}
0\\f_0 \end{array}\right),\quad \tilde f_1 =\left(\begin{array}
{c} 0\\f_1 \end{array}\right)
$$
with $2n$ components, a vector $\tilde a=(0,a)$ with $2n$
components, a $2n\times 2n-$matrix
$$
\tilde A(t)=\left(\begin{array} {rl}O_{n,n}&-q(t)\\ -E_n
&O_{n,n}\end{array}\right)
$$
matrices $B_0=B_1=\bar B_0=\bar B_1=(O_{n,n},E_n),$ and $\hat
B_0=\hat B_1=(E_n,O_{n,n})$, and an operator
$$ \tilde L
\tilde\varphi(t)=-\frac{d\tilde\varphi(t)}{dt}+\tilde
A^{T}(t)\tilde\varphi(t).
$$
Then system (\ref{za}) can be written as
\begin{equation}
\label{za1} \tilde L
\tilde\varphi(t)=\frac{d\tilde\varphi(t)}{dt}+A\tilde\varphi(t) =
\tilde f(t),\quad B_{0}\tilde\varphi(0)=\tilde f_0, \quad
B_1\tilde\varphi(T)=\tilde f_1.
\end{equation}
Applying Theorems 1 and 2 and performing necessary transformations
in the resulting equations that are similar to (\ref{d8}) and
(\ref{tk}) (in terms of the designations introduced above) we
prove the following
\begin{pred}
The minimax estimate of expression $(a,\varphi(s))_n$ has the form
$$
\widehat{\widehat{(a,\varphi(s))_n}}=\int_0^s(\hat{u}_1(t),y(t))_ldt+
\int_s^T(\hat{u}_2(t),y(t))_ldt=(a,\hat{\varphi}_2(s))_n,
$$
where
$$
\sigma^2=(a,p_2(s))_n,
$$
$\hat{u}_1(t)=Q(t)H(t)p_2(t),$ $\hat{u}_2(t)=Q(t)H(t)p_3(t),$ and
vector-functions $\hat{\varphi}_2(t),$ $p_2$, and $p_3$ are
determined from the  solution to the equation systems
$$
L_1z_2(t)=-H^T(t)Q(t)H(t)p_2(t),\quad 0<t<s, \quad z_2(0)=0,
$$
$$
L_1z_3(t)=-H^T(t)Q(t)H(t)p_3(t), \quad s<t<T, \quad
z_2(s)-z_3(s)=a, \quad z_3(T)=0,
$$
$$
L_1p_2(t)=Q_2^{-1}(t)z_2(t), \quad 0<t<s, \quad p_2(0)=Q_0^{-1}
z_2(0),
$$
$$
L_1p_3(t)=Q_2^{-1}(t)z_3(t),\quad s<t<T,\quad p_2(s)=p_3(s),\quad
p_3(T)=-Q_1^{-1} z_3(T),
$$
$$
L_1\hat{p}_2(t)=H^T(t)Q(t)(y(t)-H(t)\hat{\varphi}_2(t)),\quad
0<t<s,\quad \hat{p}_2(0)=0,
$$
$$
L_1\hat{p}_3(t)=H^T(t)Q(t)(y(t)-H(t)\hat{\varphi}_3(t)),\quad
s<t<T,
$$
$$
\hat{p}_3(s)=\hat{p}_2(s),\quad  \hat{p}_3(T)=0,
$$
$$
L_1\hat{\varphi}_2(t)=Q_2^{-1}(t)\hat p_2(t),\quad 0<t<s,\quad
\hat{\varphi}_2(0)=Q_0^{-1} \hat{p}_2(0),
$$
$$
L_1\hat{\varphi}_3(t)=Q_2^{-1}(t)\hat p_3(t),\quad s<t<T, \quad
\hat{\varphi}_3(s)=\hat{\varphi}_2(s), \quad
\hat{\varphi}_3(T)=-Q_1^{-1} \hat{p}_3(T).
$$
\end{pred}

\section{
Minimax estimates of solutions subject to incomplete restrictions
on unknown parameters
}


%

Assume again that observations have form (\ref{d1}) and
undetermined parameters $f_0$, $f_1$ and $f(t)$  belong to the
domain
\begin{equation}
\label{n4-1}
G=\{ \tilde{F}=(\tilde f_0, \tilde f_1,
\tilde f):\int_{0}^{T}(Q_2(t)\tilde f(t),\tilde f(t))dt \leq 1 \},
\end{equation}
where $Q_2(t)$ is given in (\ref{d7}). 
The correlation function of process $\xi(t)$  belongs to domain
(\ref{d6}). 

Introduce the set
\begin{equation}
\label{n4-2} U=\{ u(\cdot) : \bar{B}_0 z_1(0,u)=0, \bar{B}_1
z_4(T,u)=0 \}
\end{equation}
here $u(t)= \left\{
\begin{array}{lc}
u_1(t),&\alpha<t<s,\\ u_2(t),&s<t<\beta,
\end{array}
\right.   $ where $z_i(t,u), i=\overline{1,4},$
 is the solution to BVP (\ref{d5}). 

\begin{predl}
\begin{equation*}
\sigma (u,c)= \left\{
\begin{array}{lc}
\infty,&u\notin U,\\ \sigma_1(u,c),&u\in U,
\end{array}
\right.
\end{equation*}
where
\begin{equation}
\label{n4-3} \sigma_1(u,c)=\int_0^T(Q_2^{-1}\tilde{z}(t,u),
\tilde{z}(t,u))dt + \int_{\alpha}^{\beta}(Q^{-1}(t)u(t),u(t))dt +
c^2 = J(u) + c^2.
\end{equation}
\end{predl}

This lemma can be proved using formula (\ref{xq}). 

\begin{predl}
$U$ is a convex closed set in the space $L_2(\alpha,\beta)$.
\end{predl}
\begin{proof}
The convexity of set $U$ is obvious. Let us prove that this set is
closed.

Note that functions $z_1(0,u)$ and $z_4(T,u)$ can be represented
as
\begin{equation}
\label{n4-4}
\begin{array}{c}
z_1(0,u)= a_1 + \int_{\alpha}^{\beta} \Phi_1(t)u(t) dt,\\
z_4(T,u)= a_2 + \int_{\alpha}^{\beta} \Phi_2(t)u(t) dt,
\end{array}
\end{equation}
where $\Phi_1(t)$ and $\Phi_2(t)$ are known matrix functions with
the elements from $L_2(\alpha,\beta)$ and $a_1$ and $a_2$ are
vectors. Expression (\ref{n4-4}) can be obtained if we introduce a
vector $z_0$ such that $z_1(0,u)=z_0.$  Then $z_1(\alpha,u)= \Phi
(\alpha,0)z_0, $ where $\Phi(t,\tau)$ is a solution to the
equation
\begin{equation*}
\frac{d \Phi (t, \tau)}{dt} = -A^{\ast}(t) \Phi(t,\tau), \Phi
(\tau, \tau) = E.
\end{equation*}
\begin{equation*}
z_2(s,u)= \Phi (s,\alpha)z_2(\alpha,u)=\Phi (s,\alpha) z_1
(\alpha, u)+ \int_{\alpha}^{s} \Phi (s,\tau) H^T (\tau)
u_1(\tau)d\tau =
\end{equation*}
\begin{equation*}
\Phi (s,0) z_0 + \int_{\alpha}^{s} \Phi (s,\tau) H^T (\tau)
u_1(\tau)d\tau.
\end{equation*}
Next,
\begin{equation*}
z_4 (T,u)= \Phi (T,0) z_0 + \int_{\alpha}^{\beta} \Phi (T,\tau)
H^T (\tau) u(\tau)d\tau + \Phi (T,s)a.
\end{equation*}

Since BVP (\ref{d5}) is uniquely solvable, 
there exists one and only one vector $z_0$ satisfying the
algebraic equation system
\begin{equation*}
\left\{
\begin{array}{l}
\bar{B}_0z_0=0,\\ B_1 \Phi (T,0)z_0 = - \int_{\alpha}^{\beta}
\Phi(t,\tau)H^T(\tau)u(\tau) d\tau -\Phi(T,s)a.
\end{array} \right.
\end{equation*}
Solving this system we determine $z_0$ in the form
\begin{equation*}
z_0=b + \int_{\alpha}^{\beta} \Phi_0(\tau)u(\tau) d\tau,
\end{equation*}
where $\Phi_0(\tau)$ is a known matrix function continuous on
$[\alpha,\beta]$ and $b$ is a known vector.
Taking into account this equality, we obtain expression (\ref{j1}). 
From these relationships, it follows that if a sequence $u_n (t)$
converges in $L_2(\alpha, \beta)$ to a function $u_0 (t),$ then
$$
\lim_{n\rightarrow \infty } \bar{B}_0 z_1(0,u_n)= \bar{B}_0 z_1
(0,u_0),
$$
$$
\lim_{n\rightarrow \infty } \bar{B}_1 z_4(0,u_n)= \bar{B}_1 z_4
(0,u_0),
$$
which proves that $U$ is a closed set.
\end{proof}

Assume now that $U$ is nonempty. Then the following statement is
valid.
\begin{pred}
\label{ntm4-1} There exists the unique minimax estimate of
expression $(a, \varphi (s))$ which can be represented in the form (\ref{ddd7}) 
at $\hat{c}=0,$ where vector-functions $p_2(t)$ and $p_3(t)$ solve
the equations
$$
L^{\ast}z_1(t)=0,\quad 0<t<\alpha, \quad \hat{B}_0z_1(0)=0,
$$
$$
L^{\ast}z_2(t)=-H^{T}(t)Q(t)H(t)p_2(t),\quad \alpha<t<s,\quad
z_2(\alpha)=z_1(\alpha),
$$
$$
L^{\ast}z_3(t)=-H^{T}(t)Q(t)H(t)p_3(t),\quad s<t<\beta,\quad
z_3(s)=z_2(s)-a,
$$
$$
L^{\ast}z_4(t)=0,\quad \beta<t<T,\quad z_4(\beta)=z_3(\beta),
$$
\begin{equation}\label{n4-5}
\hat{B}_1z_4(T)=0, \quad \bar{B}_0z_1(0), \quad \bar{B}_1z_4(T)=0,
\end{equation}
$$
Lp_1(t)=Q_2^{-1}(t)z_1(t),\quad  0<t<\alpha,
$$
$$
Lp_2(t)=Q_2^{-1}(t)z_2(t),\quad  \alpha<t<s, \quad  p_2(\alpha)=
p_1(\alpha),
$$
$$
Lp_3(t)=Q_2^{-1}(t)z_3(t), \quad  s<t<\beta,\quad  p_3(s)=p_2(s),
$$
$$
Lp_4(t)=Q_2^{-1}(t)z_4(t),\quad s<t< \beta.
$$
\end{pred}
\begin{proof}
Similarly to Theorem \ref{tm1} 
one can show that for $u \in U$ the following equality holds
$$
\sigma (u,c) = J(u)+c^2,
$$
where
$$
J(u) =  \int_0^{\alpha} (Q_2^{-1}(t)z_1(t,u), z_1(t,u))_n dt +
\int_{\alpha}^s (Q_2^{-1}(t)z_2(t,u), z_2(t,u))_n dt
$$
$$
+ \int_{s}^{\beta} (Q_2^{-1}(t) z_3(t,u), z_3(t,u))_n dt +
\int_{\beta}^T (Q_2^{-1}(t) z_4(t,u), z_4(t,u))_n dt
$$
$$
+\int^s_{\alpha}(Q^{-1}(t)u_1(t),u_1(t))dt+
\int_s^{\beta}(Q^{-1}(t)u_2(t),u_2(t))dt,
$$
and $z_i(t,u), i=\overline{1,4},$ are solutions to equations
(\ref{n4-5}) at $\bar{B}_0 z_1(0,u)=0 $ and $\bar{B}_1z_4(T,u)=0.$
$J(u)$ is a strictly convex lower semicontinuous functional on a
closed convex set $U$ and $\lim_{\| u \| \rightarrow \infty } J(u)
= \infty.$  Therefore there exists one and only one vector $
\hat{u}$ such that $\min_{u \in U} J(u)=J(\hat{u}).$ This vector
can be determined from the relationship
$$
\left.
\frac{d}{d\tau} J_{\mu} (\hat{u}_1+\tau v_1, \hat{u}_2+\tau
v_2)\right|_{\tau=0} \equiv 0, \quad \forall v=(v_1, v_2) \in H,
$$
where
$$
J_{\mu}(u)=J(u_1,u_2)+(\mu_1, \bar{B}_0z_1(0,u))+(\mu_2,
\bar{B}_1z_4(T,u)),
$$
$\mu=(\mu_1,\mu_2),$ $\mu_1\in \mathbb R^m,$ and $\mu_2\in \mathbb R^{n-m}$ are Lagrange multipliers.

Further analysis is similar to the proof of Theorem
\ref{tm1}. 
\end{proof}
Let vector-functions $\hat{p}_1(t),
\hat{p}_2(t),\hat{p}_3(t),\hat{p}_4(t), \hat{\varphi}_1(t),
\hat{\varphi}_2(t), \hat{\varphi}_3(t), \hat{\varphi}_4(t)$ be
solutions to the system
$$
L^{\ast}\hat{p}_1(t)=0,\quad  0<t<\alpha,
$$
$$
L^{\ast}\hat{p}_2(t)=H^T(t)Q(t)[y(t)-H(t)\hat{\varphi}_2(t)),\quad
\alpha<t<s,
$$
$$
L^{\ast}\hat{p}_3(t)= H^T(t)Q(t)[y(t)-H(t)\hat{\varphi}_3(t)),
\quad  s<t<\beta,
$$
$$
L^{\ast}\hat{p}_4(t)=0,\quad \beta <t< T,
$$
$$
L\hat{\varphi}_1(t)=Q_2^{-1}(t)\hat p_1(t), \quad 0<t<\alpha,
$$
\begin{equation}
\label{n4-6} L\hat{\varphi}_2(t)=Q_2^{-1}(t)\hat p_2(t), \quad \alpha
<t<s,
\end{equation}
$$
L\hat{\varphi}_3(t)=Q_2^{-1}(t)\hat p_3(t), \quad s<t<\beta,
$$
$$
L\hat{\varphi}_4(t)=Q_2^{-1}(t)\hat p_4(t), \quad \beta<t<T,
$$
$$
\hat{B}_0\hat{p}_1(0)=0, \quad
\hat{p}_2(\alpha)=\hat{p}_1(\alpha), \quad
\hat{p}_3(s)=\hat{p}_2(s), \quad
\hat{p}_4(\beta)=\hat{p}_3(\beta),
$$
$$
\hat{B}_1\hat{p}_1(T)=0, \quad \bar{B}_0\hat
p_1(0)=0, \quad
\bar{B}_1\hat p_4(T)=0,
$$
$$
\hat{\varphi}_2(\alpha)=\hat{\varphi}_1(\alpha), \quad
\hat{\varphi}_3(s)=\hat{\varphi}_2(s) \quad
\hat{\varphi}_4(\beta)=\hat{\varphi}_3(\beta).
$$
\begin{pred}
\label{ntm4-2} Assume that for any vector $a \in \mathbb{R}^n$ set
$U$ is nonempty. Then system (\ref{n4-6}) is uniquely solvable and
the equality
$$
(\widehat{\widehat{a,\varphi(s)}})=(a,\hat{\varphi}_2(s))
$$
holds
\end{pred}
\begin{proof}
Introduce functions $\hat{p}_i(t,v)$ as solutions to the BVP
$$
L^{\ast}\hat{p}_1(t,v)=0,\quad  0<t<\alpha,
$$
$$
L^{\ast}\hat{p}_2(t,v)=d(t)-H^T(t)v_1(t),\quad \alpha<t<s,
$$
$$
L^{\ast}\hat{p}_3(t,v)=d(t)-H^T(t)v_2(t), \quad s<t<\beta,
$$
$$
L^{\ast}\hat{p}_4(t,v)=0,\quad  \beta<t<T,
$$
$$
\hat{B}_0\hat{p}_1(0,v)=0, \quad \hat{B}_1\bar{p}_4(T,v)=0,
$$
$$
\hat{p}_2(\alpha,v)=\hat{p}_1(\alpha,v), \quad
\hat{p}_3(s,v)=\hat{p}_2(s,v), \quad
\hat{p}_4(\beta,v)=\hat{p}_3(\beta,v),
$$
where $d(t)=H^T(t)Q(t)y(t).$

Define a set
$$
U_1=\{ v: \bar{B}_0\hat p_1(0,v)=0, \bar{B}_1\hat p_4(T,v)=0  \}.
$$
Since $U$ is nonempty, the same is valid for $U_1$ for any vector
$a$. Similarly to the case of $U,$ one can show that  $U_1$ is a
convex closed set. Denote by $J_1(v)$ the functional of the form
$$
J_1(v) =  \int_0^{\alpha} (Q_2^{-1}(t)\hat {p}_1(t,v),
\hat{p}_1(t,v)) dt + \int_{\alpha}^s (Q_2^{-1}(t)\hat{p}_2(t,v),
\hat{p}_2(t,v)) dt
$$
$$
+ \int_{s}^{\beta} (Q_2^{-1}(t) \hat{p}_3(t,v), \hat{p}_3(t,v))dt
+ \int_{\beta}^T (Q_2^{-1}(t) \hat{p}_4(t,v), \hat{p}_4(t,v)) dt
$$
$$
+\int^s_{\alpha}(Q^{-1}(t)v_1(t), v_1(t))dt+
\int_s^{\beta}(Q^{-1}(t)v_2(t),v_2(t))dt.
$$
One can show, following Theorem \ref{ntm4-1}, that on set $U_1$
there is one and only one point of minimum of functional $J_1(v),$
namely,
$$
\hat{v}_1(t)=Q^{-1}(t)H(t)\hat{\varphi}_2(t), \quad
\hat{v}_2(t)=Q^{-1}(t)H(t)\hat{\varphi}_3(t),
$$
where functions $\hat{\varphi}_2(t) $ and $ \hat{\varphi}_3(t)$
are determined from system (\ref{n4-6}). The proof of the equality
$$
\widehat{\widehat{(a,\varphi(s))_n}}=(a,\hat{\varphi}_2(s))_n
$$
is similar to that in Theorem \ref{tm2}.
\end{proof}

\section{Elimination technique in minimax estimation problems}
\newtheorem{predl1}{Предложение}
Assume that a  vector-function
\begin{equation} \label{dd1}
y_1(t)=C_{11}\varphi (t)+C_{12}\varphi '(t)+\xi_1(t), \quad
y_2(t)=C_{21}\varphi (t)+C_{22}\varphi '(t)+\xi_2(t)
\end{equation}
is observed on interval $(0,1)$. Here $C=C_{ij}(t)$ is an $m\times
n$ matrix with the entries continuous on $[0,1]$; $\xi_i(t)$ are
 vector random processes continuous in the mean square sense
 and such that  $M\xi_i(t)=0,$
$i=1,2$; function $\varphi(t)$ is a solution to the BVP
\begin{equation} \label{dd2}
\varphi''(t)=A(t)\varphi(t)+B(t)f(t),\quad \varphi'(0)=0,\quad
\varphi(1)=0,
\end{equation}
where $A(t)$ and $B(t)$ are, respectively, $n\times n$ and
$n\times r$ matrices with the entries continuous on $[0,1]$ and
$A(t)$ is a symmetric nonnegative definite matrix; and $f(t)$ is a
square integrable vector-function on $(0,1)$.

Set $\varphi=\varphi_1$ and $\varphi_1'=\varphi_2$  and rewrite
(\ref{dd2}) as an equation system
\begin{equation} \label{dd3}
\varphi_1'=\varphi_2,\quad \varphi_2'=A(t)\varphi_1+B(t)f,\quad
\varphi_2(0)=\varphi_1(1)=0.
\end{equation}

Reduce the solution of BVP  (\ref{dd3}) to the solution of a
Cauchy problem using elimination.

We look for function $\varphi_2(t)$ in the form
$\varphi_2(t)=P(t)\varphi_1(t)+\psi(t),$ where matrix $P(t)$ and
vector-function $\psi(t)$ is chosen so that
$$
\varphi_2'(t)=A(t)\varphi_1(t)+B(t)f(t).
$$
We have
$$
\varphi_2'(t)=P'(t)\varphi_1(t)+P(t)\varphi_1'(t)+\psi'(t)=
P'(t)\varphi_1(t)+P(t)\varphi_2(t)+\psi'(t)=
$$
$$
=P'(t)\varphi_1(t)+P(t)[P(t)\varphi_1(t)+\psi(t)]+\psi'(t)=
A(t)\varphi_1(t)+B(t)f(t);
$$
therefore, functions $P(t)$ and $\psi(t)$ must satisfy the
equations
\begin{equation} \label{dd4}
\!P'(t)+P^2(t)=A(t),\,\, P(0)=0, \,\,
\psi'(t)+P(t)\psi(t)=B(t)f(t),\,\, \psi(0)=0.
\end{equation}

Thus we have reduced BVP (\ref{dd2}) to the Cauchy problem for
functions $P(t),$ $\psi(t),$ and $\varphi_1(t).$

Note that for function $P(t)$ a Riccati equation is obtained which
is uniquely solvable on any finite interval; in addition, there is
one and only one function $\psi(t)$ which is absolutely continuous
and satisfies the corresponding equation almost everywhere.

Using the notations introduced above, we can write the expressions
for functions $y_1(t)$ and $y_2(t)$ as
\begin{gather*}
y_1(t)=[C_{11}(t)+C_{12}(t)P(t)]\varphi_1(t)+
C_{12}(t)\psi(t)+\xi_1(t),   \\
y_2(t)=[C_{21}(t)+C_{22}(t)P(t)]\varphi_1(t)+
C_{22}(t)\psi(t)+\xi_2(t),
\end{gather*}
or
\begin{equation} \label{dd5}
y(t)=H(t)x(t)+\xi(t),
\end{equation}
where
\begin{gather*}
H(t)=\left(H_{ij}(t)\right)_{i,j=1,2},\quad
H_{11}=C_{11}+C_{12}P,\\ H_{12}=C_{12},\quad
H_{21}=C_{21}+C_{22}P,\quad H_{22}=C_{22},       \\ x(t)=\left(
\begin{array}{c}
\varphi_1(t) \\ \psi(t)
\end{array}
\right),
\quad \xi(t)=\left(
\begin{array}{c}
\xi_1(t) \\ \xi_2(t)
\end{array}
\right),
\end{gather*}
and function $x(t)$ is a solution to the equation
\begin{equation} \label{dd6}
\frac{dx}{dt}=A_1(t)x(t)+B_1(t)f(t),\quad  \varphi_1(1)=0, \quad
\psi(0)=0,
\end{equation}
where
\begin{gather*}
A_1(t)=\left( \begin{array}{cc} P(t)& E \\ 0& -P(t)\end{array}
\right), \quad B_1(t)=\left(
\begin{array}{c} 0 \\ B(t)
\end{array}
\right).
\end{gather*}

Denote by  $V$ a class of random processes $\tilde \xi(t)$ whose
correlation matrices $\tilde R(t,s)=M\tilde \xi(t)\tilde \xi^T(s)$
satisfy the inequality
$$
\int\limits_0^1q_1^2(t)Sp\,\tilde R(t,t)dt\leq 1,
$$
here $q_1(t)$ is a continuous function on $[0,T]$ such that
$|q_1(t)|\geq \rm{q}>0$ where $\rm{q}$ is a certain number. Assume
in what follows that correlation matrix $R(t,s)$ of process
$\xi(t)$ belongs to $V.$

Assume also that function $f(t)$ belongs to a set
\begin{equation} \label{dd8}
G=\left\{ \tilde f: \, \int\limits_0^1(Q\tilde f,\tilde f)dt\leq 1 \right\} ,
\end{equation}
where $Q$ is a positive definite matrix.

Our problem is to find a minimax estimate of expression
$(a_1,\varphi_1(s))+(a_2,\varphi_2(s))$ using observations
(\ref{dd1}). Since
$(a_1,\varphi_1(s))+(a_2,\varphi_2(s))=(a_1+Pa_2,\varphi_1(s))+(a_2,\psi(s))
=(b,x(s))$ where $b=\left(
\begin{array}{c}
a_1+Pa_2 \\ a_2
\end{array}
\right),$ there is a one-to-one correspondence that couples the
estimates.

We look for an estimate\footnote{In this section $(\cdot,\cdot)$
will denote the inner product in the corresponding Euclidean
space} of expression $(b,x(s))$ in the class of estimates  of the
form
\begin{equation} \label{dd9}
(\widehat{b,x(s)})= \int\limits_0^s(u_1(t),y(t))dt+
\int\limits_s^1(u_2(t),y(t))dt+d
\end{equation}
linear with respect to observations where $u_1(t) $ and $ u_2(t)$
are square integrable functions on $(0,s)$ and $(s,1)$,
respectively, and $d \in \mathbb R^1.$

As before, we look for minimax estimate $(\widehat{\widehat{b,
x(s)}})=\int\limits_0^s(\hat u_1(t),\tilde
y(t))dt+ \int\limits_s^1(\hat u_2(t),\tilde y(t))dt+\hat d$ of inner product $(b,x(s))$
in the class of linear estimates of the form
$(\widehat{b,x(s)})= \int\limits_0^s(u_1(t),
y(t))dt+ \int\limits_s^1(u_2(t),y(t))dt+d.$
For this estimate
vector-function $\hat u(t)=(\hat{u}_1(t),\hat{u}_2(t))$ and constant
$\hat{d}$
are determined from the condition
\begin{equation} \label{dd10}
\inf_{u_1\in L^2(0,s),u_2\in L^2(s,1),d\in \mathbb R}\sup_{\tilde \xi \in V, \tilde f \in G}M[(b,\tilde x(s))-(\widehat{b,\tilde x(s)})]^2=
\sup_{\tilde \xi \in V,\tilde f \in G}M[(b,\tilde x(s))-(\widehat{\widehat{b,\tilde x(s)}})]^2,
\end{equation}
where $\tilde x$ is a solution to problem \eqref{dd2} at $f=\tilde
f,$ $(\widehat{b,\tilde x(s)})= \int\limits_0^s(u_1(t),\tilde
y(t))dt+ \int\limits_s^1(u_2(t),\tilde y(t))dt+d,$ $\tilde
y(t)=H(t)\tilde x(t)+\tilde \xi(t).$

Introduce functions $z_1(s) $ and $ z_2(s)$ as solutions to the
equations
$$
z_1'=-A_1^Tz_1+H^Tu_1, \quad z_{11}(0)=0, \quad 0<t<s,
$$
\begin{equation} \label{dd11}
z_2'=-A_1^Tz_2+H^Tu_2, \quad z_{22}(1)=0, \quad s<t<1,
\end{equation}
$$
-z_2(s)+z_1(s)=b.
$$
\begin{predllll} Let $z_1(s)$ and $z_2(s)$ be solutions to the BVP for equations
 (\ref{dd11}); then the following equality holds
\begin{equation*}
\sup_{\tilde \xi \in V, \tilde f \in G}M[(b,\tilde x(s))-(\widehat{b,\tilde x(s)})]^2=
\int\limits_0^s(Q_1z_1,z_1)dt+\int\limits_s^1(Q_1z_2,z_2)dt+
\end{equation*}
\begin{equation} \label{dd12}
+\int\limits_0^sq_1^{-2}(u_1,u_1)dt+\int\limits_s^1q_1(u_2,u_2)dt+
d^2,
\end{equation}
where
$$
Q_1=B_1Q^{-1}B_1^T.
$$
\end{predllll}

The proof of this statement is similar to the proof of the
corresponding assertion from Section 2.1.

{\bf Remark.} The BVP is uniquely solvable in the class of
absolutely continuous  functions.

Indeed, writing the equations for the components of vectors $z_1$
and $z_2$ we obtain an equation system
$$
z_{11}'=-Pz_{11}+(H^Tu_1)_1,   \quad z_{11}(0)=0,
$$
\begin{equation}  \label{dd13}
z_{12}'=-z_{11}+Pz_{12}+(H^Tu_1)_1,   \quad
-z_{21}(s)+z_{11}(s)=b_1,
\end{equation}
$$
z_{21}'=-Pz_{21}+(H^Tu_2)_1,  \quad -z_{22}(s)+z_{12}(s)=b_2,
$$
$$
z_{22}'=-z_{21}+Pz_{22}+(H^Tu_2)_2,   \quad z_{22}(1)=0.
$$
For  function $z_{11}$ we have a Cauchy problem; solving this
problem we find $z_{21}(s)$ and, consequently, $z_{22}(s).$
Solving the Cauchy problem for $z_{21}(t),$ we obtain this
function for any $t\in (s,1).$
Function $z_{12}(t)$ can be
determined in a similar manner.
\begin{predllll}
Let functions $u_1(t)$ and $u_2(t)$ have the form
$u_1(t)=p_1(t)H(t)z_1(t)$ and $u_2(t)=p_2(t)H(t)z_2(t),$ where
$p_1(t)$ and $p_2(t)$ are continuous, respectively, on $[0,s]$ and
$[s,1].$ Then the following representation is valid
\begin{equation}  \label{dd14}
\int\limits_0^s(u_1(t),y(t))dt+\int\limits_s^1(u_2(t),y(t))dt=
(b,\hat{x}_1(s)),
\end{equation}
where $\hat{x}_1(s)$ is determined from the solution to the
equation system
$$
\hat{x}_1'(t)=A_1\hat{x}_1+p_1(t)H^T(y(t)-H\hat{x}_1(t)),
$$
\begin{equation}  \label{dd15}
\hat{x}_2'(t)=A_1\hat{x}_2+p_2(t)H^T(y(t)-H\hat{x}_2(t)),
\end{equation}
$$
\hat{x}_1(s)=\hat{x}_2(s), \quad
\hat{x}_{12}(0)=\hat{x}_{21}(1)=0.
$$
\end{predllll}
\begin{proof} Since $u_2(t)=p_2(t)H(t)z_2(t),$ we have
$$
\int\limits_s^1(u_2(t),y(t))dt=
\int\limits_s^1(z_2(t),p_2(t)H^T(t)y(t))dt.
$$
Let $\hat{x}_2(s)$ be a solution determined from system
(\ref{dd15}). Multiplying both sides of equations (\ref{dd15}) by
$z_2(t)$ and integrating from $s$ to $1,$ we obtain
$$
\int\limits_s^1(z_2,p_2H^Ty)dt=\int\limits_s^1(\hat{x}_2',z_2)dt-
\int\limits_s^1(A_1\hat{x}_2,z_2)dt+
\int\limits_s^1(p_2H^TH\hat{x}_2,z_2)dt,
$$
however,
$$
\int\limits_s^1(\hat{x}_2',z_2)dt=
-\int\limits_s^1(\hat{x}_2,z_2')dt+(\hat{x}_2(1),z_2(1))-
(\hat{x}_2(s),z_2(s))=
$$
$$
=\int\limits_s^1(\hat{x}_2,A_1^Tz_2)dt-
\int\limits_s^1(\hat{x}_2,p_2H^THz_2)dt.
$$
Thus
$$
\int\limits_s^1(z_2,p_2H^Ty)dt=-(\hat{x}_2(s),z_2(s)).
$$
In a similar way, we can show that
$$
\int\limits_0^s(u_1,y)dt=\int\limits_0^s(z_1,p_1H^Ty)dt=
(\hat{x}_1(s),z_1(s)).
$$
The desired representation follows now from the latter
relationship and the equalities $\hat{x}_1(s)=\hat{x}_2(s),$
$z_1(s)-z_2(s)=b$.
\end{proof}

Introduce functions $p_1(t)$, $p_2(t)$, $\hat{x}_1(t)$, and
$\hat{x}_1(t)$ as solutions to the initial value problems
$$
z_1'=-A_1^Tz_1+H^Tq_1^2Hp_1, \quad z_{11}(0)=0,
$$
$$
z_2'=-A_1^Tz_2+H^Tq_1^2Hp_2, \quad z_{22}(1)=0,
$$
\begin{equation} \label{dd17}
z_1(s)-z_2(s)=b,
\end{equation}
$$
p_1'=A_1p_1+Q_1z_1, \quad p_{12}(0)=0,
$$
$$
p_2'=A_1p_2+Q_1z_2, \quad p_{21}(1)=0,
$$
$$
p_1(s)=p_2(s),
$$
$$
\hat{x}_1'=A_1\hat{x}_1+Q_1\hat{p}_1, \quad \hat{x}_{12}(0)=0,
$$
$$
\hat{x}_2'=A_1\hat{x}_2+Q_1\hat{p}_2, \quad \hat{x}_{21}(1)=0,
$$
\begin{equation} \label{dd18}
\hat{x}_1(s)=\hat{x}_2(s),
\end{equation}
$$
-\hat{p}_1'=A_1^T\hat{p}_1+H^Tq_1^2(y-H\hat{x}_1), \quad
\hat{p}_{11}(0)=0,
$$
$$
-\hat{p}_2'=A_1^T\hat{p}_2+H^Tq_1^2(y-H\hat{x}_2), \quad
\hat{p}_{22}(1)=0,
$$
$$
\hat{p}_1(s)=\hat{p}_2(s).
$$
The following statement is valid.
\begin{predllll}
Let the set $G$ have form (\ref{dd8}). The the minimax estimate
admits the representation
\begin{equation} \label{dd16}
(\widehat{\widehat{b,x(s)}})=\int\limits_0^s(q_1^2Hp_1,y)dt+
\int\limits_s^1(q_1^2Hp_2,y)dt=(b,\hat{x}_1(s)),
\end{equation}
where functions $p_1$, $p_2$, and $\hat{x}_1$ are determined from
the solution to equation systems (\ref{dd17}), (\ref{dd18}).
\end{predllll}

The proof of this statement is similar to the proof of the
corresponding assertion from Section 2.1.

Write equations (\ref{dd18}) in a detailed form, row by row,
$$
\hat{x}_{11}'=P\hat{x}_{11}+\hat{x}_{11}, \quad \hat{x}_{12}(0)=0,
$$
$$
\hat{x}_{12}'=-P\hat{x}_{12}+\tilde Q\hat{p}_{12}, \quad
\hat{x}_{21}(1)=0,
$$
$$
\hat{x}_{21}'=P\hat{x}_{21}+\hat{x}_{22}, \quad
\hat{x}_{11}(s)=\hat{x}_{21}(s),
$$
$$
\hat{x}_{22}'=-P\hat{x}_{22}+\tilde Q\hat{p}_{22}, \quad
\hat{x}_{12}(s)=\hat{x}_{22}(s),
$$
where $\tilde Q=BQ_1B^T.$ Let us go back to equation (\ref{dd6}).
Since the equation for function $\psi(t)$ has the form
$$
\psi'(t)=-P(t)\psi(t)+B(t)f(t), \quad \psi(0)=0,
$$
the Cauchy formula yields
$$
\psi(1)=\int\limits_0^1\Phi(1,t)B(t)f(t)dt=\mathcal{D}f,
$$
where $\Phi(s,t)$ is a solution to the equation
$$
\frac{\partial \Phi(s,t)}{\partial s}=P(s)\Phi(s,t),
$$
$$
\Phi(t,t)=E.
$$

Below, we will study estimates of solutions to equation
(\ref{dd6}) subject to the conditions $\varphi_1(1)=0,$
$\psi(1)=\mathcal{D}f.$ Let the function $f(t)$ belong to a
bounded subset $G$ of $L_2(0,1)$.

Introduce functions $z_1(t;u)$ and $z_2(t;u)$ as solutions to the
equations
$$
z_1'(\cdot;u)=-A_1^Tz_1(\cdot;u)+H^Tu_1, \quad z_1(0)=0,
$$
\begin{equation} \label{dd19}
z_2'(\cdot;u)=-A_1^Tz_2(\cdot;u)+H^Tu_2, \quad z_1(s;u)-z_2(s;u)=b.
\end{equation}

\begin{predllll}
Let $z_1(t;u)$ and $z_2(t;u)$ be solutions to the equation system
(\ref{dd19}). Then
$$
\sup_{\tilde \xi \in V, \tilde f \in G}M((b,\tilde x(s))-(\widehat{b,\tilde x(s)}))^2=
\sup_{\tilde f \in G}\Biggl(\Biggr.\int_0^s\left(z_1(t;u)+\Phi^T(1,t)z_1(1;u),
B(t)\tilde f(t)\right)dt+
$$
$$
+\int_s^1\left(z_2(t;u)+\Phi^T(1,t)z_2(1;u),B(t)\tilde f(t)\right)dt-d\Biggl.\Biggr)^2+ \int_0^s
q_1^{-2}(t)(u_1(t),u_1(t))dt+
$$
$$
+\int_s^1q_1^{-2}(t)(u_2(t),u_2(t))dt.
$$
\end{predllll}

This statement can be easily proved using the methods set forth in
Section 2.1.

We will call minimax estimate $(\widehat{\widehat{b,x(s)}})$
$\mathcal{U}$ optimal if it has the form
$$
(\widehat{\widehat{b,x(s)}})=\int_0^s\left(\hat{u}_1(t),y(t)\right)dt
+\int_s^1\left(\hat{u}_2(t),y(t)\right)dt+\hat{d},
$$
where functions $\hat{u}_1$ and $\hat{u}_2$ and number $\hat{d}$
are determined from the condition
$$
\inf_{(u_1,u_2)\in
\mathcal{U},d\in \mathbb R}\sup_{\tilde \xi \in V, \tilde f \in G}M\left((b,\tilde x(s))-(\widehat{b,\tilde x(s)})\right)^2=
\sup_{\tilde \xi \in V, \tilde f \in G}M\left((b,\tilde x(s))-(\widehat{\widehat{b,\tilde x(s)}})\right)^2
$$
and $\mathcal{U}$  is a subset of $L_2(0,1).$

Let the set $\mathcal{U}$ be given by
$\mathcal{U}=\{(u_1,u_2):z_2(1;u)=0\},$ and $G$ is determined by
formula  (\ref{dd8}).
\begin{predllll}
Assume that set $\mathcal{U}$ is not empty. Then $\hat{d}=0,$
$\hat{u}_1(t)=q_1(t)H(t)p_1(t),$ and
$\hat{u}_2(t)=q_1(t)H(t)p_2(t),$ where $p_1(t)$ and $p_2(t)$ are
determined from the  solution to the equation system
$$
z_1'=-A_1^Tz_1+H^T\hat{u}_1, \quad z_1(0)=0,
$$
$$z_2'=-A_1^Tz_2+H^T\hat{u}_2, \quad z_1(s)-z_2(s)=b, \quad z_2(1)=0,
$$
$$
p_1'=A_1p_1+Q_1z_1, \quad 0<t<s,
$$
\begin{equation}
\label{hgf}
p_2'=A_1p_2+Q_1z_2, \quad s<t<1,\quad p_1(s)=p_2(s).
\end{equation}
\end{predllll}
\begin{proof} If set $U$ is not empty, then from Proposition 4, it follows that
for $(u_1,u_2)\in U$,
\begin{gather*}
\sup_{\tilde \xi \in V, \tilde f \in G}M((b,\tilde x(s))-(\widehat{b,\tilde x(s)}))^2=\\
\sup_{\tilde f\in G}\left[\int_0^s(z_1(t;u),B(t)\tilde f(t))dt+\int_s^1(z_2(t;u),B(t)\tilde f(t))dt-d
\right]^2+
\\
+\int_0^sq_1^{-2}(t)(u_1(t),u_1(t))dt+
\int_s^1q_1^{-2}(t)(u_2(t),u_2(t))dt=J(u,d).
\end{gather*}
Taking into account the definition of set $G$, we obtain
\begin{gather*}
J(u,d)=\int_0^s(Q_1z_1(\cdot;u),z_1(\cdot;u))dt
+\int_s^1(Q_1z_2(\cdot;u),z_2(\cdot;u))dt +\\
+\int_0^sq_1^{-2}(t)(u_1(t),u_1(t))dt+
\int_s^1q_1^{-2}(t)(u_2(t),u_2(t))dt+d^2.
\end{gather*}
This expression yields
$$
\inf_{(u_1,u_2)\in U,d\in\mathbb R}J(u,d)=\inf_{(u_1,u_2)\in U}J(u,0).
$$
It is easy to see that there exists a unique function $\hat u(t)=(\hat u_1(t), \hat u_2(t))$ such that $$J(\hat u,0)=\inf_{u=(u_1,u_2)\in U}J(u,0).$$
Set $z_1(t):=z_1(t;\hat u)$ and $z_2(t):=z_2(t;\hat u).$
Finding conditional extremum of functional  $J(u,0)$ we obtain
that for the function $\hat u(t)$ at which the infimum is attained, the
following equalities hold
$$
\hat{u}_1(t)=q_1(t)H(t)p_1(t), \quad
\hat{u}_2(t)=q_1(t)H(t)p_2(t),
$$
where $p_1(t)$ and $p_2(t)$
are
determined from the solution to the equation system
(\ref{hgf}).
The theorem is proved.
\end{proof}

Introduce functions $\hat{p}_1(t),$ $ \hat{p}_2(t),$
$\hat{x}_1(t),$ and $\hat{x}_2(t)$ as solutions to the equation
system
$$
-\hat{p}_1'=A_1^T\hat{p}_1+q_1H^T(y(t)-H\hat{x}_1(t)), \quad
\hat{p}_1(0)=0,
$$
$$
-\hat{p}_2'=A_1^T\hat{p}_2+q_1H^T(y(t)-H\hat{x}_2(t)), \quad
\hat{p}_2(1)=0, \quad \hat{p}_1(s)=\hat{p}_2(s),
$$
$$
\hat{x}_1'=A_1\hat{x}_1+Q_1\hat{p}_1, \quad 0<t<s,
$$
$$
\hat{x}_2'=A_1\hat{x}_2+Q_1\hat{p}_2, \quad s<t<1,\quad
\hat{x}_1(s)=\hat{x}_2(s).
$$
\begin{predllll}
The minimax estimate admits the representation
$(\widehat{\widehat{b,x(s)}})=(b,\hat{x}_1(s)).$
\end{predllll}
\begin{proof} Since $\hat{u}_1(s)=q_1(t)H(t)p_1(t),$ we have
$$
\int\limits_0^s(\hat{u}_1,y)dt=\int\limits_0^s(q_1Hp_1,y)dt=
\int\limits_0^s(p_1,q_1H^Ty)dt=
$$
$$
=-\int\limits_0^s(p_1,\hat{p}_1')dt-
\int\limits_0^s(p_1,A_1^T\hat{p}_1)dt+
\int\limits_0^s(p_1,q_1H^TH\hat{x}_1)dt=
$$
$$
=\int\limits_0^s(p_1',\hat{p}_1)dt-
\int\limits_0^s(A_1^Tp_1,\hat{p}_1)dt-(p_1(s),\hat{p}_1(s))+
\int\limits_0^s(p_1,q_1H^TH\hat{x}_1)dt=
$$
$$
=\int\limits_0^s(z_1,Q_1\hat{p}_1)dt-
(p_1(s),\hat{p}_1(s))+\int\limits_0^s(p_1,q_1H^TH\hat{x}_1)dt,
$$
$$
\int\limits_0^s(z_1,Q_1\hat{p}_1)dt=
\int\limits_0^s(\hat{x}_1'-A_1\hat{x}_1,z_1)dt=
$$
$$
=-\int\limits_0^s(\hat{x}_1,z_1'+A_1^Tz_1)dt+
(\hat{x}_1(s),z_1(s))-\int\limits_0^s(p_1,q_1H^TH\hat{x}_1)dt.
$$
The latter relationships yield
$\int\limits_0^s(\hat{u}_1,y)dt=(\hat{x}_1(s),z_1(s)).$

In a similar manner, we can show that
$$
\int\limits_s^1(\hat{u}_2,y)dt=-(\hat{x}_2(s),z_2(s))=
-(\hat{x}_1(s),z_2(s)).
$$
Therefore,
$$
\int\limits_0^s(\hat{u}_1,y)dt+
\int\limits_s^1(\hat{u}_2,y)dt=(b,\hat{x}_1(s)),
$$
The proposition is proved.
\end{proof}


\section{
Minimax estimation of the solutions to the boundary value  problems
from point observations
}


In this chapter we study minimax estimation problems in the case
of point observations and propose constructive methods for
obtaining minimax estimates.

Let $t_i,$ $i=\overline{1,N}$ be a given system of points on an
interval $(0,T).$ Set $t_0=0$ and $t_{N+1}=T.$ The problem is to
estimate the expression
\begin{equation}\label{2p3v}
(a,\varphi(s))_n=\sum_{i=1}^{n}a_i\varphi_i(s),
\end{equation}
from observations of the form
\begin{equation}\label{1p3v}
y_i=\varphi(t_i)+\xi_i,\,\,i=\overline{1,N},
\end{equation}
over the state  $\varphi$ of a system described by BVP
(\ref{dy1}), (\ref{dy2}) in the class of estimates
\begin{equation}\label{3p3v}
\widehat{(a,\varphi(s))_n}=\sum_{i=1}^N (u_i, y_i)_n+c
\end{equation}
linear with respect to observations (\ref{1p3v}); here $s\in
(t_{i_0-1}, t_{i_0}),$ $i_0\in \{1, \ldots, N+1\}.$ The
assumptions are as follows\footnote{Set $G$ is defined by
\eqref{d7}}
$F=(f_0,f_1,f(\cdot))\in G,$
$ \xi:=(\xi_{1},\ldots,\xi_{N})\in V,$
where $\xi_i$ are errors in estimations  \eqref{1p3v} that are
realizations of random vectors $\xi_i=\xi_i(\omega)\in \mathbb
R^n$ and $V$ denotes the set of random elements $\tilde
\xi=(\tilde \xi_{1},\ldots,\xi_{N})$ whose components $\xi_i$ has
integrable second moments  $M\tilde \xi_i^2$, zero means
$M\tilde \xi_i=0,$, and correlation matrices
 $\tilde R_i=M\tilde\xi_i\tilde\xi_i^{T}$ satisfying the
 condition
\begin{equation}\label{d6llv}
\sum_{i=1}^NSp\,[\tilde Q_i\tilde R_i]\leq
1,
\end{equation}
where $\tilde Q_i$ are positive definite $n\times n$ matrices,
$\mbox{Sp}\,B$ denotes the trace of the matrix
$B=\{b_{ij}\}_{i,j=1}^l,$ i.e., the quantity $\sum_{i=1}^lb_{ii},$
$u_i\in \mathbb R^n,$ and $c\in \mathbb R.$

Set $ u:=(u_1,\ldots,u_N)\in \mathbb R^{N\times n}. $

{\bf Definition.} {\it The estimate
$$
\widehat{\widehat
{(a,\varphi(s))_n}}=\sum_{i=1}^N(\hat u_i,
y_i)_n+\hat c,
$$
in which vectors $\hat u_i,$ and a number $\hat c$
are determined from the condition
\begin{equation} \label{11fv}
\sup_{\tilde F \in V,\, \tilde \xi
\in G}M|(a,\tilde\varphi(s))_n-\widehat
{(a,\tilde\varphi(s))_n}|^2 \to \inf_{u\in \mathbb R^{N\times n},\,c\in \mathbb R} ,
\end{equation}
where
\begin{equation} \label{llxfv}
\widehat
{(a,\tilde\varphi(s))_n}=\sum_{i=1}^N(u_i,
\tilde y_i)_n +c,
\end{equation}
\begin{equation} \label{4iufv}
\tilde y_{i}=\tilde \varphi(t_i)+ \tilde \eta_{i},\quad i= \overline{1,N},
\end{equation}
and $\tilde\varphi(t)$ is the solution to the BVP (\ref{dy1}), (\ref{dy2}) at  $f=\tilde f,$ $f_0=\tilde f_0,$ and $f_1=\tilde f_1,$
will be called the minimax estimate of expression (\ref{2p3v}).

The quantity
\begin{equation} \label{12dfv}
\sigma:=\{\sup_{\tilde F \in G,\,
\tilde \xi \in V}M|(a,\varphi(s))_n-\widehat{\widehat
{(a,\varphi(s))_n}}|^2\}^{1/2}
\end{equation}
will be called the error of the minimax estimation of $(a,\varphi(s))_n.$}

Let again $t_0=0$, $t_{N+1}=T$, and $s\in (t_{i_0-1},t_{i_0}),$
$i_0=\overline{1,N+1}.$ For any fixed $u:=(u_1,\ldots, u_N)\in
\mathbb R^{N\times n}$ introduce vector-functions $z_1(\cdot;u)\in
H^1(t_0,t_1)^n,$ $\ldots,$ $z_{i_0-1}(\cdot;u)\in
H^1(t_{i_0-1},t_{i_0})^n,$ $z_{i_0}^{(1)}(\cdot;u)\in
H^1(t_{i_0},s)^n,$ $z_{i_0}^{(2)}(\cdot;u)\in H^1(s,t_{i_0})^n,$
$z_{i_0+1}(\cdot;u)\in H^1(t_{i_0},t_{i_0+1})^n,$ $\ldots,$
$z_{N+1}(\cdot;u)\in H^1(t_{i_N},t_{i_N+1})^n,$  as solution to
the following BVP
\begin{gather}
L^* z_1(t;u)=0,\quad 0<t<t_1, \quad \hat B_0z_1(0;u)=0,\notag\\
L^* z_2(t;u)=0,\quad t_1<t<t_2, \quad z_2(t_1;u)=z_1(t_1;u)+u_1,\notag\\
\ldots\ldots\ldots\ldots\ldots\ldots\ldots\ldots\ldots
\ldots\ldots\ldots\ldots\ldots\ldots\ldots\ldots\notag\\
L^* z_{i_0-1}(t;u)=0,\quad t_{i_0-2}<t<t_{i_0-1}, \quad z_{i_0-1}(t_{i_0-2};u)=z_{i_0-2}(t_{i_0-2};u)+u_{i_0-2},\notag\\
L^* z_{i_0}^{(1)}(t;u)=0,\quad t_{i_0-1}<t<s, \quad z_{i_0}^{(1)}(t_{i_0-1};u)=z_{i_0-1}(t_{i_0-1};u)+u_{i_0-1},\notag\\
L^* z_{i_0}^{(2)}(t;u)=0,\quad s<t<t_{i_0}, \quad z_{i_0}^{(2)}(s;u)=z_{i_0}^{(1)}(s;u)-a, \label{sopr1v}\\
L^* z_{i_0+1}(t;u)=0,\quad t_{i_0}<t<t_{i_0+1}, \quad z_{i_0+1}(t_{i_0};u)=z_{i_0}^{(2)}(t_{i_0};u)+u_{i_0},\notag\\
\ldots\ldots\ldots\ldots\ldots\ldots\ldots\ldots\ldots
\ldots\ldots\ldots\ldots\ldots\ldots\ldots\ldots\notag\\
L^* z_N(t;u)=0,\quad t_{N-1}<t<t_{N}, \quad z_N(t_{N-1};u)=z_{N-1}(t_{N-1};u)+u_{N-1},\notag\\
L^* z_{N+1}(t;u)=0,\quad t_{N}<t<t_{N+1}, \quad z_{N+1}(t_{N};u)=z_{N}(t_{N};u)+u_{N},\quad \hat B_1z_{N+1}(T;u)=0.\notag
\end{gather}

Using a reasoning similar to the proof of the unique solvability
of BVP \eqref{d33}, one can show that problem \eqref{sopr1v} is
uniquely solvable.
\begin{predl}
Finding the minimax estimate of functional
$(a,\varphi(s))$ is equivalent to the problem of optimal control of the system described by BVP (\ref{sopr1v})
with the cost function
\begin{multline*}
I(u)=\sum_{i=1, i\neq i_0}^{N+1}\int_{t_{i-1}}^{t_{i}}(Q_2^{-1}(t)z_{i}(t;u),z_{i}(t;u))_ndt\\
+\int_{t_{i_0-1}}^{s}(Q_2^{-1}(t)z_{i_0}^{(1)}(t;u),z_{i_0}^{(1)}(t;u))_ndt
+\int_{s}^{t_{i_0}}(Q_2^{-1}(t)z_{i_0}^{(2)}(t;u),z_{i_0}^{(2)}(t;u))_ndt
\end{multline*}
\begin{equation}\label{N4llv}
+(Q_0^{-1}\bar{B}_0 z_1(0;u),\bar{B}_0 z_1(0;u))_m \!+\!
(Q_1^{-1}\bar{B}_1z_4(T;u),\bar{B}_1z_4(T;u))_{n-m}
\!+\!\sum_{i=1}^N(\tilde Q_i^{-1}u_i,u_i)_ndt
\to \!\!\inf_{u\in \mathbb R^{N\times n}}.
\end{equation}
\end{predl}
\begin{proof}
It is easy to see that the following equalities hold
\begin{equation}\label{sopr10v}
(z_{i}(t_{i-1};u),\tilde\varphi(t_{i-1}))_n-(z_{i}(t_{i};u),\tilde\varphi(t_{i}))_n
+\int_{t_{i-1}}^{t_{i}}(z_{i}(t;u),\tilde f(t))_ndt=0,
\end{equation}
$$
i=1,\ldots,i_0-1,i_0+1,\ldots,N+1,
$$
\begin{equation}\label{sopr11v}
(z_{i_0}^{(1)}(t_{i_0-1};u),\tilde\varphi(t_{i_0-1}))_n-
(z_{i_0}^{(1)}(s;u),\tilde\varphi(s))_n
+\int_{t_{i_0-1}}^{s}(z_{i_0}^{(1)}(t;u),\tilde f(t))_ndt=0,
\end{equation}
\begin{equation}\label{sopr12v}
(z_{i_0}^{(2)}(s;u),\tilde\varphi(s))_n-
(z_{i_0}^{(2)}(t_{i_0};u),\tilde\varphi(t_{i_0}))_n
+\int_{s}^{t_{i_0}}(z_{i_0}^{(2)}(t;u),\tilde f(t))_ndt=0.
\end{equation}
Using \eqref{llxfv}, \eqref{4iufv}, and \eqref{sopr1v}, we find
$$
(a,\tilde\varphi(s))_n-\widehat
{(a,\tilde\varphi(s))_n}=(a,\tilde\varphi(s))_n-\sum_{i=1}^N(u_i,
\tilde y_i)_n -c
$$
$$
=(z_{i_0}^{(1)}(s;u),\tilde\varphi(s))_n-(z_{i_0}^{(2)}(s;u),\tilde\varphi(s))_n
-\sum_{i=1,i\neq i_0-1,i_0}^{N}\left(z_{i+1}(t_i;u)-z_{i}(t_i;u),\tilde\varphi(t_i)\right)_n
$$
\begin{equation}\label{0hkv}
-\left(z_{i_0+1}(t_{i_0};u)-z_{i_0}^{(2)}(t_{i_0};u),\tilde\varphi(t_{i_0})\right)_n -\left(z_{i_0}^{(1)}(t_{i_0-1};u)
-z_{i_0-1}(t_{i_0-1};u),\tilde\varphi(t_{i_0-1})\right)_n
-\sum_{i=1}^N(u_i,\xi_i)_n-c,
\end{equation}
where the fourth or the third term on the right-hand side of
\eqref{0hkv} should be taken equal to $0$ at, respectively,
$i_0=1$ or $i_0=N+1$.

Taking into account the latter, relationships
\eqref{sopr10v}--\eqref{sopr12v}, and the equalities (see page
\pageref{pt})
$$
\left(z_1(0;u),\tilde \varphi(0)\right)_n=\left(\bar{B}_0z_1(0;u),\tilde f_0\right)_m,
$$
$$
\left(z_{N+1}(T;u),\tilde \varphi(T)\right)_n
=\left(\bar{B}_1z_{N+1}(T;u),\tilde f_1\right)_{n-m},
$$
we obtain
$$
(a,\tilde\varphi(s))_n-\widehat
{(a,\tilde\varphi(s))_n}=
\left(z_1(0;u),\tilde \varphi(0)\right)_n-\left(z_{N+1}(T;u),\tilde \varphi(T)\right)_n
$$
$$
+ \int^T_0(\tilde
z(t;u),\tilde f(t))_ndt-\sum_{i=1}^{N}(u_i,\xi_i)_n-c
$$
\begin{equation}\label{sopr12v1}
=\left(\bar{B}_0z_1(0;u),\tilde f_0\right)_m+ \int^T_0(\tilde
z(t;u),\tilde f(t))_ndt
- \left(\bar{B}_1z_{N+1}(T;u),\tilde f_1\right)_{n-m}
-\sum_{i=1}^{N}(u_i,\xi_i)_n-c,
\end{equation}
where
$$
\tilde z(t;u)=\left\{
\begin{array}{lc}
z_1(t;u),&t_0=0<t<t_1;\\ \cdots\cdots\cdots&\cdots\cdots\cdots\cdots\\
z_{i_0-1}(t;u),&t_{i_0-2}<t<t_{i_0-1};\\z_{i_0}^{(1)}(t;u),&t_{i_0-1}<t<s;\\ z_{i_0}^{(2)}(t;u),&s<t<t_{i_0};\\z_{i_0+1}(t;u),&t_{i_0}<t<t_{i_0+1};\\
\cdots\cdots\cdots&\cdots\cdots\cdots\cdots\\
z_{N+1}(t;u),&t_{N}<t<t_{N+1}=T.
\end{array}
\right.
$$
Thus,
$$
\inf_{c \in \mathbb R^1} \sup_{\tilde F \in G,\, \tilde \xi \in V}
M[(a,\tilde \varphi(s))_n-(\widehat{a,\tilde \varphi(s)})_n]^2=
$$
$$
=\inf_{c \in \mathbb R^1}\sup_{\tilde F \in
G}\left[(\bar{B}_0z_1(0;u),\tilde f_0)_m+ \int^T_0(\tilde
z(t;u),\tilde f(t))_ndt- (\bar{B}_1z_{N+1}(T;u),\tilde f_1)_{n-m}-c \right]^2
$$
\begin{equation}\label{xq1llv}
+\sup_{\tilde \xi \in V} M\left[\sum_{i=1}^{N}(u_i,\xi_i)_n\right]^2.
\end{equation}
Calculating the supremum on the right-hand side of (\ref{xq1llv})
and taking into consideration \eqref{d7} and \eqref{d6llv}, we
find
$$
\inf_{c \in \mathbb R^1} \sup_{\tilde F \in G,\, \tilde \xi \in V}
M[(a,\tilde \varphi(s))_n-(\widehat{a,\tilde \varphi(s)})_n]^2=I(u),
$$
where $I(u)$ is given by \eqref{N4llv}.
 \end{proof}

 Starting from this lemma and applying the reasoning that led from
 Lemma \ref{lem1} to Theorems  \ref{tm1} and \ref{tm2}, we obtain
 the following results.
\begin{pred}
The minimax estimate of expression $(a,\varphi(s))$ has the form
$$
\widehat{\widehat
{(a,\varphi(s))_n}}=\sum_{i=1}^N(\hat u_i,
y_i)_n+\hat c,
$$
where
\begin{equation}\label{ddd7llv}
\hat{u}_i=\tilde Q_ip_i(t_i),\quad i=1,\ldots,i_0-2,i_0+1,\ldots,N,
\quad \hat{u}_{i_0-1}=\tilde Q_{i_0-1}p_{i_0}^{(1)}(t_{i_0-1}), \quad \hat{u}_{i_0}=\tilde Q_{i_0}p_{i_0}^{(2)}(t_{i_0}),
\end{equation}
$$
\hat c=(\bar{B}_0z_1(0),f_0^{(0)})_m - (\bar{B}_1z_{N+1}(T),f_1^{(0)})_{n-m}+
\int_0^T (\tilde{z}(t),f^{(0)}(t))_n\,dt,
$$
$$
\tilde z(t)=\left\{
\begin{array}{lc}
z_1(t),&t_0=0<t<t_1;\\ \cdots\cdots\cdots&\cdots\cdots\cdots\cdots\\
z_{i_0-1}(t),&t_{i_0-2}<t<t_{i_0-1};\\z_{i_0}^{(1)}(t),&t_{i_0-1}<t<s;\\ z_{i_0}^{(2)}(t),&s<t<t_{i_0};\\z_{i_0+1}(t),&t_{i_0}<t<t_{i_0+1};\\
\cdots\cdots\cdots&\cdots\cdots\cdots\cdots\\
z_{N+1}(t),&t_{N}<t<t_{N+1}=T,
\end{array}
\right.
$$
and vector-functions $p_i(t),$ $z_i(t)$, $i=\overline{1,N+1},
i\neq i_0,$ $z_{i_0}^{(1)}(t),$ $p_{i_0}^{(1)}(t),$
$z_{i_0}^{(2)}(t),$ and $p_{i_0}^{(2)}(t),$ are determined from
the solution to the equation systems
\begin{gather}
L^* z_1(t)=0,\quad 0<t<t_1, \quad \hat B_0z_1(0;u)=0,\notag\\
L^* z_2(t)=0,\quad t_1<t<t_2, \quad z_2(t_1)=z_1(t_1)+\tilde Q_1p_1(t_1),\notag\\
\ldots\ldots\ldots\ldots\ldots\ldots\ldots\ldots\ldots
\ldots\ldots\ldots\ldots\ldots\ldots\ldots\ldots\notag\\
L^* z_{i_0-1}(t)=0,\quad t_{i_0-2}<t<t_{i_0-1}, \quad z_{i_0-1}(t_{i_0-2})=z_{i_0-2}(t_{i_0-2})+\tilde Q_{i_0-2}p_{i_0-2}(t_{i_0-2}),\notag\\
L^* z_{i_0}^{(1)}(t)=0,\quad t_{i_0-1}<t<s, \quad z_{i_0}^{(1)}(t_{i_0-1})=z_{i_0-1}(t_{i_0-1})+\tilde Q_{i_0-1}p_{i_0}^{(1)}(t_{i_0-1}),\notag\\
L^* z_{i_0}^{(2)}(t)=0,\quad s<t<t_{i_0}, \quad z_{i_0}^{(2)}(s)=z_{i_0}^{(1)}(s)-a, \notag\\
L^* z_{i_0+1}(t)=0,\quad t_{i_0}<t<t_{i_0+1}, \quad z_{i_0+1}(t_{i_0})=z_{i_0}^{(2)}(t_{i_0})+\tilde Q_{i_0}p_{i_0}^{(2)}(t_{i_0}),\notag\\
\ldots\ldots\ldots\ldots\ldots\ldots\ldots\ldots\ldots
\ldots\ldots\ldots\ldots\ldots\ldots\ldots\ldots\notag\\
L^* z_N(t)=0,\quad t_{N-1}<t<t_{N}, \quad z_N(t_{N-1})=z_{N-1}(t_{N-1})+\tilde Q_{N-1}p_{N-1}(t_{N-1}),\notag\\
L^* z_{N+1}(t)=0,\quad t_{N}<t<t_{N+1}, \quad z_{N+1}(t_{N})=z_{N}(t_{N})+\tilde Q_Np_N(t_N),\quad \hat B_1z_{N+1}(T;u)=0.\notag\\
Lp_1(t)=Q_2^{-1}(t)z_1(t),\quad  0<t<t_1, \quad
B_0p_1(0)=Q_0^{-1}\bar B_0z_1(0),\notag\\
Lp_2(t)=Q_2^{-1}(t)z_2(t),\quad  t_1<t<t_2, \quad  p_2(t_1)=
p_1(t_1),\notag\\
\ldots\ldots\ldots\ldots\ldots\ldots\ldots\ldots\ldots
\ldots\ldots\ldots\ldots\ldots\ldots\ldots\ldots\notag\\
Lp_{i_0-1}(t)=Q_2^{-1}(t)z_{i_0-1}(t),\quad t_{i_0-2}<t<t_{i_0-1}, \quad p_{i_0-1}(t_{i_0-2})=p_{i_0-2}(t_{i_0-2}),\notag\\
Lp_{i_0}^{(1)}(t)=Q_2^{-1}(t)z_{i_0}^{(1)}(t),\quad t_{i_0-1}<t<s, \quad p_{i_0}^{(1)}(t_{i_0-1})=p_{i_0-1}(t_{i_0-1}),\notag\\
Lp_{i_0}^{(2)}(t)=Q_2^{-1}(t)z_{i_0}^{(2)}(t),\quad s<t<t_{i_0}, \quad p_{i_0}^{(2)}(s)=p_{i_0}^{(1)}(s),\notag\\
Lp_{i_0+1}(t)=Q_2^{-1}(t)z_{i_0+1}(t),\quad t_{i_0}<t<t_{i_0+1}, \quad p_{i_0+1}(t_{i_0})=p_{i_0}^{(2)}(t_{i_0}),\notag\\
\ldots\ldots\ldots\ldots\ldots\ldots\ldots\ldots\ldots
\ldots\ldots\ldots\ldots\ldots\ldots\ldots\ldots\notag\\
Lp_N(t)=Q_2^{-1}(t)z_N,\quad t_{N-1}<t<t_{N}, \quad p_N(t_{N-1})=p_{N-1}(t_{N-1}),\notag\\
Lp_{N+1}(t)=Q_2^{-1}(t)z_{N+1}(t),\,\, t_N<t<T,\,\,
p_{N+1}(t_N)=p_{N}(t_N), \,\, B_1p_{N+1}(T)=-Q_1^{-1}\bar B_1z_{N+1}(T).\label{syst5.1v}
\end{gather}
Here $z_i,p_i\in H^1(t_{i-1},t_{i})^n,$
$ i=\overline{1, N+1},i\neq i_0,$ $z_{i_0}^{(1)},p_{i_0}^{(1)}\in
H^1(t_{i_0-1},s)^n,$ and $z_{i_0}^{(2)},p_{i_0}^{(2)}\in
H^1(s,t_{i_0})^n.$ The minimax estimation error
\begin{equation} \label{ddd8v}
\sigma=(a,p_{i_0}^{(2)}(s))_n^{1/2}.
\end{equation}
System (\ref{syst5.1v}) is uniquely solvable.
\end{pred}

\begin{pred}
The following representation is valid
$$
\widehat{\widehat{(a,\varphi(s))_n}}=(a,\hat{\varphi}_{i_0}^{(2)}(s))_n,
$$
where vector-functions $\hat\varphi_i(t),$ $i=\overline{1,N+1}, i\neq i_0,$ $\hat\varphi_{i_0}^{(1)}(t),$
$\hat\varphi_{i_0}^{(2)}(t)$
are determined from the solution to the equation systems
\begin{gather*}
L^* \hat p_1(t)=0,\quad 0<t<t_1, \quad \hat B_0\hat p_1(0;u)=0,\notag\\
L^* \hat p_2(t)=0,\quad t_1<t<t_2, \quad \hat p_2(t_1)=\hat p_1(t_1)+\tilde Q_1(\hat\varphi_1(t_1)-y_1),\notag\\
\ldots\ldots\ldots\ldots\ldots\ldots\ldots\ldots\ldots
\ldots\ldots\ldots\ldots\ldots\ldots\ldots\ldots\notag\\
L^* \hat p_{i_0-1}(t)=0,\quad t_{i_0-2}<t<t_{i_0-1}, \quad \hat p_{i_0-1}(t_{i_0-2})=\hat p_{i_0-2}(t_{i_0-2})+\tilde Q_{i_0-2}(\hat\varphi_{i_0-2}(t_{i_0-2})-y_{i_0-2}),\notag\\
L^* \hat p_{i_0}^{(1)}(t)=0,\quad t_{i_0-1}<t<s, \quad \hat p_{i_0}^{(1)}(t_{i_0-1})=\hat p_{i_0-1}(t_{i_0-1})+\tilde Q_{i_0-1}(\hat\varphi_{i_0}^{(1)}(t_{i_0-1})-y_{i_0-1}),\notag\\
L^* \hat p_{i_0}^{(2)}(t)=0,\quad s<t<t_{i_0}, \quad \hat p_{i_0}^{(2)}(s)=\hat p_{i_0}^{(1)}(s), \notag\\
L^* \hat p_{i_0+1}(t)=0,\quad t_{i_0}<t<t_{i_0+1}, \quad \hat p_{i_0+1}(t_{i_0})=\hat p_{i_0}^{(2)}(t_{i_0})+\tilde Q_{i_0}(\hat\varphi_{i_0}^{(2)}(t_{i_0})-y_{i_0}),\notag\\
\ldots\ldots\ldots\ldots\ldots\ldots\ldots\ldots\ldots
\ldots\ldots\ldots\ldots\ldots\ldots\ldots\ldots\notag\\
L^* \hat p_N(t)=0,\quad t_{N-1}<t<t_{N}, \quad \hat p_N(t_{N-1})=\hat p_{N-1}(t_{N-1})+\tilde Q_{N-1}(\hat\varphi_{N-1}(t_{N-1})-y_{N-1}),\notag\\
L^* \hat p_{N+1}(t)=0,\quad t_{N}<t<t_{N+1}, \quad \hat p_{N+1}(t_{N})=\hat p_{N}(t_{N})+\tilde Q_N(\hat\varphi_N(t_N)-y_N),\quad \hat B_1\hat p_{N+1}(T;u)=0.\notag\\
L\hat\varphi_1(t)=Q_2^{-1}(t)\hat p_1(t)+f^{(0)}(t),\quad  0<t<t_1, \quad
B_0\hat\varphi_1(0)=Q_0^{-1}\bar B_0\hat p_1(0)+f_0^{(0)},\notag\\
L\hat\varphi_2(t)=Q_2^{-1}(t)\hat p_2(t)+f^{(0)}(t),\quad  t_1<t<t_2, \quad  \hat\varphi_2(t_1)=
\hat\varphi_1(t_1),\notag\\
\ldots\ldots\ldots\ldots\ldots\ldots\ldots\ldots\ldots
\ldots\ldots\ldots\ldots\ldots\ldots\ldots\ldots\notag\\
L\hat\varphi_{i_0-1}(t)=Q_2^{-1}(t)\hat p_{i_0-1}(t)+f^{(0)}(t),\quad t_{i_0-2}<t<t_{i_0-1}, \quad \hat\varphi_{i_0-1}(t_{i_0-2})=\hat\varphi_{i_0-2}(t_{i_0-2}),\notag\\
L\hat\varphi_{i_0}^{(1)}(t)=Q_2^{-1}(t)\hat p_{i_0}^{(1)}(t)+f^{(0)}(t),\quad t_{i_0-1}<t<s, \quad \hat\varphi_{i_0}^{(1)}(t_{i_0-1})=\hat\varphi_{i_0-1}(t_{i_0-1}),\notag\\
L\hat\varphi_{i_0}^{(2)}(t)=Q_2^{-1}(t)\hat p_{i_0}^{(2)}(t)+f^{(0)}(t),\quad s<t<t_{i_0}, \quad \hat\varphi_{i_0}^{(2)}(s)=\hat\varphi_{i_0}^{(1)}(s),\notag\\
L\hat\varphi_{i_0+1}(t)=Q_2^{-1}(t)\hat p_{i_0+1}(t)+f^{(0)}(t),\quad t_{i_0}<t<t_{i_0+1}, \quad \hat\varphi_{i_0+1}(t_{i_0})=\hat\varphi_{i_0}^{(2)}(t_{i_0}),\notag\\
\ldots\ldots\ldots\ldots\ldots\ldots\ldots\ldots\ldots
\ldots\ldots\ldots\ldots\ldots\ldots\ldots\ldots\notag\\
L\hat\varphi_N(t)=Q_2^{-1}(t)\hat p_N(t)+f^{(0)}(t),\quad t_{N-1}<t<t_{N}, \quad \hat\varphi_N(t_{N-1})=\hat\varphi_{N-1}(t_{N-1}),\notag\\
\end{gather*}\\[-59pt]
\begin{multline}
L\hat\varphi_{N+1}(t)=Q_2^{-1}(t)\hat p_{N+1}(t)+f^{(0)}(t),\quad t_N<t<T,\quad
\hat\varphi_{N+1}(t_N)=\hat\varphi_{N}(t_N), \\ B_1\hat\varphi_{N+1}(T)=-Q_1^{-1}\bar B_1\hat p_{N+1}(T)-f_1^{(0)}.\label{syst5.1llv}
\end{multline}
Here $\hat p_i=\hat p_i(\cdot;\omega),\hat\varphi_i=\hat\varphi_i(\cdot;\omega)\in H^1(t_{i-1},t_{i})^n,$
$
i=\overline{1, N+1},i\neq i_0,$
$\hat p_{i_0}^{(1)}=\hat p_{i_0}^{(1)}(\cdot;\omega),\hat\varphi_{i_0}^{(1)}
=\hat\varphi_{i_0}^{(1)}(\cdot;\omega)\in H^1(t_{i_0-1},s)^n,$ $\hat p_{i_0}^{(2)}=p_{i_0}^{(2)}(\cdot;\omega),\hat\varphi_{i_0}^{(2)}
=\hat\varphi_{i_0}^{(2)}(\cdot;\omega)\in H^1(s,t_{i_0})^n,$
and equalities (\ref{syst5.1llv}) are fulfilled with probability $1.$
System (\ref{syst5.1llv}) is uniquely solvable.
\end{pred}

\section{
Minimax estimation of functionals of solutions to
boundary value problems for linear differential equations of order
$n$}




In this chapter we propose a method for minimax estimation of
parameters of general two-point BVPs  for linear ordinary
differential equations of order $n$;  their solutions are
determined to within functions that are solutions to the
corresponding homogeneous problems and exist if the right-hand
sides of the equations and boundary conditions entering the problem
statement satisfy certain solvability conditions.

%

\subsection{Auxiliary results}
If $H_0$ is a Hilbert space over  $\mathbb C$ with the inner
product $(\cdot,\cdot)_{H_0}$ and norm $\|\cdot\|_{H_0},$ then by
$J_{H_0}\in \mathcal L(H_0,H_0')$ we will denote an operator
called isometric isomorphism; this operator acts from  $H_0$ on
its conjugate space $H_0'$ and is defined by en equality
\footnote{This operator exists by the Riesz theorem.} $
(v,u)_{H_0}=$ $<v,J_{H_0} u>_{H_0\times H_0'}$ $\forall u,v\in
H_0,$ where $<x, f>_{H_0\times H_0'}$ $:=f(x)$ for $x\in H_0,$
$f\in H_0'.$

Denote by $L^2(a,b)$ the space of functions square integrable on
$(a,b).$ For any $n\geq 1$ denote by $W^n_2(a,b)$ the space of
functions absolutely continuous on $[a,b]$ together with the
derivatives up to order $(n-1)$ for which the derivative of order
$n$ that exists almost everywhere on $(a,b)$ belongs to
$L^2(a,b).$

Assume that functions $p_i(t)$ defined on a closed interval
$[a,b]$ are such that  $p_i^{(n-i)}\in C[a,b],$
$i=\overline{0,n},$ and $p_0(t)\neq 0$ on $[a,b].$ Let  functions
$\varphi(t),$ belong to the space $W^n_2(a,b);$ define a
differential operator $L$ acting in $L^2(a,b)$ by the formula
$$
L\varphi(t)=p_0(t)\varphi^{(n)}(t)+p_1(t)\varphi^{(n-1)}(t)+\ldots+
p_n(t)\varphi(t).
$$

Assume next that there are given a function $f\in L^2(a,b)$ and
numbers $\alpha_i,$ $i=\overline{1,m}.$ Consider the following
BVP: find $\varphi\in W^n_2(a,b)$ that satisfies the equation
\begin{equation}\label{skkv1}
L\varphi(t)=f(t)
\end{equation}
almost everywhere on  $(a,b)$ and the boundary conditions
\begin{equation}\label{skk1v1}
B_i(\varphi)=\alpha_i,\,\,i=\overline{1,m},
\end{equation}
where
\begin{equation}\label{scdv1}
B_i(\varphi)=\sum_{j=0}^{n-1}(\alpha_{i,j}\varphi^{(j)}(a)
+\beta_{i,j}\varphi^{(j)}(b)), \quad i=\overline{1,m},
\end{equation}
is a given system of $m$ linearly independent\footnote{It means
that the rank of the matrix composed of the coefficients of these
forms equals $m.$} forms of  $2n$ variables
\begin{equation}\label{skk2v1}
\varphi(a),\ldots,\varphi^{(n-1)}(a),
\varphi(b),\ldots,\varphi^{(n-1)}(b).
\end{equation}

In order to describe right-hand sides $f\in L^2(a,b)$ and
$\alpha_i,$ $i=\overline{1,m},$ for which BVP
(\ref{skkv1}), (\ref{skk1v1}) is solvable and also to formulate the results obtained in this work, it is necessary
to introduce a problem adjoint to (\ref{skkv1}), (\ref{skk1v1}). To this end, complement $m$
linearly independent forms $B_i(\varphi),$ $i=\overline{1,m},$ 
given by (\ref{skk2v1}) by some linear forms
$S_{1}(\varphi),\ldots,S_{2n-m}(\varphi)$ to a linearly independent system of $2n$ forms  $ \{B_1(\varphi),\ldots,B_m(\varphi),
S_{1}(\varphi),\ldots,$ $ S_{2n-m}(\varphi)\} $ with respect to the same variables. Denote by $L^+$ a differential operator which is called formally adjoint to $L$; this operator is defined on functions from  $W^n_2(a,b)$ and act in $L^2(a,b)$ according to
$$
L^+\psi(t)=(-1)^n(\overline{
p_0(t)}\psi(t))^{(n)}+(-1)^{(n-1)}(\overline{
p_1(t)}\psi(t))^{(n-1)}+\ldots+\overline{p_n(t)}\psi(t).
$$
Using $B_i(\varphi),$ $i=\overline{1,m},$ and $S_{j}(\varphi),$
$j=\overline{1,2n-m},$ one can construct systems of linear forms $B_j^+(\psi),\,$ $j=\overline{1,2n-m},$ and $S_{i}^+(\psi),\,$
$\,i=\overline{1,m},$ with respect to variables
$
\psi(a),\ldots,$ $\psi^{(n-1)}(a),$ $\psi(b),$ $\ldots,$ $\psi^{(n-1)}(b).
$
The constructed forms possess the following properties:\\ $(i)$ system of forms
$
\{S_{1}^+(\psi),\ldots,S_{m}^+(\psi),
B_1^+(\psi),\ldots,B_{2n-m}^+(\psi) \}
$
is linearly independent;\\ $(ii)$ for any $\varphi,\psi\in W^n_2(a,b)$
the Green formula is valid:
\begin{equation}\label{greenv1}
\int_{a}^b L\varphi(t)\overline{\psi(t)}\, dt+\sum_{j=1}^m
B_j(\varphi)\overline{S_j^+(\psi)}=\sum_{j=1}^{2n-m}
S_j(\varphi)\overline{B_j^+(\psi)}+\int_{a}^b\varphi(t)
\overline{L^+\psi(t)}\,dt.
\end{equation}

Formulate a BVP: find a function $\psi\in W^n_2(a,b)$ that satisfies the equation
\begin{equation}\label{skk3v1}
L^+\psi(t)=g(t)
\end{equation}
almost everywhere on $(a,b)$ and the boundary conditions
\begin{equation}\label{skk4v1}
B_j^+(\psi)=\beta_j,\,\,j=\overline{1,2n-m},
\end{equation}
where $g\in L^2(a,b)$ is a given function and $\beta_j,$
$j=\overline{1,2n-m}$ are given numbers. This problem will be called adjoint to BVP (\ref{skkv1}), (\ref{skk1v1}).

Problems (\ref{skkv1}), (\ref{skk1v1}) and (\ref{skk3v1}), (\ref{skk4v1})
give rise to operators $A_B$ and $A^+_{B^+}$ defined on functions from
$W^n_2(a,b)\subset L^2(a,b)$ according to\label{rrv1}
\begin{equation}\label{A_Bv1}
A_B(\varphi)=\{L\varphi;B_1(\varphi),\ldots,B_m(\varphi)\},
\end{equation}
and
\begin{equation}\label{A_B+v1}
A^+_{B^+}(\psi)=\{L^+\psi;B_1^+(\psi),\ldots, B_{2n-m}^+(\psi)\};
\end{equation}
the operators act to the spaces $H:=L^2(a,b)\times \mathbb
C^m$ and $\tilde H:=L^2(a,b)\times \mathbb C^{2n-m},$
respectively. From the results proved in \cite{Krein}, it follows that

1)
 $A_B$ and  $A^+_{B^+}$ are Noether operators
 \footnote{Recall that
$A$ is a Noether operator if the dimensionality of its kernel $N(A)$
is finite and its image $R(A)$ is closed and has a finite codimensionality; then its index
$\chi_A=\mbox{dim}\,N(A)-\mbox{codim}\,R(A).$} acting, respectively, from
 $L^2(a,b)$ to $H$ and $\tilde
H$. Kernel $N(A_B)$ of operator $A_B$ has finite dimensionality $n-r$ and coincides with the set $N(A_B)=\{\varphi_0\in
C^n[a,b]:L\varphi_0=0\,\,\mbox{on}\,\,(a,b),
\,B_i(\varphi_0)=0,\,i=\overline{1,m} \}$ of all solutions to homogeneous BVP (\ref{skkv1}), (\ref{skk1v1}); kernel
$N(A^+_{B^+})$ of operator $A^+_{B^+}$ has dimensionality  $m-r$ and coincides with the set $N(A^+_{B^+})=\{\psi_0\in
C^n[a,b]:L^+\psi_0=0\,\,\mbox{on}\,\,(a,b), $ $\,B_j^+(\psi_0)
=0,\,j=\overline{1,2n-m} \}$ of all solutions to homogeneous BVP (\ref{skk3v1}), (\ref{skk4v1}); number $r$ is the rank of the matrix
$$
\left(
\begin{array}{cccc}
B_1(y_1)& B_1(y_2)&\ldots &B_1(y_n)\\ B_2(y_1)& B_2(y_2)&\ldots
&B_2(y_n)\\ \vdots& \vdots& \ddots& \vdots \\ B_m(y_1)&
B_m(y_2)&\ldots &B_m(y_n)
\end{array}
\right)
$$
and $y_1(t),\ldots,y_n(t)$ is a fundamental system of solutions to homogeneous equation (\ref{skkv1}).

2) BVP (\ref{skkv1}), (\ref{skk1v1}) is solvable for given 
$f\in L^2(a,b),$
$\alpha_i\in \mathbb C,$ $i=\overline{1,m},$ if and\label{rr+1v1} only if the solvability condition
\begin{equation}\label{skk5v1}
\int_a^bf(t)\overline{\psi_0(t)}\,dt+\sum_{i=1}^m\alpha_i
\overline{S_i^+(\psi_0)}=0 \quad\forall \psi_0\in N(A^+_{B^+})
\end{equation}
holds. If $\varphi(t)$ is a solution to (\ref{skkv1}),
(\ref{skk1v1}), then $\varphi(t)+\varphi_0(t)$ is also a solution for any $\varphi_0(t)\in N(A_B)$.

3) BVP (\ref{skk3v1}), (\ref{skk4v1}) is solvable for given 
$g\in L^2(a,b),$ $\beta_j\in
\mathbb C,$ $j=\overline{1,2n-m},$ if and only if the solvability condition
\begin{equation}\label{skk6v1}
\int_a^bg(t)\overline{\varphi_0(t)}\,dt
+\sum_{j=1}^{2n-m}\beta_j\overline{S_j(\varphi_0)}=0 \quad\forall
\varphi_0\in N(A_B)
\end{equation}
holds. If $\psi(t)$ is a solution to (\ref{skk3v1}), (\ref{skk4v1}), then $\psi(t)+\psi_0(t)$ is also a solution for any $\psi_0(t)\in N(A^+_{B^+})$.

Let $\varphi_1(t),\ldots,\varphi_{n-r}(t)$ and
$\psi_1(t),\ldots,\psi_{m-r}(t)$ denote in what follows bases of null-spaces $N(A_B)$ and
$N(A^+_{B^+})$ of operators $A_B$ and $A^+_{B^+}$, respectively.

\subsection{Statement of the estimation problem
}
An estimation problem can be formulated as follows:
to find the optimal (in a certain sense) estimate of the value of the functional
\begin{equation}\label{sk8fv1}
l(\varphi)=\int_a^b\overline{l_0(t)}\varphi(t)\,dt
\end{equation}
from observations of the form
\begin{equation}\label{skk7v1}
y=C\varphi+\eta
\end{equation}
in the class of estimates
\begin{equation}\label{sk9fv1}
\widehat {l(\varphi)}=(y,u)_{H_0}+c,
\end{equation}
linear with respect to observations; here $\varphi(x)$ is a solution to BVP (\ref{skkv1}),
$u$ is an element of Hilbert space\footnote{If $H_0$ is a finite-dimensional space then it is assumed that
$\mbox{dim}\,H_0> n-r.$} $H_0,$ $c\in \mathbb C,$ and $l_0\in
L^2(a,b)$ is a given function. It is assumed that right-hand sides
$f(t),$ $\alpha_1,\ldots,\alpha_n$ in (\ref{skkv1}),
(\ref{skk1v1}) and errors $\eta=\eta(\omega)$ in observations
(\ref{skk7v1}) that are random elements defined on a probability space $(\Omega, \mathcal B, P)$ with values in
$H_0$ are not known and it is known only that the element $F:=(f(\cdot),\mathbf
\alpha)\in G_0$ and $\eta\in G_1.$ Here $\varphi(x)$ is a solution to BVP (\ref{skkv1}), (\ref{skk1v1}); $C\in \mathcal
L(L^2(a,b),H_0)$ is a linear continuous operator such that its restriction on subspace  $N(A_B)$ is injective;
$\mathbf{\alpha}:=(\alpha_1,\ldots,\alpha_m)^T$ $\in \mathbb C^m$
is a vector with components $\alpha_1,\ldots,\alpha_m;$ $G_0$
denotes the set of elements
$$
\tilde F:=(\tilde f(\cdot),\widetilde{\mathbf\alpha})=(\tilde
f(\cdot),(\tilde \alpha_1,\ldots,\tilde \alpha_m)^T)\in
L^2(a,b)\times \mathbb C^m,
$$
satisfying the condition
\begin{equation}\label{skk10v1}
\int_a^b\tilde f(t)\overline{\psi_0(t)}\,dt+\sum_{i=1}^m\tilde
\alpha_i\overline{S_i^+(\psi_0)}=0 \quad\forall \psi_0\in
N(A^+_{B^+})
\end{equation}
and the inequality
\begin{equation}\label{skk11v1}
\int_a^bQ(\tilde f(t)- f^{(0)}(t)) \overline{(\tilde f(t)-
f^{(0)}(t))}\,dt +
(Q_1(\widetilde{\mathbf\alpha}-\mathbf\alpha^{(0)}),
\widetilde{\mathbf\alpha}-\mathbf\alpha^{(0)})_{\mathbb C^m} \leq
1,
\end{equation}
in which $(\cdot,\cdot)_{\mathbb C^m}$ is the inner product in
$\mathbb C^m,$ element $(
f^{(0)}(\cdot),\mathbf\alpha^{(0)})=(f^{(0)}(\cdot),$ $(
\alpha^{(0)}_1,\ldots,$ $\alpha^{(0)}_m)^T)$ $\in L^2(a,b)\times
\mathbb C^m$ and satisfies (\ref{skk10v1}), and $G_1$ is a\label{G_1v1} set of random elements $\tilde
\eta=\tilde\eta(\omega)$ defined on a probability space
$(\Omega, \mathcal B, P)$ with values in  $H_0,$ zero means, and finite second moments $\mathbb
M\|\tilde\eta\|_{H}^2<\infty$ satisfying
\begin{equation}\label{skk12v1}
M(Q_0\tilde\eta,\tilde\eta)_{H_0}\leq 1,
\end{equation}
where $Q$ and $Q_0$ are Hermitian operators in
$L^2(a,b)$ and $H_0$, respectively, $Q_1$ is a Hermitian
$m\times m$ matrix for which there exist, respectively, bounded inverse operators
$Q^{-1}$ and $Q_0^{-1},$ and inverse matrix $Q_1^{-1}.$

\begin{predll}\label{oz1v1}
An estimate
$$
\widehat{\widehat {l(\varphi)}}=(y,\hat
u)_{H_0}+\hat c
$$
for which an element $\hat{u}$
and a constant $\hat{c}$
are determined from the condition
$$
\sigma(u,c):=\sup_{\tilde F\in
G_0, \tilde \eta\in G_1}
M|l(\tilde\varphi)-\widehat
{l(\tilde\varphi)}|^2\to\inf_{u\in H_0, c \in \mathbb
C}:=\sigma^2,
$$
where
\begin{equation}\label{otsv1}
\widehat {l(\tilde\varphi)}
=(\tilde y,u)_{H_0}+c,
\end{equation}
$\tilde y=C\tilde\varphi+\tilde\eta,$
and $\tilde\varphi$ is any solution to BVP \eqref{skkv1},
\eqref{skk1v1} at $f(t)=\tilde f(t),$ $\alpha_i=\tilde \alpha_i,
i=\overline{1,m},$ will be called a minimax estimate of $l(\varphi).$

The quantity
$$
\sigma=\sup_{\tilde F\in
G_0, \tilde \eta\in G_1} \{M|l(\tilde \varphi)-\widehat{\widehat
{l(\tilde \varphi)}}|^2\}^{1/2} $$ will be called the minimax estimation error of $l(\varphi).$
\end{predll}


\subsection{Representations for minimax estimates of the values of functionals from  solutions and estimation errors}
Using the statements formulated in Section 2.1, we arrive at the following results.

\begin{predl}\label{lemma2v1}
Finding a minimax estimate of the value of functional $l(\varphi)$ is equivalent to the problem of optimal control
of the integro-differential equation system
\begin{equation}\label{gbk21v1}
L^+z(t;u)=l_0(t)-C^*J_{H_0}u\quad \mbox{на}\quad (a,b),
\end{equation}
\begin{equation}\label{gbk22v1}
 B_j^+(z(\cdot;u))=0\quad  (j=1,\ldots,2n-m),
\end{equation}
\begin{equation}\label{gbk23v1}
\int_a^b Q^{-1}z(t;u)\overline{\psi_i(t)} dt+ (Q_1^{-1}\mathbf
S^+(z(\cdot;u)),\mathbf S^+(\psi_i))_{\mathbb C^m}=0,\quad
i=\overline{1,m-r},
\end{equation}
with the cost function
\begin{multline}\label{gbk24''v1}
I(u)=\int_a^b Q^{-1}z(t;u)\overline{z(t;u)}dt +(Q_1^{-1}\mathbf
S^+(z(\cdot;u)), \mathbf S^+(z(\cdot;u)))_{\mathbb C^m}\\+
(Q_0^{-1}u,u)_{H_0}\!\to \inf_{u\in U},
\end{multline}
where
\begin{equation*}
U=\{u\in H_0: \int_a^b\bigl(l_0(t)-(C^*J_{H_0}u)(t)\bigr)
\overline{\varphi_0(t)}dt=0\quad \forall\varphi_0\in N(A_B)\},
\end{equation*}
$C^*:H_0'\to L_2(a,b)$ is the operator adjoint to $C$ defined by the relationship
$$
<Cv,w>_{H_0\times H_0'}=\int_a^bv(x)\overline{C^*w(x)}dt
\quad  \forall v\in L^2(a,b),\,\,w\in H_0',
$$
$\mathbf
l=(l_1,\ldots,l_m)^T,$ $\mathbf
S^+(z(\cdot;u)):=(S^+_1(z(\cdot;u)),\ldots,S^+_m(z(\cdot;u)))^T$ and
$\mathbf S^+(\psi_i):=(S^+_1(\psi_i),\ldots,S^+_m(\psi_i))^T
\in
\mathbb C^m$ are vectors with components, respectively, $l_j,$ $S_j^+(z(\cdot;u)),$
and $S_j^+(\psi_i),$ $j=\overline{1,m}$.
\end{predl}
\begin{proof}
Show first that set $U$ is nonempty. Indeed, it is easy to see that $U$ is the meet of $n-r$ hypersurfaces
\begin{equation}\label{hipv1}
(C\varphi_i,u)_{H_0}=\gamma_i
\end{equation}
in space $H_0,$ where $\gamma_i=\int_a^b\varphi_i(t)l_0(t)\,dt$ and
$\varphi_i(t),$ $i=1,\ldots,n-r,$ is a basis of subspace $N(A_B).$

Denote by $\mbox{span}\{C\varphi_1,\ldots,C\varphi_{n-r}\}$
a subspace  in  $H_0$ spanned over vectors  $C\varphi_1,\ldots,$
$C\varphi_{n-r}$ and prove that there is one and only one element  $u_0\in
\mbox{span}\{C\varphi_1,$ $\ldots,C\varphi_{n-r}\}$ that belongs to set
$U.$ To this end, representing $u_0$ as
$u_0=\sum_{j=1}^{n-r}\beta_jC\varphi_j,$ where $\beta_j\in\mathbb C,$
and substituting this into (\ref{hipv1}), we see that $u_0$ belongs to $U$ if and only if the linear equation system
\begin{equation}\label{slauv1}
\sum_{j=1}^{n-r}\bar\beta_j(C\varphi_i,C\varphi_j)_{H_0}=\gamma_i
,\,\,i=1,\ldots,n-r,
\end{equation}
with respect to unknowns  $\beta_j$ is solvable. Indeed, operator $C$ is injective on $N(A_B)$; therefore, $C\varphi_j,$
$j=1,\ldots,n-r,$ are linearly independent and
$\mbox{det}\{(C\varphi_i,C\varphi_j)_{H_0}\}_{i,j=1}^{n-r}\ne 0$, so that system (\ref{slauv1}) has unique solution
$\beta_1,\ldots,\beta_{n-r}$. Consequently,  the element
$u_0=\sum_{j=1}^{n-r}\beta_jC\varphi_j,$ belongs, as well as
$u_0+u^\perp$ for any $u^\perp\in
H_0\circleddash\mbox{span}\{C\varphi_1,\ldots,C\varphi_{n-r}\}$,
to set $U.$ Thus $U\neq\emptyset.$

Next, let us show that for every fixed  $u\in U$, function
$z(t;u)$ can be uniquely determined from equalities (\ref{gbk21v1})--(\ref{gbk23v1}). Indeed, the condition $u\in U$ coincides, according to (\ref{skk6v1}), with the solvability condition for problem (\ref{gbk21v1})--(\ref{gbk22v1}). Let
$z_0(t;u)\in W^n_2(a,b)$ be a solution to this problem, e.g. a solution that is orthogonal to subspace $N(A^+_{B^+})$; then the function\label{20v1}
\begin{equation}\label{gbk24v1}
z(t;u):=z_0(t;u)+\sum_{i=1}^{m-r}c_i\psi_i(t),
\end{equation}
also satisfies (\ref{gbk21v1})--(\ref{gbk22v1}) for any
$c_i,\in\mathbb C^1,$ $i=\overline{1,m-r}.$ Let us prove that coefficients $c_i,$ $i=\overline{1,m-r},$ can be chosen so that this function would also satisfy  (\ref{gbk23v1}).
Substituting expression (\ref{gbk24v1}) for $z(u)$ into (\ref{gbk23v1})
we obtain a  system of $m-r$ linear algebraic equations with $m-r$ unknowns $c_1,\ldots, c_{m-r}:$
\begin{equation} \label{gbk26v1}
\sum_{i=1}^{m-r}a_{ij}c_i=b_j(u), \quad j=1,\ldots,m-r,
\end{equation}
where
\begin{equation}\label{a_ijv1}
a_{ij}=(Q^{-1} \psi_i,\psi_j)_{L^2(a,b)}+(Q_1^{-1}\mathbf
S^+(\psi_i),\mathbf S^+(\psi_j))_{\mathbb C^m},
\end{equation}
\begin{equation}\label{b_jv1}
b_j=-\int_a^b Q^{-1}z_0(t;u)\overline{\psi_j(t)} dt-
(Q_1^{-1}\mathbf S^+(z_0(\cdot;u)),\mathbf S^+(\psi_j))_{\mathbb
C^m}.
\end{equation}
Show that matrix $[a_{ij}]_{i,j=1}^{m-r}$ of system
(\ref{gbk26v1}) додатно визначена.
Indeed, taking into account that
$Q^{-1}$ is a Hermitian positive definite operator in $L^2(a,b)$ and
$Q_1^{-1}$ is a Hermitian positive definite $m\times m$ matrix, we have
$$
\sum_{i=1}^{m-r}\sum_{j=1}^{m-r}a_{ij}\lambda_i\bar\lambda_j
=(Q^{-1}\sum_{i=1}^{m-r}
\lambda_i\psi_i,\sum_{j=1}^{m-r}\lambda_j\psi_j)_{L^2(a,b)}
$$
$$
+(Q_1^{-1}(\mathbf S^+(\sum_{i=1}^{m-r} \lambda_i\psi_i),\mathbf
S^+(\sum_{j=1}^{m-r}\lambda_j\psi_j))_{\mathbb C^m} \geq
c\sum_{i=1}^{m-r}|\lambda_i|^2,\quad c=\mbox{const}>0,
$$
for any  $\lambda_i,$ $i=1,\ldots,m-r,$ such that
$\sum_{i=1}^{m-r}|\lambda_i|^2\neq 0.$
The latter implies that matrix  $[a_{ij}]_{i,j=1}^{m-r}$ is positive definite.\label{sav1}
Thus, $\det[a_{ij}]\neq 0$ and system (\ref{gbk26v1}) has unique solution,
$c_1,\ldots,c_{m-r}.$ Therefore, problem (\ref{gbk21v1})--(\ref{gbk23v1})
is uniquely solvable.\label{usv1} Indeed, we have shown that
there exists a solution to problem (\ref{gbk21v1})--(\ref{gbk23v1}); let us prove that this solution is unique. Assume that there are two solutions to this problem, $z_1(t)$ and $z_2(t)$. Then
\begin{equation}\label{gbk211v1}
L^+z_1(t;u)=l_0(t)-C^*J_{H_0}u\quad \mbox{on}\quad (a,b),
\end{equation}
\begin{equation}\label{gbk222v1}
 B_j^+(z_1(\cdot;u))=0\quad  j=\overline{1,2n-m},
\end{equation}
\begin{equation}\label{gbk233v1}
\int_a^b Q^{-1}z_1(t;u)\overline{\psi_i(t)} dt+ (Q_1^{-1}\mathbf
S^+(z_1(\cdot;u)),\mathbf S^+(\psi_i))_{\mathbb C^m}=0,\quad
i=\overline{1,m-r},
\end{equation}

\begin{equation}\label{gbk214v1}
L^+z_2(t;u)=l_0(t)-C^*J_{H_0}u\quad \mbox{on}\quad (a,b),
\end{equation}
\begin{equation}\label{gbk225v1}
 B_j^+(z_2(\cdot;u))=0\quad  j=\overline{1,2n-m},
\end{equation}
\begin{equation}\label{gbk236v1}
\int_a^b Q^{-1}z_2(t;u)\overline{\psi_i(t)} dt+ (Q_1^{-1}\mathbf
S^+(z_2(\cdot;u)),\mathbf S^+(\psi_i))_{\mathbb C^m}=0,\quad
i=\overline{1,m-r},
\end{equation}
Subtract (\ref{gbk214v1})--(\ref{gbk236v1}) from equalities (\ref{gbk211v1})--(\ref{gbk233v1}) to obtain
\begin{equation}\label{gbk217v1}
L^+(z_1(t;u)-z_2(t;u))=0,
\end{equation}
\begin{equation}\label{gbk228v1}
 B_j^+(z_1(\cdot;u)-z_2(\cdot;u))=0\quad  j=\overline{1,2n-m},
\end{equation}
\begin{equation}\label{gbk239v1}
\int_a^b Q^{-1}(z_1(t;u)-z_2(t;u))\overline{\psi_i(t)} dt+
(Q_1^{-1}\mathbf S^+(z_1(t;u)-z_2(\cdot;u)),\mathbf
S^+(\psi_i))_{\mathbb C^m}=0,
\end{equation}
$$
\quad i=\overline{1,m-r}.
$$
Set $z(t;u)=z_1(t;u)-z_2(t;u)$; then
\begin{equation}\label{gbk212v1}
L^+z(t;u)=0,
\end{equation}
\begin{equation}\label{gbk242v1}
 B_j^+(z(\cdot;u))=0\quad  j=\overline{1,2n-m},
\end{equation}
\begin{equation}\label{gbk253v1}
\int_a^b Q^{-1}z(t;u)\overline{\psi_i(t)} dt+ (Q_1^{-1}\mathbf
S^+(z(\cdot;u)),\mathbf S^+(\psi_i))_{\mathbb C^m}=0,\quad
i=\overline{1,m-r}.
\end{equation}
Since $z(t;u)$ solves homogeneous problem
(\ref{gbk21v1})--(\ref{gbk22v1}), this function has the form
\begin{equation}\label{gbk555v1}
z(t;u)=\sum_{i=1}^{m-r}c_i\psi_i(t).
\end{equation}
Substituting (\ref{gbk555v1}) into (\ref{gbk253v1}) we obtain
\begin{equation}\label{gbk273v1}
\sum_{i=1}^{m-r}c_i\left(\int_a^b
Q^{-1}\psi_i(t)\overline{\psi_j(t)} dt+
\sum_{i=1}^{m-r}(Q_1^{-1}\mathbf S^+(\psi_i),\mathbf
S^+(\psi_j)_{\mathbb C^m}\right)=0,\quad j=\overline{1,m-r},
\end{equation}
or, in line with (\ref{a_ijv1}),
\begin{equation}\label{gbk283v1}
\sum_{i=1}^{m-r}c_ia_{ij}=0,\quad j=\overline{1,m-r}.
\end{equation}
We see that coefficients $c_i$ satisfy a linear homogeneous algebraic equation system with nonsingular matrix $[a_{ij}]_{i,j=1}^{m-r}$; therefore, this system has only the trivial solution $c_i=0, i=\overline{1,m-r}$.

By virtue of (\ref{gbk555v1}), $z(t;u)=z_1(t;u)-z_2(t;u)=0$ identically, that is,
$z_1(t;u)=z_2(t;u)$ which proves the unique solvability of problem (\ref{gbk21v1}),(\ref{gbk22v1}).

Next, since any solution  $\tilde \varphi$ of problem \eqref{skkv1}, \eqref{skk1v1}
can be written as
$\tilde \varphi=\tilde \varphi_\bot+\varphi_0,$
 where $\tilde\varphi_0\in N(A_{B})$ and $\tilde \varphi_\bot$ is the unique solution to this problem orthogonal to subspace $N(A_{B}),$
we have\label{page14v1}
\begin{equation}\label{infwv1}
\sup_{\tilde F\in G_0, \tilde \eta\in
G_1} M|l(\tilde
\varphi)-\widehat {l(\tilde \varphi)}|^2=
\sup_{\tilde F\in G_0, \tilde \eta\in
G_1}\,\sup_{\varphi_0\in N(A_{B})} M|l(\tilde
\varphi_\bot+\varphi_0)-\widehat {l(\tilde \varphi_\bot+\varphi_0)}|^2.
\end{equation}
Taking into account (\ref{sk8fv1}), (\ref{otsv1}), and the fact that \label{L2.1v1}
$$
\widehat{l(\tilde\varphi)}=\widehat{l(\tilde\varphi_\bot+\varphi_0)}=(C(\tilde
\varphi_\bot+\varphi_0),u)_{H_0} + (\tilde\eta,u)_{H_0}+c
$$
$$
=<C(\tilde \varphi_\bot+\varphi_0),J_{H_0}u>_{H_0\times
H_0'}+(\tilde\eta,u)_{H_0}+c
$$
$$
=\int_a^b(\tilde
\varphi_\bot(t)+\varphi_0(t))\overline{C^*J_{H_0}u(t)}\,dt
+(\tilde\eta,u)_{H_0}+c
$$
$$
 =\int_a^b\tilde
\varphi_\bot(t)\overline{C^*J_{H_0}u(t)}\,dt
+\int_a^b\varphi_0(t)\overline{C^*J_{H_0}u(t)}\,dt
+(\tilde\eta,u)_{H_0}+c,
$$
for arbitrary $u\in H_0$, we have
\begin{equation*}
l(\tilde \varphi)-\widehat {l(\tilde
\varphi)}=\int_a^b\tilde
\varphi_\bot(t)\overline{(l_0(t)-C^*J_{H_0}u(t))}\,dt\\
+\int_a^b\varphi_0(t)\overline{(l_0(t)-C^*J_{H_0}u(t)}\,dt
-(\tilde\eta,u)_{H_0}-c.
\end{equation*}
From the latter equality we obtain, taking into consideration the relationship $D\xi=M|\xi-M\xi|^2=M\xi_1^2-(M\xi_1)^2+
M\xi_2^2-(M\xi_2)^2$ that couples dispersion $D\xi$
of a complex-valued random quantity  $\xi=\xi_1+i\xi_2$ and its expectation $ M\xi= M\xi_1+i
M\xi_2$,
$$
\sup_{\varphi_0\in N(A_B)} M\left|l(\tilde
\varphi_\bot(t)+\varphi_0)-\widehat {l(\tilde \varphi_\bot(t)+\varphi_0)}\right|^2
=\sup_{\varphi_0\in N(A_B)} \left|\int_a^b\tilde
\varphi_\bot(t)(t)\overline{(l_0(t)-C^*J_{H_0}u(t))}\,dt \right.
$$
\begin{equation}\label{ltkkgv1}
\left.+\int_a^b\varphi_0(t)
\overline{(l_0(t)-C^*J_{H_0}u(t))}\,dt-c\right|^2
+M|(\tilde\eta,u)_{H_0}|^2.
\end{equation}
Since function $\varphi_0(t)$ under the integral sign may be an arbitrary element of space $N(A_B),$ the quantity
$$ \sup_{\varphi_0\in N(A_B)} M\left|l(\tilde
\varphi_\bot(t)+\varphi_0)-\widehat {l(\tilde
\varphi_\bot(t)+\varphi_0)}\right|^2$$ will be finite if and only if  $u\in U,$ i.e. if the second integral on the right-hand side of (\ref{ltkkgv1}) vanishes. Assuming now that $u\in U$ and using
(\ref{gbk21v1})--(\ref{gbk23v1}) and (\ref{greenv1}), we obtain
$$
\int_a^b\tilde\varphi_\bot(t)\overline{(l_0(t)-C^*J_{H_0}u(t))}\,dt -c
=\int_a^b\tilde \varphi_\bot(t)\overline{L^+z(t;u)}\,dt-c
$$
$$
=\int_{a}^b L\tilde \varphi_\bot(t)\overline{z(t;u)}\, dt+\sum_{j=1}^m
B_j(\tilde \varphi_\bot)\overline{S_j^+(z(\cdot;u))}-c
$$
$$
 =\int_{a}^b \tilde f(t)\overline{z(t;u)}\, dt+\sum_{j=1}^m
\tilde \alpha_j\overline{S_j^+(z(\cdot;u))}-c
$$
$$
=(\tilde f,z(\cdot;u))_{L^2(a,b)}+(\tilde \alpha,\mathbf
S^+(z(\cdot;u)))_{\mathbb C^m}-c.
$$
Making use of the latter result together with (\ref{infwv1}) and (\ref{ltkkgv1}), we find
$$ \inf_{c \in
\mathbb C}\sup_{\tilde F\in G_0, \tilde \eta\in
G_1} M|l(\tilde
\varphi)-\widehat {l(\tilde \varphi)}|^2=
$$
\begin{equation}\label{exhv1}
=\inf_{c \in \mathbb C}\sup_{\tilde F\in G_0}\left|(\tilde
f,z(\cdot;u))_{L^2(a,b)}+(\tilde \alpha,\mathbf
S^+(z(\cdot;u)))_{\mathbb C^m}-c\right|^2+ \sup_{ \tilde \eta\in
G_1}M|(\tilde\eta,u)_{H_0}|^2.
\end{equation}
To calculate the first term on the right-hand side of (\ref{exhv1})
we apply the generalized Cauchy--Bunyakovsky inequality
(\ref{skk11v1}):
$$
\inf_{c \in \mathbb C}\sup_{\tilde F\in G_0,}\left|(\tilde
f,z(\cdot;u))_{L^2(a,b)}+(\tilde \alpha,\mathbf
S^+(z(\cdot;u)))_{\mathbb C^m}-c\right|^2
$$
\begin{multline*}
=\inf_{c \in \mathbb C}\sup_{\tilde F\in
G_0,}\left|\overline{(z(\cdot;u),\tilde
f-f_0)_{L^2(a,b)}}+\overline{(\mathbf S^+(z(\cdot;u)),\tilde
\alpha-\alpha^{(0)})_{\mathbb C^m}}\right.\\+
\left.(f_0,z(\cdot;u))_{L^2(a,b)}+( \alpha^{(0)},\mathbf
S^+(z(\cdot;u)))_{\mathbb C^m}-c\right|^2
\end{multline*}
\begin{multline*}
\leq \left\{(Q^{-1}z(\cdot;u),z(\cdot;u))_{L^2(a,b)}+
(Q_1^{-1}(\mathbf S^+(z(\cdot;u))), \mathbf
S^+(z(\cdot;u)))_{\mathbb C^m}\right\}\\ \times \left\{(Q(\tilde
f-f^{(0)}),\tilde f-f^{(0)})_{L^2(a,b)}+ (Q_1(\tilde \alpha
-\alpha^{(0)}),\tilde \alpha -\alpha^{(0)})_{\mathbb C^m}\right\}
\end{multline*}
\begin{equation}\label{ex2hv1} \leq
(Q^{-1}z(\cdot;u),z(\cdot;u))_{L^2(a,b)}+(Q_1^{-1}(\mathbf
S^+(z(\cdot;u))),\mathbf S^+(z(\cdot;u)))_{\mathbb C^m}.
\end{equation}
Performing direct substitution it is easy to check that inequality
(\ref{ex2hv1}) turns to equality at the element $\tilde F=(\tilde
f(\cdot),\tilde\alpha)=\tilde F^{(0)}:=(\tilde
f^{(0)}(\cdot),\tilde\alpha^{(0)})=\left(\tilde
f^{(0)}(\cdot),(\tilde\alpha_1^{(0)},\ldots,\right.$ $
\left.\tilde\alpha_m^{(0)})^T\right)\in L^2(a,b)\times \mathbb
C^m,$ where
$$
\tilde f^{(0)}(t):=\frac 1{d}Q^{-1}z(t,u))+f_0(t),
$$
$$
\tilde\alpha_i^{(0)}:=\frac 1{d}Q_1^{-1}\mathbf
S^+(z(\cdot;u))_i+\alpha_i^{(0)}, i=\overline{1,m},
$$
$$
d={\Bigl((Q^{-1}z(\cdot;u),z(\cdot;u))_{L^2(a,b)}+
(Q_1^{-1}\mathbf S^+(z(\cdot;u)),\mathbf S^+(z(\cdot;u)))_{\mathbb
C^m}\!\Bigr)^{1/2}},
$$
and $Q_1^{-1}(\mathbf S^+(z(\cdot;u)))_j$ is the $j$th component
of vector $Q_1^{-1}(\mathbf S^+(z(\cdot;u)))\in \mathbb C^m.$
Element $\tilde F^{(0)}\in G_0$ because, obviously, condition
(\ref{skk11v1})is fulfilled; in addition,
$$
(\tilde f^{(0)},\psi_0)_{L^2(a,b)}+\sum_{i=1}^m\tilde
\alpha_i^{(0)}\overline{S_i^+(\psi_0)}=
$$
$$
=\Bigl(Q^{-1}z(\cdot;u),z(\cdot;u))_{L^2(a,b)}+ (Q_1^{-1}(\mathbf
S^+(z(\cdot;u))),\mathbf S^+(z(\cdot;u)))_{\mathbb
C^m}\Bigr)^{-1/2}
$$
$$
\times \Bigl(Q^{-1}z(\cdot;u),\psi_0)_{L^2(a,b)}+ \sum_{i=1}^m
Q_1^{-1}(\mathbf S^+(z(\cdot;u)))_i\overline{S_i^+(\psi_0)}\Bigr)
$$
$$
+(f^{(0)},\psi_0)_{L^2(a,b)}+\sum_{i=1}^m
\alpha^{(0)}_i\overline{S_i^+(\psi_0)}=0 \quad\forall \psi_0\in
N(A^+_{B^+})
$$
which yields, by virtue of (\ref{gbk23v1}), the validity of
condition (\ref{skk10v1}). Therefore,
$$
\inf_{c \in \mathbb C}\sup_{\tilde F\in G_0}\left|(\tilde
f,z(\cdot;u))_{L^2(a,b)}+(\tilde \alpha,\mathbf
S^+(z(\cdot;u)))_{\mathbb C^m}-c\right|^2
$$
\begin{equation}\label{sk111hv1}
=\int_a^b Q^{-1}z(t;u)\overline{z(t;u)}dt +(Q_1^{-1}(\mathbf
S^+(z(\cdot;u))), \mathbf S^+(z(\cdot;u)))_{\mathbb C^m}
\end{equation}
at $c=\int_a^b\overline{z(t;u)} f^{(0)}(t)dt
+(\alpha^{(0)},\mathbf S^+(z(\cdot;u)))_{\mathbb C^m}.$

In order to calculate the second term in (\ref{exhv1}), note that
from the Cauchy--Bunyakovsky inequality (\ref{skk12v1}) it follows
that
$$
|(u,\tilde \eta)_{H_0}|^2\leq (Q_0^{-1}u,u)_{H_0}(Q_0\tilde
\eta,\tilde \eta)_{H_0},
$$
which yields
$$
\sup_{\tilde \eta\in G_1}M|(u,\tilde \eta)_{H_0}|^2\leq
(Q_0^{-1}u,u)_{H_0}.
$$
We have\label{L1v1}
$$
M|(u,\tilde \eta^{(0)})_{H_0}|^2=(Q_0^{-1}u,u)_{H_0}
$$
where $\tilde \eta^{(0)}=\nu
Q_0^{-1}u[(Q_0^{-1}u,u)_{H_0}]^{-1/2}\in G_1$ and $\nu$ is a
random quantity with $M\nu=0$ and $M|\nu|^2=1.$
Therefore\label{L2.2v1}
\begin{equation}\label{liehv1}
\sup_{\tilde \eta\in G_1}M|(u,\tilde
\eta)_{H_0}|^2=(Q_0^{-1}u,u)_{H_0}.
\end{equation}
Now the statement of Lemma 2.1 follows directly from relationships
(\ref{exhv1}), (\ref{sk111hv1}), and (\ref{liehv1}).
\end{proof}
\begin{pred}\label{th3v1}
The minimax estimate of $l(\varphi)$ can be represented as
\begin{equation}\label{minijv1}
\widehat{\widehat {l(\varphi)}}=(y,\hat
u)_{H_0}+\hat c,
\end{equation}
where
\begin{equation}\label{mjv1}
\hat u=Q_0Cp,\quad \hat c=\int_a^b\overline{z(t)}
f^{(0)}(t)dt+\sum_{i=1}^m \overline{S_i^+(z)}\alpha_i^{(0)}
\end{equation}
and functions $p(t)$ and $z(t)$ are determined from the
integro-differential equation system
\begin{equation}\label{llkyfv1}
L^+z(t)=l_0(t)-C^*J_{H_0}Q_0Cp(t)\quad \mbox{on}\quad (a,b),
\end{equation}
\begin{equation}\label{llk1yfv1}
B_j^+(z)=0,\quad j=\overline{1,2n-m},
\end{equation}
\begin{equation}\label{sgb30'fv1}
\int_a^bQ^{-1}z(t)\overline{\psi_i(t)}\, dt+ (Q_1^{-1}\mathbf
S^+(z), \mathbf S^+(\psi_i))_{\mathbb C^m}=0, \quad
i=\overline{1,m-r},
\end{equation}
\begin{equation}\label{sgb23'fv1}
Lp(t)=Q^{-1}z(t)\quad \mbox{on}\quad (a,b),
\end{equation}
\begin{equation}\label{slk1yfv1}
B_{j}(p)=Q_1^{-1}\mathbf S^+(z)_{j},\quad j=\overline{1,m},
\end{equation}
\begin{equation}\label{sgb25'fv1}
\int_a^b(l_0(t)-C^*J_{H_0}Q_0Cp(t))\overline{\varphi_{i}(t)}dt=0
,\,\, i=\overline{1,n-r},
\end{equation}
where $Q_1^{-1}\mathbf S^+(z)_j$ denotes the $j$th component of
vector $Q_1^{-1}\mathbf S^+(z)\in \mathbb C^m.$ Problem
\eqref{llkyfv1}--\eqref{sgb25'fv1} is uniquely solvable.
Estimation error $\sigma$ is determined by
$\sigma=l(p)^{1/2}.$
\end{pred}
\begin{proof}
Let us prove that the solution to the optimal control problem
(\ref{gbk21v1})--(\ref{gbk24''v1}) can be reduced to the solution
of system (\ref{llkyfv1})--(\ref{sgb25'fv1}). Show first that
there exists one and only one element  $\hat u \in U$ at which the
minimum of functional  (\ref{gbk24''v1}) is attained: $I(\hat
u)=\inf_{u \in U}I(u).$

For an arbitrary $u\in U$ set $u=\bar u+v$ where $\bar u$ is a
fixed element of $U$ and $v=u-\bar u$ belongs to the linear
subspace $V:=\{w\in H_0: \int_a^b C^*J_{H_0}w(t)
\overline{\varphi_0(t)}dt=0\quad \forall\varphi_0\in N(A_B)\}$ of
$H_0$. Let $z(t;\bar u)$ be the unique solution to the problem
\begin{equation}\label{igbbk21v1}
L^+z(t;\bar u)=l_0(t)-C^*J_{H_0}\bar u(t)\quad \mbox{on}\quad
(a,b),
\end{equation}
\begin{equation}\label{igbbk22v1}
B_j^+(z(\cdot;\bar u))=0 \quad j=\overline{1,2n-m},
\end{equation}
\begin{equation}\label{igbbk23v1}
\int_a^b Q^{-1}z(t;\bar u)\overline{\psi_i(t)} dt+
(Q_1^{-1}\mathbf S^+(z(\cdot;\bar u)),\mathbf
S^+(\psi_i))_{\mathbb C^m}=0,\quad i=\overline{1,m-r},
\end{equation}
and $\tilde z(t;v)$ the unique solution to the
problem\footnote{Existence and uniqueness of solution to this
problem follows from the condition $v\in V$ and the reasoning that
is used in the proof of the unique solvability of problem
(\ref{igbbk21v1})--(\ref{igbbk23v1}).}
\begin{equation}\label{igbb21xv1}
L^+(t)\tilde z(t;v)=-C^*J_{H_0}v\quad \mbox{on}\quad (a,b),
\end{equation}
\begin{equation}\label{igbb22xv1}
 B_j^+(\tilde z(\cdot;v))=0\quad  (j=1,\ldots,2n-m),
\end{equation}
\begin{equation}\label{igbb23xv1}
\int_a^b Q^{-1}\tilde z(t;v)\overline{\psi_i(t)}\,dt+ (\tilde
Q^{-1}\mathbf S^+(\tilde z(\cdot;v)),\mathbf S^+(\psi_i))_{\mathbb
C^m}=0,\quad i=\overline{1,m-r}.
\end{equation}
Solution $z(t;u)$ of BVP (\ref{gbk21v1})--(\ref{gbk24''v1}) can be
represented in the form
\begin{equation}\label{ieqv1}
z(t;u)=z(t;\bar u)+\tilde z(t;v);
\end{equation}
here if $v$ is an arbitrary element of space $V,$ then $u=\bar
u+v$ can be an arbitrary element of $U.$ Indeed, adding termwise
equalities (\ref{igbbk21v1})--(\ref{igbbk23v1}) to the
corresponding equalities (\ref{igbb21xv1})--(\ref{igbb23xv1}) we
obtain
\begin{equation*}
L^+(z(t;\bar u)+\tilde z(t;v))=l_0(t)-C^*J_{H_0}(\bar
u(t)+v(t))=l_0(t)-C^*J_{H_0}u(t)\quad \mbox{on}\quad (a,b),
\end{equation*}
\begin{equation*}
B_j^+(z(\cdot;\bar u)+\tilde z(\cdot;v))=0 \quad
j=\overline{1,2n-m},
\end{equation*}
\begin{equation*}
\int_a^b Q^{-1}(z(t;\bar u)+\tilde z(t;v))\overline{\psi_i(t)} dt
+(Q_1^{-1}(\mathbf S^+(z(\cdot;\bar u)+\tilde
z(\cdot;v)),\mathbf S^+(\psi_i))_{\mathbb C^m}=0,\quad
i=\overline{1,m-r}.
\end{equation*}
Equating these equalities with the corresponding ones from
(\ref{gbk21v1})--(\ref{gbk23v1}) we prove, taking into account the
unique solvability of problem (\ref{gbk21v1})--(\ref{gbk23v1})
proved on page \pageref{usv1}, the required representation of
solution in the form  (\ref{ieqv1}).

Prove that the solution to BVP
(\ref{igbb21xv1})--(\ref{igbb23xv1}) is continuously dependent on
$v\in V.$ Consider first operator  $A_B$ defined by (\ref{A_Bv1});
from Green's formula (\ref{greenv1}) and the reasoning similar to
that on pp. 90--91 in \cite{Krein}, it follows that the image
$\mathcal D(A_B^*)$ of operator $A_B^*$ adjoint to $A_B,$ is
formed by elements of the form
$(g(t),(S^+_1(g),\ldots,S^+_m(g))^T),$ where $g$ is an arbitrary
function from $W_2^n(a,b)$ satisfying the boundary conditions
\begin{equation}\label{adjv1}
B_j^+(g)=0\quad  j=1,\ldots,2n-m,
\end{equation}
and the operator $A_B^*:L^2(a,b)\times \mathbb C^m\to L^2(a,b)$
acts according to
\begin{equation}\label{adj1v1}
A_B^*(g(t),(S^+_1(g),\ldots,S^+_m(g))^T)=L^+g.
\end{equation}
Note that $A_B^*$ is a Noether operator because it is adjoint to a
Noether one, $A_B$; therefore,$A_B^*$ is a closed operator and the
equations
\begin{equation}\label{normsolvv1}
A_B^*(\tilde z_0(\cdot;v),(S^+_1(\tilde z_0(\cdot;v)),
\ldots,S^+_m(\tilde z_0(\cdot;v)))^T)=L^+\tilde
z_0(\cdot;v)=-C^*J_{H_0}v
\end{equation}
are normally solvable (that is, $R(A_B^*)=\overline{R(A_B^*)}$).
In line with (\ref{adjv1}) and (\ref{adj1v1}), the latter equation
is equivalent to BVP (\ref{igbb21xv1})--(\ref{igbb22xv1}). Using
this fact and applying Theorem 2.3 from
\cite{Krein}\footnote{Formulate this theorem.
\begin{predd} The equation
$$Ax=y$$ with a closed operator $A$ acting from a Hilbert space $E$
to a Hilbert space $F$ is normally solvable (i.e.
$R(A)=\overline{R(A)}$) if and only if for every $y\in R(A)$ there
is an $x\in\mathcal D(A)$ such that $y=Ax$ and $\|x\|\leq k\|Ax\|
=k\|y\|$ where $k>0$ and is independent of $y\in R(A).$
\end{predd}} to equation (\ref{normsolvv1})
we obtain that for every $v\in V$ there is a solution $\tilde
z_0(t;v)\in W_2^n(a,b)$ of problem
(\ref{igbb21xv1})--(\ref{igbb22xv1}) such that
\begin{multline*}
\|(\tilde z_0(\cdot;v),(S^+_1(\tilde z_0(\cdot;v)),
\ldots,S^+_m(\tilde z_0(\cdot;v)))^T)\|_{L^2(a,b)\times \mathbb
C^m}\\ \leq a_1\|L^+\tilde
z_0(\cdot;v)\|_{L^2(a,b)}=a_1\|C^*J_{H_0}v\|_{L^2(a,b)}\leq
a\|v\|_{H_0},
\end{multline*}
or, equivalently,
\begin{equation}\label{contv1}
\left( \int_a^b|\tilde z_0(t;v)|^2dt+\sum_{i=1}^m|S^+_i(\tilde
z_0(\cdot;v))|^2 \right)^{1/2}\leq a\|v\|_{H_0},
\end{equation}
where $a$ and $a_1$ are constants independent of $v.$

Proceeding similarly to the proof on p. \pageref{20v1}, we
conclude that the unique solution $\tilde z(t;v)$ to BVP
(\ref{igbb21xv1})--(\ref{igbb23xv1}) can be represented in the
form
\begin{equation}\label{comsolv1}
\tilde z(t;v)=\tilde z_0(t;v)+\sum_{i=1}^{m-r}c_i\psi_i(t),
\end{equation}
where the coefficients $c_i=c_i(v)\in \mathbb C$ are uniquely
determined from the linear algebraic equation system
\begin{equation}\label{Slauv1}
\sum_{i=1}^{m-r}a_{ij}c_i=b_j(v) ,\quad j=1,\ldots,m-r
\end{equation}
by the formulas
\begin{equation}\label{frv1}
c_i(v)=\frac{d_i(v)}{d},
\end{equation}
where
$$
d=\left|
\begin{array}{ccccccc}
a_{1,1} &\cdots & a_{i-1,1} & a_{i,1} & a_{i+1,1}&\cdots &
a_{m-r,1}\\ a_{1,2} &\cdots & a_{i-1,2}& a_{i,2}& a_{i+1,2}&\cdots
& a_{m-r,2}\\ \hdotsfor[2]{7}\\ a_{1,m-r} & \cdots & a_{i-1,m-r}&
a_{i,m-r} &a_{i+1,m-r}& \cdots & a_{m-r,m-r}
\end{array}
\right|,
$$
$$
d_i(v)=\left|
\begin{array}{ccccccc}
a_{1,1} &\cdots & a_{i-1,1} & b_1(v) & a_{i+1,1}&\cdots &
a_{m-r,1}\\ a_{1,2} &\cdots & a_{i-1,2}& b_2(v)& a_{i+1,2}&\cdots
& a_{m-r,2}\\ \hdotsfor[2]{7}\\ a_{1,m-r} & \cdots & a_{i-1,m-r}&
b_{m-r}(v) &a_{i+1,m-r}& \cdots & a_{m-r,m-r}
\end{array}
\right|
$$
elements $a_{i,j},$ $i,j=1,\ldots,m-r,$ of positive definite
matrix $[a_{i,j}]$ are determined from (\ref{a_ijv1}), and
$$
b_j(v)=-\int_a^b Q^{-1}\tilde z_0(t;v)\overline{\psi_j(t)} dt-
(Q_1^{-1}\mathbf S^+(\tilde z_0(\cdot;v)),\mathbf
S^+(\psi_j))_{\mathbb C^m},\,\,j=\overline{1,m-r}.
$$

Expanding determinant $d_i(v)$ entering the right-hand side of
(\ref{frv1}) in elements of the $i$th column, we have
$$
c_i(v)=\gamma_1^{(i)}b_1(v)+\gamma_2^{(i)}b_2(v)
+\cdots+\gamma_{m-r}^{(i)}b_{m-r}(v),\,\,i=1,\ldots,m-r,
$$
here
$$
\gamma_j^{(i)}=\frac{A_{ji}}{d} ,\quad i,j=1\ldots,m-r,
$$
are constants independent of $v$, where $A_{ji}$ is the algebraic
complement of the element of the $i$th column that enters the
$j$th row of determinant $d_i(v)$ which is independent of $v.$
Applying the generalized Cauchy--Bunyakovsky inequality to the
representation for $c_i(v)$, we obtain
$$
|c_i(v)|\leq \sum_{j=1}^{m-r}|\gamma_j^{(i)}||b_j(v)|
=\sum_{j=1}^{m-r}|\gamma_j^{(i)}|\Bigl|\int_a^b Q^{-1}\tilde
z_0(t;v)\overline{\psi_j(t)}\,dt\Bigr.
$$
$$
\Bigl.+ (Q_1^{-1}\mathbf S^+(\tilde z_0(\cdot;v)),\mathbf
S^+(\psi_j))_{\mathbb C^m}\Bigr|\leq
\sum_{j=1}^{m-r}|\gamma_j^{(i)}|\Bigl|\int_a^b Q^{-1}\tilde
z_0(t;v)\overline{\tilde z_0(t;v)}\,dt\Bigr.
$$
$$
\Bigl.+ (Q_1^{-1}\mathbf S^+(\tilde z_0(\cdot;v)),\mathbf
S^+(\tilde z_0(\cdot;v)))_{\mathbb C^m}\Bigr|^{1/2}
$$
$$ \times
\Bigl|\int_a^b Q^{-1}\psi_j(t)\overline{\psi_j(t)}\,dt\Bigr.
\Bigl.+ (Q_1^{-1}\mathbf S^+(\psi_j),\mathbf S^+(\psi_j))_{\mathbb
C^m}\Bigr|^{1/2}
$$
\begin{equation}\label{in1v1}
=A_i\Bigl|\int_a^b Q^{-1}\tilde z_0(t;v)\overline{\tilde
z_0(t;v)}\,dt\Bigr. \Bigl.+ (Q_1^{-1}\mathbf S^+(\tilde
z_0(\cdot;v)),\mathbf S^+(\tilde z_0(\cdot;v)))_{\mathbb
C^m}\Bigr|^{1/2}
\end{equation}
where
$$
A_i=\sum_{j=1}^{m-r}|\gamma_j^{(i)}|\Bigl|\int_a^b
Q^{-1}\psi_j(t)\overline{\psi_j(t)}\,dt\Bigr. \Bigl.+
(Q_1^{-1}\mathbf S^+(\psi_j),\mathbf S^+(\psi_j))_{\mathbb
C^m}\Bigr|^{1/2}, \,\,i=\overline{1,m-r},
$$
are constants independent of $v.$ Next, taking into consideration
the following estimate obtained with the help of the
Cauchy--Bunyakovsky inequality and (\ref{contv1})
$$
\Bigl|\int_a^b Q^{-1}\tilde z_0(t;v)\overline{\tilde
z_0(t;v)}\,dt\Bigr. \Bigl.+ (Q_1^{-1}\mathbf S^+(\tilde
z_0(\cdot;v)),\mathbf S^+(\tilde z_0(\cdot;v)))_{\mathbb
C^m}\Bigr|^{1/2}
$$
$$
\leq (\|Q^{-1}\tilde z_0(\cdot;v)\|_{L^2(a,b)}\|\tilde
z_0(\cdot;v)\|_{L^2(a,b)}+\|Q_1^{-1}\mathbf S^+(\tilde
z_0(\cdot;v))\|_{\mathbb C^m}\|\mathbf S^+(\tilde
z_0(\cdot;v))\|_{\mathbb C^m})^{1/2}
$$
$$
\leq (\|Q^{-1}\|\|\tilde
z_0(\cdot;v)\|^2_{L^2(a,b)}+\|Q_1^{-1}\|\|\mathbf S^+(\tilde
z_0(\cdot;v))\|^2_{\mathbb C^m})^{1/2}
$$
$$
\leq \max\{\|Q^{-1}\|^{1/2},\|Q_1^{-1}\|^{1/2}\}\left(\int_a^b
|\tilde z_0(t;v)|^2dt+\sum_{i=1}^m |S_i^+(\tilde z_0(\cdot;v))|^2
\right)^{1/2}
$$
\begin{equation}\label{in2v1}
\leq a\max\{\|Q^{-1}\|^{1/2},\|Q_1^{-1}\|^{1/2}\}\|v\|_{H_0},
\end{equation}
where constant $a$ enters the right-hand side of (\ref{contv1}).

Estimates (\ref{in1v1}) and (\ref{in2v1}) yield the inequality
\begin{equation}\label{in3v1}
|c_i(v)|\leq C_i\|v\|_{H_0}, \,\,i=1,\ldots,m-r,
\end{equation}
where $C_i=A_ia\max\{\|Q^{-1}\|^{1/2},\|Q_1^{-1}\|^{1/2}\}$ are
constants independent of $v.$

Using inequalities (\ref{in3v1}), (\ref{contv1}) and
representation (\ref{comsolv1}) of solution $\tilde z(t;v)$ to BVP
(\ref{igbb21xv1})--(\ref{igbb23xv1}) we will prove that this
solution satisfies the inequality
\begin{equation}\label{cont1v1}
\left( \int_a^b|\tilde z(t;v)|^2dt+\sum_{i=1}^m|S^+_i(\tilde
z(\cdot;v))|^2 \right)\leq K\|v\|_{H_0}^2,
\end{equation}
where $K$ is a constant independent of $v.$ Taking into notice the
inequality $\|a+b\|^2\leq 2(\|a\|^2+\|b\|^2)$ which is valid for
any elements $a$ and $b$ from a normed space, we have
$$
\int_a^b|\tilde z(t;v)|^2dt+\sum_{i=1}^{m}|S^+_i(\tilde
z(\cdot;v))|^2=\int_a^b|\tilde z(t;v)|^2dt+\|\mathbf
S^+(z(\cdot;v))\|_{\mathbb C^m}^2
$$
$$
=\int_a^b|\tilde z_0(t;v)+\sum_{i=1}^{m-r}c_i(v)\psi_i(t)|^2dt
+\|\mathbf
S^+(z_0(\cdot;v)+\sum_{i=1}^{m-r}c_i(v)\psi_i(\cdot))\|_{\mathbb
C^m}^2
$$
$$
\leq 2\left(\int_a^b|\tilde
z_0(t;v)|^2dt+\sum_{i=1}^{m-r}|c_i(v)|^2\int_a^b|\psi_i(t)|^2dt\right)
$$
$$
+2\left(\|\mathbf S^+(z_0(\cdot;v)\|^2_{\mathbb C^m}+
\sum_{i=1}^{m-r}|c_i(v)|^2\|\mathbf S^+(\psi_i)\|^2_{\mathbb
C^m}\right)
$$
$$
=2\left(\int_a^b|\tilde z_0(t;v)|^2dt+\sum_{i=1}^{m}|S^+_i(\tilde
z_0(\cdot;v))|^2\right)
$$
$$
+2\sum_{i=1}^{m-r}|c_i(v)|^2\left(\int_a^b|\psi_i(t)|^2dt+\|\mathbf
S^+(\psi_i)\|^2_{\mathbb C^m} \right)
$$
$$
\leq 2a^2\|v\|^2_{H_0}+2\sum_{i=1}^{m-r}C_i^2\|v\|^2_{H_0}
\left(\int_a^b|\psi_i(t)|^2dt+\|\mathbf S^+(\psi_i)\|^2_{\mathbb
C^m} \right)=K\|v\|^2_{H_0},
$$
where
$$
K=2a^2+2\sum_{i=1}^{m-r}C_i^2\left(\int_a^b|\psi_i(t)|^2dt+\|\mathbf
S^+(\psi_i)\|^2_{\mathbb C^m} \right).
$$
Thus, inequality (\ref{cont1v1}) is proved.

Reduce now the minimization of functional $I(u)$ on set $U$ to the
problem of finding the minimum of the functional
$$
I_V(v):=I(\bar u+v)
$$
on a linear subspace $V$ of space $H_0.$ To this end, make use of
representation (\ref{ieqv1}) for $z(t;u)$ and write $I(u)$ in the
form
$$
I(u)=\int_a^bQ^{-1}z(t;u)\overline{z(t;u)}dt +(Q_1^{-1}\mathbf
S^+(z(\cdot;u)), \mathbf S^+(z(\cdot;u))_{\mathbb C^m}+
(Q_0^{-1}u,u)_{H_0}
$$
$$
=\int_a^bQ^{-1}(z(t;\bar u)+\tilde z(t;v))\overline{(z(t;\bar
u)+z(t;u))}dt
$$
$$
+(Q_1^{-1}(\mathbf S^+(z(\cdot;\bar u))+\tilde z(\cdot;v)),
\mathbf S^+(z(\cdot;\bar u)+\tilde z(\cdot;v)))_{\mathbb C^m}+
(Q_0^{-1}(\bar u+v),\bar u+v)_{H_0}
$$
$$
=I(\bar u)+\int_a^bQ^{-1}\tilde z(t;v)\overline{\tilde z(t;v)}dt+
(Q_1^{-1}\mathbf S^+(\tilde z(\cdot;v)), \mathbf S^+(\tilde
z(\cdot;v)))_{\mathbb C^m}+ (Q_0^{-1}v,v)_{H_0}
$$
$$
+2\mbox{Re}\int_a^bQ^{-1}\tilde z(t;v)\overline{z(t;\bar u)}dt
+2\mbox{Re}(Q_1^{-1}\mathbf S^+(\tilde z(\cdot;v)), \mathbf
S^+(z(\cdot;\bar u)))_{\mathbb C^m}
$$
\begin{equation}\label{repv1}
+2\mbox{Re}(Q_0^{-1}v,\bar u)_{H_0}=I(\bar u)+\tilde
I(v)+2\mbox{Re}L(v),
\end{equation}
where, by virtue of estimate (\ref{cont1v1}),
\begin{equation}\label{rep1v1}
\tilde I(v)=\int_a^bQ^{-1}\tilde z(t;v)\overline{\tilde z(t;v)}dt+
(Q_1^{-1}\mathbf S^+(\tilde z(\cdot;v)), \mathbf S^+(\tilde
z(\cdot;v)))_{\mathbb C^m}+ (Q_0^{-1}v,v)_{H_0}
\end{equation}
is a quadratic functional in $V$ associated with the semi-bilinear
continuous Hermitian form\label{cont2v1}
\begin{equation}
\label{rep2v1}
\pi(v,w)=\int_a^bQ^{-1}\tilde z(t;v)\overline{\tilde z(t;w)}dt
+(Q_1^{-1}\mathbf S^+(\tilde z(\cdot;v)), \mathbf S^+(\tilde
z(\cdot;w)))_{\mathbb C^m}+ (Q_0^{-1}v,w)_{H_0}
\end{equation}
on $V\times V;$ the functional satisfies the inequality
\begin{equation}\label{rep3v1}
\tilde I(v)\geq c \|v\|^2_{H_0}\,\,\forall v\in
V,\,\,c=\mbox{const},
\end{equation}
and
\begin{equation}\label{rep4v1}
L(v)=\int_a^bQ^{-1}\tilde z(t;v)\overline{z(t;\bar u)}dt\Bigr.
\Bigl. +(Q_1^{-1}\mathbf S^+(\tilde z(\cdot;v)), \mathbf
S^+(z(\cdot;\bar u)))_{\mathbb C^m}+(Q_0^{-1}v,\bar u)_{H_0}
\end{equation}
is a linear continuous functional in $V.$

Prove, for example, the continuity of form (\ref{rep2v1}), that
is, the inequality
\begin{equation}\label{rep5v1}
\pi(v,w)\leq c  \|v\|_{H_0}  \|w\|_{H_0} \quad\forall v,w\in
V,\,\,c=\mbox{const}
\end{equation}
(the continuity of linear functional $L(v)$ can be proved in a
similar manner). Using estimate (\ref{cont1v1}) and the
Cauchy--Bunyakovsky inequality, we have
$$
|\pi(v,w)| \leq\left(\int_a^bQ^{-1}\tilde z(t;v)\overline{\tilde
z(t;v)}dt\right)^{1/2}\left(\int_a^bQ^{-1}\tilde
z(t;w)\overline{\tilde z(t;w)}dt\right)^{1/2}
$$
$$
+\left(Q_1^{-1}\mathbf S^+(\tilde z(\cdot;v)), \mathbf S^+(\tilde
z(\cdot;v))\right)_{\mathbb C^m}^{1/2}\left(Q_1^{-1}\mathbf
S^+(\tilde z(\cdot;w)), \mathbf S^+(\tilde
z(\cdot;w))\right)_{\mathbb C^m}^{1/2}
$$
$$
+(Q_0^{-1}v,v)_{H_0}^{1/2}(Q_0^{-1}w,w)_{H_0}^{1/2}
$$
$$
\leq \left(\int_a^bQ^{-1}\tilde z(t;v)\overline{\tilde z(t;v)}dt+
(Q_1^{-1}\mathbf S^+(\tilde z(\cdot;v)), \mathbf S^+(\tilde
z(\cdot;v)))_{\mathbb C^m}+(Q_0^{-1}v,v)_{H_0}\right)^{1/2}
$$
$$
\times \left(\int_a^bQ^{-1}\tilde z(t;w)\overline{\tilde
z(t;w)}dt+ (Q_1^{-1}\mathbf S^+(\tilde z(\cdot;w)), \mathbf
S^+(\tilde z(\cdot;w)))_{\mathbb
C^m}+(Q_0^{-1}w,w)_{H_0}\right)^{1/2}
$$
\begin{multline*}
\leq \Biggl\{\left(\int_a^b|Q^{-1}\tilde z(t;v)|^2dt\right)^{1/2}
\left(\int_a^b|\tilde z(t;v)|^2dt\right)^{1/2}\Biggr.\\ \Biggl.+
\|Q_1^{-1}\mathbf S^+(\tilde z(\cdot;v))\|_{\mathbb C^m}\|\mathbf
S^+(\tilde z(\cdot;v))\|_{\mathbb
C^m}+\|Q_0^{-1}v\|_{H_0}\|v\|_{H_0}\Biggr\}^{1/2}
\end{multline*}
\begin{multline*}
\times \{\left(\int_a^b|Q^{-1}\tilde z(t;w)|^2dt\right)^{1/2}
\left(\int_a^b|\tilde z(t;w)|^2dt\right)^{1/2}\Biggr.\\ \Biggl.+
\|Q_1^{-1}\mathbf S^+(\tilde z(\cdot;w))\|_{\mathbb C^m}\|\mathbf
S^+(\tilde z(\cdot;w))\|_{\mathbb
C^m}+\|Q_0^{-1}w\|_{H_0}\|w\|_{H_0}\Biggr\}^{1/2}
\end{multline*}
$$
\leq \max\{\|Q^{-1}\|,\|Q_1^{-1}\|,\|Q_0^{-1}\|\}
\Biggl\{\int_a^b|\tilde z(t;v)|^2dt+ \|\mathbf S^+(\tilde
z(\cdot;v))\|^2_{\mathbb C^m}+\|v\|^2_{H_0}\Biggr\}^{1/2}
$$
$$
\times \Biggl\{\int_a^b|\tilde z(t;w)|^2dt+ \|\mathbf S^+(\tilde
z(\cdot;w))\|^2_{\mathbb C^m}+\|w\|^2_{H_0}\Biggr\}^{1/2}
$$
$$
\leq \max\{\|Q^{-1}\|,\|Q_1^{-1}\|,\|Q_0^{-1}\|\}(K\|v\|^2_{H_0}
+\|v\|^2_{H_0})^{1/2}(K\|w\|^2_{H_0} +\|w\|^2_{H_0})^{1/2}
$$
$$
\leq c \|v\|_{H_0}  \|w\|_{H_0},
$$
where
$$
c=\max\{\|Q^{-1}\|,\|Q_1^{-1}\|,\|Q_0^{-1}\|\}(K+1).
$$
Thus, we have proved inequality (\ref{rep5v1}) and consequently
the continuity of form (\ref{rep2v1}).

Taking into consideration the continuity of (\ref{rep2v1}) and
Remark 1.1 to Theorem 1.1 from \cite{BIBL20} \label{d2v1} we see
that there exists the unique element  $\hat v\in V$ (dependent on
$\bar u$) such that
\begin{equation*}
I_V(\hat v)=\inf_{v \in V}I_V(v), =\inf_{v \in V} I(\bar
u+v)=\inf_{u-\bar u \in V} I(u)= \inf_{u \in \bar u+V}
I(u)=\inf_{u \in U} I(u).
\end{equation*}
Setting $\hat u=\bar u+\hat v$ and using the equality
$$
I_V(\hat v)=I(\bar u+\hat v)=I(\hat u)
$$
we conclude that there exists one and only one element $\hat
u=\bar u+\hat v,$ $\hat u\in U,$ such that functional  $I(u)$
attaints the minimum at  $u\in U$. Therefore, for any $\tau\in R$
and $v\in V$,
\begin{equation}\label{z43v1}
\frac{d}{d\tau}I(\hat u+\tau
v)\Bigl.\Bigr|_{\tau=0}=0\quad\mbox{and}\quad
\frac{d}{d\tau}I(\hat u+i\tau v)\Bigl.\Bigr|_{\tau=0}=0,
\end{equation}
where $i=\sqrt{-1}.$ Since
$z(t;\hat u+\tau v)=z(t;\hat u)+\tau\tilde z(t;v),$ where $\tilde
z(t;v)$ is the unique solution to BVP
(\ref{gbk21v1})--(\ref{gbk23v1}) at $u=v$ and $l_0=0,$ from the
first relationship (\ref{z43v1}) we obtain
$$
0=\frac 1{2}\frac{d}{d\tau}I(\hat u+\tau v)|_{\tau=0}
$$
\begin{equation*} =\lim_{\tau\to 0}\frac
1{2\tau}\Bigl\{\Bigl[(Q^{-1}z(\cdot;\hat u+\tau v),z(\cdot;\hat
u+\tau v))_{L^2(a,b)}
-(Q^{-1}z(\cdot;\hat
u),z(\cdot;\hat u))_{L^2(a,b)}\Bigr]\Bigr.
\end{equation*}
$$
+\Bigl[(Q_1^{-1}(\mathbf S^+(z(\cdot;\hat u+\tau v))),\mathbf
S^+(z(\cdot;\hat u+\tau v)))_{\mathbb C^m}-\Bigr.
\Bigl.(Q_1^{-1}(\mathbf S^+(z(\cdot;\hat u))),\mathbf
S^+(z(\cdot;\hat u)))_{\mathbb C^m}\Bigr]
$$
$$
\Bigl.+\left[Q^{-1}_0(\hat u+\tau v),\hat u+\tau v)_{H_0}-(Q^{-1}_0\hat
u,\hat u)_{H_0}\right]\Bigr\}
$$
$$
=\mbox{Re}\{(Q^{-1}z(\cdot;\hat u),\tilde z(\cdot;v
))_{L^2(a,b)}+(Q_1^{-1}(\mathbf S^+(z(\cdot;\hat u))),\mathbf
S^+(\tilde z(\cdot;v)))_{\mathbb C^m}+(Q_0^{-1}\hat u,v)_{H_0}\}.
$$
Similarly, taking into account that $z(t;\hat u+i\tau v)=z(t;\hat
u)+i\tau\tilde z(t;v),$
$$
0=\frac 1{2}\frac{d}{d\tau}I(\hat u+i\tau v)|_{\tau=0}
$$
$$
=\mbox{Im}\{(Q^{-1}z(\cdot;\hat u),\tilde z(\cdot;v
))_{L^2(a,b)}+(Q_1^{-1}(\mathbf S^+(z(\cdot;\hat u))),\mathbf
S^+(\tilde z(\cdot;v)))_{\mathbb C^m}\!+\!(Q_0^{-1}\hat
u,v)_{H_0}\},
$$
which yields
\begin{equation}\label{vcv1}
(Q^{-1}z(\cdot;\hat u),\tilde z(\cdot;v ))_{L^2(a,b)}+\sum_{j=1}^m
Q_1^{-1}(\mathbf S^+(z(\cdot;\hat u)))_{j}\overline{\mathbf
S_j^+(\tilde z(\cdot;v))}+(Q_0^{-1}\hat u,v)_{H_0}=0.
\end{equation}
Let $p(t)$ be a solution to the BVP\footnote{Relationship
(\ref{sgb30'fv1}) coincides with the solvability condition for
this problem by virtue of (\ref{skk5v1}).}
\begin{equation}\label{1vsgbk23'v1}
Lp(t)=Q^{-1}z(t;\hat u)\quad \mbox{on}\quad (a,b),
\end{equation}
\begin{equation}\label{1vsllk1ykv1}
B_{j}(p)=Q_1^{-1}(\mathbf S^+(z(\cdot,\hat u)))_{j},\quad
j=\overline{1,m}.
\end{equation}
Then the sum of the first two terms on the left-hand side of
(\ref{vcv1}) can be written, in view of Green's formula
(\ref{greenv1}), in the form
$$
(Q^{-1}z(\cdot;\hat u),\tilde z(\cdot;v ))_{L^2(a,b)}+\sum_{j=1}^m
Q_1^{-1}(\mathbf S^+(z(\cdot;\hat u)))_{j}\overline{\mathbf
S_j^+(\tilde z(\cdot;v))}
$$
$$
=\int_a^b Lp(t)\overline{\tilde z(t;v)}\,dt+\sum_{j=1}^m B_{j}(p)
\overline{\mathbf S_j^+(\tilde z(\cdot;v))}=\int_a^b
p(t)\overline{L^+\tilde z(t;v)}\,dt
$$
$$
=-\int_a^b p(t)\overline{C^*J_{H_0}v(t)}\,dt
=-<Cp,J_{H_0}v>_{H_0\times H_0'}=-(Cp,v)_{H_0}.
$$
From the latter equality and formula (\ref{vcv1}), it follows that
for any $v\in V$,
\begin{equation}\label{gfklhv1}
(Q_0^{-1}\hat u-Cp,v)_{H_0}=0.
\end{equation}
Let us show that in the set of solutions to problem
(\ref{1vsgbk23'v1}), (\ref{1vsllk1ykv1}) there is only one,
$p(t),$ for which
\begin{equation}\label{1gfk2v1}
Q_0^{-1}\hat u-Cp\in V.
\end{equation}
Indeed, condition (\ref{1gfk2v1}) means that for any $1\leq i\leq
n-r$ the equalities
\begin{equation}\label{1gfk3v1}
\int_a^b\varphi_i(t)\overline{C^*J_{H_0}(Q_0^{-1}\hat
u-Cp)(t)}\,dt=0
\end{equation}
hold. Since general solution $p(t)$ to BVP\label{reppv1}
(\ref{1vsgbk23'v1}), (\ref{1vsllk1ykv1}) has the form
$$
p(t)=\tilde p(t)+\sum_{j=1}^{n-r}a_j\varphi_j(t),
$$
where $\tilde p(t)$ is a particular solution to this problem and
$a_j\in\mathbb C$ $(j=\overline{1,n-r})$ are arbitrary numbers, we
conclude that in line with (\ref{1gfk3v1}), function $p(t)$
satisfies condition (\ref{1gfk2v1}) if $(a_1,\dots,a_{n-r})^T$ is
a solution to the uniquely solvable linear algebraic equation
system
$$
\sum_{i=1}^{n-r}a_i(C\varphi_i,C\varphi_j)_{H_0} =(Q_0^{-1}\hat
u-C\tilde p,C\varphi_j)_{H_0},\quad j=\overline{1,n-r},
$$
where matrix $[(C\varphi_i,C\varphi_j)_{H_0}]_{i,j=1}^{n-r}$ has a
non-zero determinant bacause it is the Gram matrix of the system
of linearly independent elements
$C\varphi_1,\dots,C\varphi_{n-r}$. It is easy to see that the
unique solvability of this system yields the existence of the
unique function
 $p(t)$ that satisfies condition
(\ref{1gfk2v1}) and equations (\ref{1vsgbk23'v1}) and
(\ref{1vsllk1ykv1}).

Setting in (\ref{gfklhv1}) $v=Q_0^{-1}\hat u-Cp$
we have $Q_0^{-1}\hat u-Cp=0,$ so that $\hat u=Q_0Cp.$
Substituting the latter into $\int_a^b (l_0(t)-C^*J_{H_0}\hat
u(t)) \overline{\varphi_0(t)}dt=0$ and denoting $z(t;\hat
u)=:z(t),$ we see that $z(t)$ and $p(t)$ satisfy system
(\ref{llkyfv1})--(\ref{sgb25'fv1});  the unique solvability of
this system follows from the uniqueness of element $\hat u.$

Prove now that $\sigma(\varphi)\leq\sigma(\hat u,\hat
c)=l(p)^{1/2}.$ Substituting $\hat u=Q_0Cp$ into $I(\hat u)$ and
taking into account the designation $z(t)=z(t;\hat u),$ we obtain
$$
I(\hat u)=\int_a^b Q^{-1}z(t)\overline{z(t)}\,dt+(Q_1^{-1} \mathbf
S^+(z),\mathbf S^+(z))_{C^m}
$$
$$
+(Cp,Q_0Cp)_{H_0}=\int_a^bLp(t)\overline{z(t)}\,dt
$$
$$
+ \sum_{j=1}^m(Q_1^{-1}\mathbf
S^+(z))_j(\overline{S_j^+(z)})+(Cp,Q_0Cp)_{H_0}
$$
$$
=\int_a^bLp(t)\overline{z(t)}\,dt +\sum_{j=1}^m
B_j(p)\overline{S_j^+(z)}+(Cp,Q_0Cp)_{H_0}
$$
\begin{equation}\label{ww1gv1}
=\int_a^bp(t)\overline{L^+z(t)}\,dt+\sum_{j=1}^{2n-m}S_j(p)
\overline{B_j^+(z)} +(Cp,Q_0Cp)_{H_0}.
\end{equation}
For the first term in (\ref{ww1gv1}) we have
$$
\int_a^bp(t)\overline{L^+z(t)}\,dt
=\int_a^bp(t)\overline{l_0(t)}\,dt
-\int_a^bp(t)(\overline{C^*J_{H_0}Q_0Cp)(t)})\,dt
$$
$$
=\int_a^bp(t)\overline{l_0(t)}\,dt-<Cp,J_{H_0}Q_0Cp>_{H_0\times
H_0'}.
$$
The later equality together with (\ref{ww1gv1}) yield $I(\hat
u)=l(p).$ The theorem is proved.
\end{proof}

\begin{pred}\label{th4v1}
The minimax estimate of $l(\varphi)$ has the form
$$
\widehat{\widehat {l(\varphi)}}=l(\hat\varphi),
$$
where function $\hat\varphi$ is determined from the solution to
the problem
\begin{equation}\label{ugb42hv1}
L^+\hat p(t)=C^*J_{H_0}Q_0(y-C\hat\varphi)(t)\quad
\mbox{on}\quad (a,b),
\end{equation}
\begin{equation}\label{ugb422hv1}
B_j^+(\hat p)=0,\quad j=\overline{1,2n-m},
\end{equation}
\begin{equation}\label{ugb44hv1}
\int_a^bQ^{-1}\hat p(t)\overline{\psi_i(t)}\, dt+ (Q_1^{-1}\mathbf
S^+(\hat p),\mathbf S^+(\psi_i))_{\mathbb C^m }=0, \quad
i=\overline{1,m-r},
\end{equation}
\begin{equation}\label{ugb46hv1}
L\hat \varphi(t)=Q^{-1}\hat p(t)+f^{(0)}(t)\quad \mbox{on}\quad
(a,b),
\end{equation}
\begin{equation}\label{usllk1yhv1}
B_{j}(\hat \varphi)=Q_1^{-1}\mathbf S^+(\hat
p)_{j}+\alpha_{j}^{(0)},\quad
 j=\overline{1,m},
\end{equation}
\begin{equation}\label{ugb48hv1}
\int_a^bC^*J_{H_0}Q_0\left(y-C\hat \varphi\right)(t)
\overline{\varphi_{i}(t)}\,dt=0,\quad i=\overline{1,n-r}.
\end{equation}
Problem \eqref{ugb42hv1}--\eqref{ugb48hv1} is uniquely solvable.
\end{pred}
\begin{proof}
Consider the problem of optimal control of the equation system
\begin{equation}\label{gbbk21gv1}
L^+\hat
p(t;u)=-(C^*J_{H_0}u)(t)+(C^*J_{H_0}Q_0y)(t)\quad
\mbox{on}\quad (a,b),
\end{equation}
\begin{equation}\label{gbbk22gv1}
 B_j^+(\hat p(\cdot;u))=0\quad  (j=1,\ldots,m),
\end{equation}
\begin{equation}\label{gbb23gv1}
\int_a^b Q^{-1}\hat p(t;u)\overline{\psi_i(t)} dt+
(Q_1^{-1}\mathbf S^+(\hat p(\cdot;u)),\mathbf
S^+(\psi_i))_{\mathbb C^m}=0,\quad i=\overline{1,m-r},
\end{equation}
with the cost function
\begin{multline}\label{gb24'gv1}
I(u)=\int_a^b Q^{-1}(\hat p(t;u)+Qf^{(0)}(t))\overline{ (\hat
p(t;u)+Qf^{(0)}(t))}dt \\+(Q_1^{-1}(\mathbf S^+(\hat
p(\cdot;u))+Q_1\alpha^{(0)}), \mathbf S^+(\hat
p(\cdot;u))+Q_1\alpha^{(0)})_{\mathbb C^m}+
(Q_0^{-1}u,u)_{H_0}\!\to \inf_{u\in \tilde U},
\end{multline}
where $ \tilde U=\{u\in H_0:
\int_a^b(C^*J_{H_0}Q_0y)(t) -C^*J_{H_0}u(t))
\overline{\varphi_0(t)}dt=0 \}$ for arbitrary solutions
$\varphi_0(t)$ of homogeneous BVP (\ref{skkv1}), (\ref{skk1v1}).

The form of functional $I(u)$ and the reasoning contained in the
proof of Theorem 2.1 suggest the existence of the unique element
$\hat u\in \tilde U$ such that
$$
I(\hat u)= \inf_{u\in \tilde U} I(u).
$$
Next, finding $\hat \varphi(t)$ as the unique solution to the BVP
$$
L\hat \varphi(t)=Q^{-1}\hat p(t;\hat u)+f^{(0)}(t)\quad
\mbox{on}\quad (a,b),
$$
$$
B_{j}(\hat \varphi)=Q_1^{-1}\mathbf S^+(\hat p(\cdot;\hat
u)_{j}+\alpha_{j}^{(0)},\quad
 j=\overline{1,m},
$$
$$
\int_a^bC^*J_{H_0}Q_0\left(y-C\hat \varphi\right)(t)
\overline{\varphi_{i}(t)}\,dt=0,\quad i=\overline{1,n-r},
$$
and denoting $\hat p(t)=\hat p(t;\hat u),$ we conclude, repeating
virtually the proof of Theorem 2.1,  that problem
(\ref{ugb42hv1})--(\ref{ugb48hv1}) is uniquely solvable.

Now let us prove the representation  $\widehat{\widehat
{l(\varphi)}}=l(\hat \varphi).$ Substituting expression
(\ref{mjv1}) for $\hat u$ and $\hat c$ into (\ref{minijv1}) and
taking into consideration relationships
(\ref{ugb42hv1})--(\ref{ugb44hv1}), we obtain
$$
\widehat{\widehat {l(\varphi)}}=(y,\hat
u)_{H_0}+\hat c =(y,Q_0Cp)_{H_0}+\hat c
=\overline{(Cp,Q_0y)_{H_0}}+\hat c
$$
$$
=\overline{<Cp,J_{H_0}Q_0y>_{H_0\times H_0'}}
=\overline{(p,C^*J_{H_0}Q_0y)_{L^2(a,b)}}+\hat c
$$
$$
=\overline{\int_a^bp(t)\overline{C^*J_{H_0}Q_0y(t)}\,dt
}+\hat c =\overline{\int_a^bp(t) \overline{L^+\hat p(t)}\,dt
}+\int_a^b\overline{z(t)} f^{(0)}(t)dt
$$
\begin{equation}\label{tranv1}
 +\sum_{j=1}^m
\overline{S_j^+(z)}\alpha_j^{(0)} +\overline{\int_a^bp(t)
\overline{C^*J_{H_0}Q_0C\hat\varphi(t)}\,dt}.
\end{equation}
Transform the sum of the first three terms on the right-hand side
of this equality using Green's formula (\ref{greenv1}) and
equalities (\ref{llkyfv1})--(\ref{sgb25'fv1}) and
(\ref{ugb46hv1})--(\ref{ugb48hv1}). We have
$$
\overline{\int_a^bp(t) \overline{L^+\hat p(t)}\,dt
}+\int_a^b\overline{z(t)} f^{(0)}(t)dt +\sum_{j=1}^m
\overline{S_j^+(z)}\alpha_j^{(0)}
$$
$$
=\overline{\int_a^b Lp(t) \overline{\hat p(t)}\,dt} +\sum_{j=1}^m
\overline{B_j(p)\overline{S_j^+(\hat p(t))}}
$$
$$
+\int_a^b\overline{z(t)} f^{(0)}(t)dt +\sum_{j=1}^m
\overline{S_j^+(z)}\alpha_j^{(0)}
$$
$$
=\overline{\int_a^b Q^{-1}z(t)\overline{\hat p(t)}\,dt }
+\sum_{j=1}^m \overline{Q_1^{-1}\mathbf
S^+(z)_j\overline{S_j^+(\hat p)}}
$$
$$
+\int_a^b\overline{z(t)} f^{(0)}(t)dt +\sum_{j=1}^m
\overline{S_j^+(z)}\alpha_j^{(0)}
$$
$$
=\overline{\int_a^b Q^{-1}z(t)\overline{\hat p(t)}\,dt }
+(\overline{Q_1^{-1}\mathbf S^+(z),\mathbf S^+(\hat p))_{\mathbb
C^m}}
$$
$$
+\int_a^b\overline{z(t)} f^{(0)}(t)dt +\sum_{j=1}^m \overline {
S_j^+(z)}\alpha_j^{(0)}
$$
$$
=\overline{\int_a^b z(t)\overline{Q^{-1}\hat p(t)}\,dt }
+(\overline{\mathbf S^+(z),Q_1^{-1}\mathbf S^+(\hat p))_{\mathbb
C^m}}
$$
$$
+\int_a^b\overline{z(t)} f^{(0)}(t)dt +\sum_{j=1}^m
\overline{S_j^+(z)}\alpha_j^{(0)}
$$
$$
=\int_a^b \overline{z(t)}Q^{-1}\hat p(t)\,dt
+\sum_{j=1}^m\overline{S_j^+(z)\overline{(Q_1^{-1}\mathbf S^+(\hat
p))_j}}
$$
$$
+\int_a^b\overline{z(t)} f^{(0)}(t)dt +\sum_{j=1}^m
\overline{S_j^+(z)}\alpha_j^{(0)}
$$
$$
=\int_a^b\overline{z(t)} Q^{-1}\hat p(t)dt +\sum_{j=1}^m\overline{
S_j^+(z)}(Q_1^{-1}\mathbf S^+(\hat p))_j
$$
$$
+\int_a^b\overline{z(t)} f^{(0)}(t)dt +\sum_{j=1}^m
\overline{S_j^+(z)}\alpha_j^{(0)}
$$
$$
=\int_a^b\overline{z(t)} (Q^{-1}\hat p(t)+f^{(0)}(t))dt
+\sum_{j=1}^m\overline{ S_j^+(z)}(Q_1^{-1}\mathbf S^+(\hat
p))_j+\alpha_j^{(0)})
$$
$$
=\int_a^bL\hat \varphi(t)\overline{z(t)}dt+\sum_{j=1}^mB_j(\hat
\varphi)\overline{ S_j^+(z)}=\int_a^b\hat
\varphi(t)\overline{L^+z(t)}dt
$$
$$
=l(\hat \varphi)-\int_a^b\hat
\varphi(t)\overline{C^*J_{H_0}Q_0Cp(t)}dt =l(\hat \varphi)-<C\hat
\varphi,J_{H_0}Q_0Cp>_{H_0\times H_0'}
$$
$$
=l(\hat \varphi)-(C\hat \varphi,Q_0Cp)_{H_0}=l(\hat
\varphi)-(Q_0C\hat \varphi,Cp)_{H_0} =l(\hat
\varphi)-\overline{(Cp,Q_0C\hat \varphi)_{H_0}}
$$
$$
=l(\hat \varphi)-\overline{<Cp,J_{H_0}Q_0C\hat \varphi>_{H_0\times
H_0'}} =l(\hat \varphi)-\overline{(p,C^*J_{H_0}Q_0C\hat
\varphi)_{L^2(a,b)}}
$$
$$
=l(\hat
\varphi)-\overline{\int_a^bp(t)\overline{C^*J_{H_0}Q_0C\hat
\varphi(t)}dt }.
$$
The latter equality together with  (\ref{tranv1}) prove the
sought-for representation.
\end{proof}

\begin{predlll} Function $\hat\varphi(t)$ can be taken as an estimate
of solution  $\varphi(t)$ to BVP \eqref{skkv1}, \eqref{skk1v1}
which is observed.
\end{predlll}

\begin{predllll} Statements similar to Theorems {\rm\ref{th3v1}} and {\rm\ref{th4v1}}
can be obtained in the case when  errors $\eta$ in observations
\eqref{skk7v1} are deterministic elements with values in space
$H_0.$
\end{predllll}

As an example, consider the case  $H_0=\left(L^2(a,b)\right)^N$
and the operator $C:$ $L^2(a,b)\to H_0$ in observations
(\ref{skk7v1}) is defined by the equality
$$
C\varphi(t)=\left(\int_a^bK_1(t,\xi)\varphi(\xi)\,d\xi,
\ldots,\int_a^bK_N(t,\xi)\varphi(\xi)\,d\xi \right)^T,
$$
where kernels $K_j\in L^2(a,b)\times L^2(a,b)$ of the integral
operators
$$
C_j\varphi(t):=\int_a^bK_j(t,\xi)\varphi(\xi)\,d\xi, \quad
j=\overline{1,N},
$$
are assumed to be such that the vector-functions
$$
C\varphi_i(t)=\left(\int_a^bK_1(t,\xi)\varphi_i(\xi)\,d\xi,
\ldots,\int_a^bK_N(t,\xi)\varphi_i(\xi)\,d\xi \right)^T\in
\left(L^2(a,b)\right)^N,
$$
$i=\overline{1,n-r},$ are linearly independent. Observations
(\ref{skk7v1}) take the form
$$
y_i(t)=\int_a^bK_i(t,\xi)
\varphi(\xi)\,d\xi+\eta_i(t) ,\quad i=\overline{1,N},
$$
where $\eta(t):=(\eta_1(t),\ldots,\eta_N(t))\in G_1$ is a random
vector process with components $\eta_j(t)$ which are random
processes with zero expectations and finite second moments and
operator $Q_0$ in condition (\ref{skk12v1}) that specifies set
$G_1$ is identified with an $N\times N$ matrix that has elements
$q_{ij}^{(0)}\in C[a,b].$ Consequently, this condition takes the
form $\int_a^b\mbox{Sp}(Q_0(t)\tilde R(t,t))\,dt\leq 1$, in which
$\tilde R(t,s)=[\tilde b_{i,j}(t,s)]_{i,j=1}^N$ denotes the
unknown correlation matrix of the vector process
$\tilde\eta=(\tilde\eta_1(t),\ldots,\tilde\eta_N(t))$ with the
elements $\tilde b_{i,j}(t,s)=\mathbf
M\tilde\eta_i(t)\overline{\tilde\eta_j(s)}$ and
$\mbox{Sp}(Q_0(t)\tilde R(t,t))$ denotes the trace of matrix
$Q_0(t)\tilde R(t,t).$

It is easy to see now that the operator $C^*:$ $(L^2(a,b))^N\to
L^2(a,b)$ adjoint to $C$ is given by the formula
$$
(C^*\psi)(t)=(C^*(\psi_1(\cdot),\ldots,\psi_N(\cdot))^T)(t)=
\sum_{j=1}^N\int_a^b\overline{K_j(\xi,t)}\psi_j(\xi)\,d\xi,
$$
and minimax estimate $\widehat{\widehat {l(\varphi)}}$ has the form
$$
\widehat{\widehat
{l(\varphi)}}=\sum_{i=1}^N\int_a^by_i(t)
\overline{\hat u_i(t)}\,dt+\hat c,
$$
where
$$
\hat u_i(t)=\sum_{j=1}^N q_{ij}^{(0)}(t)\int_a^b
K_j(t,\xi)p(\xi)\,d\xi,\quad i=\overline{1,N}.
$$
Equalities (\ref{llkyfv1}) and (\ref{sgb25'fv1}) become
$$
L^+z(t)=l_0(t)- \int_a^b\tilde K(t,\xi)p(\xi)\,d\xi\quad \mbox{on}
\quad (a,b),\eqno(2.35')
$$
and
$$
\int_a^b\left(l_0(t)-\tilde K(t,\xi)p(\xi)\,d\xi\right)
\overline{\varphi_{i'}(t)}dt=0 , \,\,
i'=\overline{1,n-r},\eqno(2.40')
$$
where
$$
\tilde K(t,\xi)=\sum_{i=1}^N\sum_{j=1}^N
\int_a^b\overline{K_i(t',t)}q_{ij}^{(0)}(t') K_j(t',\xi)\,dt'.
$$

\section{Minimax estimation of functionals from
right-hand sides of equations that enter the problem statements.
Representations for minimax estimates and estimation  errors}
An
estimation problem in question can be formulated as follows: to
find the optimal (in a certain sense) estimate of the value of the
functional
\begin{equation}\label{skk8v1}
l(F)=\int_a^b\overline{l_0(t)}f(t)\,dt
+\sum_{j=1}^m\overline{l_j}\alpha_j,
\end{equation}
from observations of the form
\begin{equation}\label{skkk7v1}
y=C\varphi+\eta
\end{equation}
in the class of estimates
\begin{equation}\label{skk9v1}
\widehat {l(F)}=(y,u)_{H_0}+c,
\end{equation}
linear with respect to observations; here $u$ belongs to Hilbert
space $H_0,$ $c\in \mathbb C,$ $l_0\in L^2(a,b)$ is a given
function, and $l_j\in \mathbb C,$ $j=\overline{1,m}$ are given
numbers; the assumption is that
$F:=(f(\cdot),\alpha)=(f(\cdot),(\alpha_1,\ldots,\alpha_n))\in
G_0$ and the errors  $\eta=\eta(\omega)$ in observations
(\ref{skkk7v1}) belong to $G_1,$ where sets $G_0$ and $G_1$ are
specified, respectively, by (\ref{skk10v1}), (\ref{skk11v1}), and
(\ref{skk12v1}).

\begin{predll}\label{oz2v1}  An estimate $\widehat{\widehat
{l(F)}}=(y,\hat u)_{H_0}+\hat c$ for which an element
$\hat u$ and a constant  $\hat c$
are determined from the condition
$$
\sup_{\tilde F\in G_0,\tilde \eta\in G_1} M|l(\tilde F)-\widehat
{l(F)}|^2\to \inf_{u\in H_0,\,c\in \mathbb C},
$$
where $\widehat {l(\tilde F)} =(\tilde y,u)_{H_0}+c,$ $\tilde
y=C\tilde\varphi+\tilde\eta,$ and $\tilde\varphi$ is any solution
to BVP \eqref{skkv1}, \eqref{skk1v1} at $f(t)=\tilde f(t)$ and
$\alpha_i=\tilde \alpha_i, i=\overline{1,m},$
 will be called a minimax estimate of $l(F).$

The quantity
$
\sigma:=\sup_{\tilde F\in G_0,\tilde \eta\in G_1} \{M|l(\tilde F)-\widehat{\widehat {l(\tilde F)}}|^2\}^{1/2} $
will be called the minimax estimation error of
$l(F)$.
\end{predll}

Using the above results and definitions, formulate and prove the
following statements.
\begin{predl}\label{lemma1v1}
Determination of the minimax estimate of $l(F)$ is
equivalent to the problem of optimal control of the operator
equation system
\begin{equation}\label{gbbk21v1}
L^+z(t;u)=-C^*J_{H_0}u(t)\quad \mbox{on}\quad (a,b),
\end{equation}
\begin{equation}\label{gbbk22v1}
B_j^+(z(\cdot;u))=0 \quad j=\overline{1,2n-m},
\end{equation}
\begin{equation}\label{gbbk23v1}
\int_a^b \!\!\!\!Q^{-1}(l_0(t)+z(t;u))\overline{\psi_i(t)} dt+
(Q_1^{-1}(\mathbf l+\mathbf S^+(z(\cdot;u))),\mathbf
S^+(\psi_i))_{\mathbb C^m}\!\!=0,\,\,\, i=\overline{1,m-r},
\end{equation}
with the cost function
\begin{multline}\label{gbk24'v1}
I(u)=\int_a^bQ^{-1}(l_0(t)+z(t;u))\overline{(l_0(t)+z(t;u))}dt
\\+(Q_1^{-1}(\mathbf l+\mathbf S^+(z(\cdot;u))), \mathbf l+\mathbf
S^+(z(\cdot;u)))_{\mathbb C^m}+ (Q_0^{-1}u,u)_{H_0}\to \inf_{u\in
V},
\end{multline}
where
\begin{equation*}
V=\{u\in H_0: \int_a^b C^*J_{H_0}u(t)
\overline{\varphi_0(t)}dt=0\quad \forall\varphi_0\in N(A_B)\}.
\end{equation*}
\end{predl}
\begin{proof}
Note first that set $V$ is is nonempty because any element $\tilde
u\in H_0$ orthogonal to the subspace spanned over vectors
$\{C\varphi_1,\ldots,$ $C\varphi_{n-r}\}$ belongs to $V$ and at
any fixed  $u\in V$function $z(t;u)$ is uniquely determined from
equations (\ref{gbbk21v1})--(\ref{gbbk23v1}). Indeed, condition
$u\in V$ coincides, by virtue of (\ref{skk6v1}), with the
solvability condition of problem
(\ref{gbbk21v1})--(\ref{gbbk22v1}); let $z_0(t;u)\in W^n_2(a,b)$
be a solution to this problem. Then the function
\begin{equation}\label{gbk24gv1}
z(t;u):=z_0(t;u)+\sum_{i=1}^{m-r}c_i\psi_i(t),
\end{equation}
also satisfies (\ref{gbbk21v1})--(\ref{gbbk22v1}) for any
$c_i,\in\mathbb C^1,$ $i=\overline{1,m-r)}.$ Let us prove that
coefficients  $c_i,$ $i=\overline{1,m-r},$ can be chosen so that
this function would also satisfy condition (\ref{gbbk23v1}).
Substituting expression (\ref{gbk24gv1}) for $z(t;u)$ into
(\ref{gbbk23v1}), we obtain a linear algebraic system of $m-r$
equations with $m-r$ unknowns $c_1,\ldots, c_{m-r}:$
\begin{equation} \label{gbk26gv1}
\sum_{i=1}^{m-r}a_{ij}c_i=b_j(u), \quad j=1,\ldots,m-r;
\end{equation}
its matrix $[a_{ij}]_{i,j=1}^{m-r}$ whose elements $a_{ij}$ are
determined according to (\ref{a_ijv1}) is positive definite, thus
$\det [a_{ij}]\neq 0$ (see p. \pageref{sav1}) and elements
$b_j(u)$ are determined from
\begin{multline}\label{b_jgv1}
b_j(u)=-\int_a^b Q^{-1}(l_0(t)+z_0(t;u))\overline{\psi_j(t)} dt\\-
(Q_1^{-1}(\mathbf l+\mathbf S^+(z_0(\cdot;u))),\mathbf
S^+(\psi_j))_{\mathbb C^m},\quad j=1,\ldots,m-r.
\end{multline}
Therefore this system has unique solution $c_1,\ldots,c_{m-r}$ and
problem (\ref{gbbk21v1})--(\ref{gbbk23v1}) is uniquely solvable.

Next, writing a solution  $\tilde \varphi$ to problem
\eqref{skkv1}, \eqref{skk1v1} in the form $\tilde \varphi=\tilde
\varphi_\bot+\varphi_0,$ where $\tilde\varphi_0$ and $\tilde
\varphi_\bot$ are introduced on p. \pageref{page14v1}, and using
the formula
$$
(\tilde y,u)_{H_0}=(C\tilde \varphi,u)_{H_0}+ (\tilde\eta,u)_{H_0}=(C(\tilde
\varphi_\bot+\varphi_0),u)_{H_0} + (\tilde\eta,u)_{H_0}
$$
$$
=<C(\tilde \varphi_\bot+\varphi_0),J_{H_0}u>_{H_0\times
H_0'}+(\tilde\eta,u)_{H_0} =\int_a^b(\tilde
\varphi_\bot(t)+\varphi_0(t))\overline{C^*J_{H_0}u(t)}\,dt
+(\tilde\eta,u)_{H_0}
$$
$$
 =\int_a^b\tilde
\varphi_\bot(t)\overline{C^*J_{H_0}u(t)}\,dt
+\int_a^b\varphi_0(t)\overline{C^*J_{H_0}u(t)}\,dt
+(\tilde\eta,u)_{H_0},
$$
for arbitrary  $u\in H_0$, we have
$$
l(\tilde F)-\widehat {l(\tilde F)}=\int_a^b\overline{l_0(t)}\tilde
f(t)\,dt +\sum_{j=1}^m\overline{l_j}\tilde \alpha_j-(\tilde y,u)_{H_0}-c
$$
$$
=\int_a^b\overline{l_0(t)}\tilde
f(t)\,dt+\sum_{j=1}^m\overline{l_j}\tilde \alpha_j
-\int_a^b\tilde\varphi_\bot(t)\overline{C^*J_{H_0}u(t)}\,dt
$$
$$
-\int_a^b\varphi_0(t)\overline{C^*J_{H_0}u(t)}\,dt
-(\tilde\eta,u)_{H_0}-c.
$$
The latter implies
$$
M\left|l(\tilde F)-\widehat
{l(\tilde F)}\right|^2 =
\left|\int_a^b\overline{l_0(t)}\tilde
f(t)\,dt+\sum_{j=1}^m\overline{l_j}\tilde \alpha_j \right.
$$
\begin{equation}\label{lttkfv1}
\left.-\int_a^b\tilde \varphi_\bot(t)(t)\overline{C^*J_{H_0}u(t)}\,dt
-\int_a^b\varphi_0(t)\overline{C^*J_{H_0}u(t)}\,dt-c\right|^2
+M|(\tilde\eta,u)_{H_0}|^2.
\end{equation}
Since function $\varphi_0(t)$ under the integral sign of the last
term is an arbitrary element of space $N(A_B),$ quantity $\mathbf
M|l(\tilde F)-\widehat {l(\tilde F)}|^2$ takes all values from
$-\infty$ to $+\infty.$ This quantity is finite if
$$
\int_a^b\varphi_0(t)\overline{C^*J_{H_0}u(t)}\,dt=0,
$$
which is a necessary condition which holds if and only if $u\in
V.$ Assuming now that $u\in V$ and taking into account
(\ref{gbbk21v1})--(\ref{gbbk23v1}) and (\ref{greenv1}), we obtain
the following representation for the expression under the sign of
absolute value in (\ref{lttkfv1})
$$
\int_a^b\overline{l_0(t)}\tilde
f(t)\,dt+\sum_{j=1}^m\overline{l_j}\tilde \alpha_j -\int_a^b\tilde
\varphi_\bot(t)\overline{C^*J_{H_0}u(t)}\,dt -c
$$
$$
=\int_a^b\overline{l_0(t)}\tilde
f(t)\,dt+\sum_{j=1}^m\overline{l_j}\tilde \alpha_j +\int_a^b\tilde
\varphi_\bot(t)\overline{L^+z(t;u)}\,dt-c
$$
$$
=\int_a^b\overline{l_0(t)}\tilde
f(t)\,dt+\sum_{j=1}^m\overline{l_j}\tilde \alpha_j +\int_{a}^b
L\tilde \varphi_\bot(t)\overline{z(t;u)}\, dt+\sum_{j=1}^m B_j(\tilde
\varphi_\bot)\overline{S_j^+(z(\cdot;u))}-c
$$
$$
=\int_a^b\overline{l_0(t)}\tilde
f(t)\,dt+\sum_{j=1}^m\overline{l_j}\tilde \alpha_j +\int_{a}^b
\tilde f(t)\overline{z(t;u)}\, dt+\sum_{j=1}^m \tilde
\alpha_j\overline{S_j^+(z(\cdot;u))}-c
$$
$$
=\int_a^b\overline{(l_0(t)+z(t;u))}\tilde f(t)\,dt+
\sum_{j=1}^m\overline{(l_j+S_j^+(z(\cdot;u)))}\tilde \alpha_j-c
$$
$$
=(\tilde f,l_0+z(\cdot;u))_{L^2(a,b)}+(\tilde \alpha,\mathbf
l+\mathbf S^+(z(\cdot;u)))_{\mathbb C^m}-c.
$$
The latter equality in combination with (\ref{lttkfv1}) yields
$$
\inf_{c \in
\mathbb C}\sup_{\tilde F\in G_0, \tilde \eta\in
G_1} M|l(\tilde F)-\widehat
{l(\tilde F)}|^2=
$$
\begin{equation}\label{exkv1}
=\inf_{c \in \mathbb C}\sup_{\tilde F\in G_0}\left|(\tilde
f,l_0+z(\cdot;u))_{L^2(a,b)}+(\tilde \alpha,\mathbf l+\mathbf
S^+(z(\cdot;u)))_{\mathbb C^m}-c\right|^2+ \sup_{ \tilde \eta\in
G_1}M|(\tilde\eta,u)_{H_0}|^2.
\end{equation}
To calculate the first term on the right-hand side of
(\ref{exkv1}) use the generalized Cauchy--Bunyakovsky inequality
([10], p. 196) and (\ref{skk11v1}) to obtain
$$
\inf_{c \in \mathbb C}\sup_{\tilde F\in G_0,}\left|(\tilde
f,l_0+z(\cdot;u))_{L^2(a,b)}+(\tilde \alpha,\mathbf l+\mathbf
S^+(z(\cdot;u)))_{\mathbb C^m}-c\right|^2
$$
\begin{multline*}
=\inf_{c \in \mathbb C}\sup_{\tilde F\in
G_0,}\left|\overline{(l_0+z(\cdot;u),\tilde
f-f_0)_{L^2(a,b)}}+\overline{(\mathbf l+\mathbf
S^+(z(\cdot;u)),\tilde \alpha-\alpha^{(0)})_{\mathbb
C^m}}\right.\\+ \left.(f_0,l_0+z(\cdot;u))_{L^2(a,b)}+(
\alpha^{(0)},\mathbf l+\mathbf S^+(z(\cdot;u)))_{\mathbb
C^m}-c\right|^2
\end{multline*}
$$
\leq
\left\{(Q^{-1}(l_0+z(\cdot;u)),l_0+z(\cdot;u))_{L^2(a,b)}\right.
$$
$$
\left.+ (Q_1^{-1}(\mathbf l+\mathbf S^+(z(\cdot;u))),\mathbf
l+\mathbf S^+(z(\cdot;u)))_{\mathbb C^m}\right\}
$$
$$
\times \left\{(Q(\tilde f-f^{(0)}),\tilde f-f^{(0)})_{L^2(a,b)}+
(Q_1(\tilde \alpha -\alpha^{(0)}),\tilde \alpha
-\alpha^{(0)})_{\mathbb C^m}\right\}
$$
$$ \leq
\left\{(Q^{-1}(l_0+z(\cdot;u)),l_0+z(\cdot;u))_{L^2(a,b)}\right.
$$
\begin{equation}\label{exk2v1}
\left. + (Q_1^{-1}(\mathbf l+\mathbf S^+(z(\cdot;u))),\mathbf
l+\mathbf S^+(z(\cdot;u)))_{\mathbb C^m}\right\}.
\end{equation}
The direct substitution shows that inequality (\ref{exk2v1}) turns
to equality at $\tilde F=(\tilde f(\cdot),\tilde\alpha)=\tilde
F^{(0)}:=(\tilde f^{(0)}(\cdot),\tilde\alpha^{(0)})=\left(\tilde
f^{(0)}(\cdot),(\tilde\alpha_1^{(0)},\ldots,\right.$ $
\left.\tilde\alpha_m^{(0)})^T\right)\in L^2(a,b)\times \mathbb
C^m,$ where
$$
\tilde f^{(0)}(t):=\frac 1{d}Q^{-1}(l_0(t)+z(t,u))+f_0(t),
$$
$$
\tilde\alpha_i^{(0)}:=\frac 1{d}Q_1^{-1}(\mathbf l+\mathbf
S^+(z(\cdot;u)))_i+\alpha_i^{(0)}, i=\overline{1,m},
$$
$$
\!\!\!\!d\!=\!\!{\Bigl(\!(Q^{-1}(l_0+z(\cdot;u)),l_0\!
+\!z(\cdot;u))_{L^2(a,b)}\!+\! (Q_1^{-1}(\mathbf l\!+\!\mathbf
S^+(z(\cdot;u))),\mathbf l\!+\!\mathbf S^+(z(\cdot;u)))_{\mathbb
C^m}\!\Bigr)\!^{1/2}},
$$
and $Q_1^{-1}(\mathbf l+\mathbf S^+(z(\cdot;u)))_j$ is the $j$th
component of vector $Q_1^{-1}(\mathbf l+\mathbf
S^+(z(\cdot;u)))\in \mathbb C^m.$ Element $\tilde F^{(0)}\in G_0$
because it obviously satisfies condition (\ref{skk11v1}) and from
the following chain of equalities
$$
(\tilde f^{(0)},\psi_0)_{L^2(a,b)}+\sum_{i=1}^m\tilde
\alpha_i^{(0)}\overline{S_i^+(\psi_0)}=
$$
$$
=\Bigl(Q^{-1}(l_0+z(\cdot;u)),l_0 +z(\cdot;u))_{L^2(a,b)}+
(Q_1^{-1}(\mathbf l+\mathbf S^+(z(\cdot;u))),\mathbf l+\mathbf
S^+(z(\cdot;u)))_{\mathbb C^m}\Bigr)^{-1/2}
$$
$$
\times \Bigl(Q^{-1}(l_0+z(\cdot;u)),\psi_0)_{L^2(a,b)}+
\sum_{i=1}^m Q_1^{-1}(\mathbf l+\mathbf
S^+(z(\cdot;u)))_i\overline{S_i^+(\psi_0)}\Bigr)
$$
$$
+(f^{(0)},\psi_0)_{L^2(a,b)}+\sum_{i=1}^m
\alpha^{(0)}_i\overline{S_i^+(\psi_0)}=0 \quad\forall \psi_0\in
N(A^+_{B^+})
$$
it follows that this element also satisfies, in line with
(\ref{gbbk23v1}), condition (\ref{skk10v1}). Therefore,
$$
\inf_{c \in \mathbb C}\sup_{\tilde F\in G_0}\left|(\tilde
f,l_0+z(\cdot;u))_{L^2(a,b)}+(\tilde \alpha,\mathbf l+\mathbf
S^+(z(\cdot;u)))_{\mathbb C^m}-c\right|^2
$$
$$
=\int_a^b Q^{-1}(l_0(t)+z(t;u))\overline{(l_0(t)+z(t;u))}dt
$$
\begin{equation}\label{skk111v1}
+(Q_1^{-1}(\mathbf l+\mathbf S^+(z(\cdot;u))), (\mathbf l+\mathbf
S^+(z(\cdot;u))))_{\mathbb C^m}
\end{equation}
at $c=\int_a^b\overline{(l_0(t)+z(t;u))} f^{(0)}(t)dt
+(\alpha^{(0)},\mathbf l+\mathbf S^+(z(\cdot;u)))_{\mathbb C^m}.$
For the second term on the right-hand side of  (\ref{exkv1}), we
have proved (see p. \pageref{L1v1}) that
\begin{equation}\label{liekv1}
\sup_{\tilde \eta\in G_1}M|(u,\tilde
\eta)_{H_0}|^2=(Q_0^{-1}u,u)_{H_0}.
\end{equation}
The statement of Lemma \ref{lemma1v1} follows now from
(\ref{exkv1}), (\ref{skk111v1}), and (\ref{liekv1}).
\end{proof}

\begin{pred}\label{th1v1}
The minimax estimate of $l(F)$ can be represented as
\begin{equation}\label{miniv1}
\widehat{\widehat {l(F)}}=(y,\hat u)_{H_0}+\hat c,
\end{equation}
where
\begin{equation}\label{repk1v1}
\hat u=Q_0Cp, \quad\hat c=\int_a^b\overline{(l_0(t)+z(t))}
f^{(0)}(t)dt +\sum_{i=1}^m
\overline{(l_i+S_i^+(z))}\alpha_i^{(0)},
\end{equation}
and functions $p(t)$ and $z(t)$ are determined from the operator
equation system
\begin{equation}\label{llkkyv1}
L^+z(t)=-C^*J_{H_0}Q_0Cp(t)\quad \mbox{on}\quad (a,b),
\end{equation}
\begin{equation}\label{llkk1yv1}
B_j^+(z)=0,\quad j=\overline{1,2n-m},
\end{equation}
\begin{equation}\label{sgbk30'v1}
\int_a^bQ^{-1}(l_0(t)+z(t))\overline{\psi_i(t)}\, dt+
(Q_1^{-1}(\mathbf l+\mathbf S^+(z)),\mathbf S^+(\psi_i))_{\mathbb
C^m}=0, \quad i=\overline{1,m-r},
\end{equation}
\begin{equation}\label{sgbk23'v1}
Lp(t)=Q^{-1}(l_0(t)+z(t))\quad \mbox{on}\quad (a,b),
\end{equation}
\begin{equation}\label{sllk1yv1}
B_{j}(p)=Q_1^{-1}(\mathbf{l}+\mathbf S^+(z))_{j},\quad
j=\overline{1,m},
\end{equation}
\begin{equation}\label{sgbk25'v1}
\int_a^bC^*J_{H_0}Q_0Cp(t)\overline{\varphi_{i}(t)}dt=0 ,\,\,
i=\overline{1,n-r},
\end{equation}
where $Q_1^{-1}(\mathbf l+\mathbf S^+(z))_j$ denotes the $j$-а
component of vector $Q_1^{-1}(\mathbf l+\mathbf S^+(z))\in \mathbb
C^m.$ Problem \eqref{llkkyv1}--\eqref{sgbk25'v1} is uniquely
solvable.
The estimation error
$$
\sigma=l(\hat{\hat P})^{1/2},\quad {\mbox{where}}\quad \hat{\hat
P}=(Q^{-1}(l_0+z),(Q_1^{-1}(\mathbf l+\mathbf
S^+(z))_1,\ldots,Q_1^{-1}(\mathbf l+\mathbf S^+(z))_m)^T).
$$
\end{pred}
\begin{proof}
Show first that the solution to optimal control problem
(\ref{gbbk21v1})--(\ref{gbk24'v1}) is reduce to the solution of
system (\ref{llkkyv1})--(\ref{sgbk25'v1}). To this end, note that
the form of functional (\ref{gbk24'v1}) and the fact that $
A^+_{B^+}$ is a Noether operator suggest that there is one and
only one element $\hat u \in V$ at which the minumum of the
functional is attained, $I(\hat u)=\inf_{u \in V}I(u).$ Indeed,
set $u=\bar u+v$ for an arbitrary $u\in V$  where $\bar u$ is a
fixed element from $V$ and $v=u-\bar u.$ Then solution $z(t;u)$ to
BVP (\ref{gbbk21v1})--(\ref{gbbk23v1}) can be represented as
\begin{equation}\label{ieqiv1}
z(t;u)=z(t;\bar u)+\tilde z(t;v),
\end{equation}
where $z(t;\bar u)$ and $\tilde z(t;v)$ are the unique solutions
of this problem at $u=\bar u$ and $u=v,$ $l_0(x)=0,$ and $l_j=0,$
$j=\overline{1,m};$ in addition, if $v$ is an arbitrary element of
$V,$ then $u=\bar u+v$ is also an arbitrary element of this space.

Using expression (\ref{ieqiv1}) for $z(t;u)$ write functional
$I(u)$ in the form
$$
I(u)=\int_a^bQ^{-1}(l_0(t)+z(t;u))\overline{(l_0(t)+z(t;u))}dt
$$
$$
+(Q_1^{-1}(\mathbf l+\mathbf S^+(z(\cdot;u))), \mathbf l+\mathbf
S^+(z(\cdot;u)))_{\mathbb C^m}+ (Q_0^{-1}u,u)_{H_0}
$$
$$
=\int_a^bQ^{-1}(l_0(t)+z(t;\bar u)+\tilde z(t;v)
)\overline{(l_0(t)+z(t;\bar u)+\tilde z(t;v))}dt
$$
$$
+(Q_1^{-1}(\mathbf l+\mathbf S^+(z(\cdot;\bar u)+\tilde
z(\cdot;v))), \mathbf l+\mathbf S^+(z(\cdot;\bar u)+\tilde
z(\cdot;v)))_{\mathbb C^m}+ (Q_0^{-1}(\bar u+v),\bar u+v)_{H_0}
$$
$$
=I(\bar u)+\int_a^bQ^{-1}\tilde z(t;v)\overline{\tilde z(t;v)}dt+
(Q_1^{-1}\mathbf S^+(\tilde z(\cdot;v)), \mathbf S^+(\tilde
z(\cdot;v)))_{\mathbb C^m}+ (Q_0^{-1}v,v)_{H_0}
$$
$$
+2\mbox{Re}\int_a^bQ^{-1}\tilde z(t;v)\overline{(l_0(t)+ z(t;\bar
u))}dt +2\mbox{Re}(Q_1^{-1}\mathbf S^+(\tilde z(\cdot;v)), \mathbf
S^+(l_0(\cdot)+ z(\cdot;\bar u))_{\mathbb C^m}
$$
\begin{equation}\label{repiv1}
+2\mbox{Re}(Q_0^{-1}v,\bar u)_{H_0}=I(\bar u)+\tilde
I(v)+2\mbox{Re}L(v),
\end{equation}
where, due to the linearity and continuity of the mapping,
$$
V\ni v\to(\tilde z(\cdot,v),(S^+_1(\tilde
z(\cdot,v)),\ldots,S^+_m(\tilde z(\cdot,v)))^T)\in L^2(a,b)\times
\mathbb C^m,$$
\begin{equation}\label{rep1iv1}
\tilde I(v)=\int_a^bQ^{-1}\tilde z(t;v)\overline{\tilde z(t;v)}dt+
(Q_1^{-1}\mathbf S^+(\tilde z(\cdot;v)), \mathbf S^+(\tilde
z(\cdot;v)))_{\mathbb C^m}+ (Q_0^{-1}v,v)_{H_0}
\end{equation}
is a linear quadratic functional in $V$ associated with the
continuous semi-bilinear form
\begin{equation}\label{rep2iv1}
\pi(v,w)=\int_a^bQ^{-1}\tilde z(t;v)\overline{\tilde z(t;w)}dt+
(Q_1^{-1}\mathbf S^+(\tilde z(\cdot;v)), \mathbf S^+(\tilde
z(\cdot;w)))_{\mathbb C^m}+ (Q_0^{-1}v,w)_{H_0}
\end{equation}
on $V\times V,;$ this functional satisfies
\begin{equation}\label{rep3iv1}
\tilde I(v)\geq c \|v\|_{H_0}\,\,\forall v\in
V,\,\,c=\mbox{const},
\end{equation}
and
\begin{multline}\label{rep4iv1}
L(v)=\int_a^bQ^{-1}\tilde z(t;v)\overline{(l_0(t)+ z(t;\bar
u))}dt\Bigr.\\ \Bigl. +(Q_1^{-1}\mathbf S^+(\tilde z(\cdot;v)),
\mathbf S^+(l_0(\cdot)+ z(\cdot;\bar u))_{\mathbb
C^m}+(Q_0^{-1}v,\bar u)_{H_0}
\end{multline}
is a linear continuous functional\footnote{The continuity of form
(\ref{rep2iv1}) is proved on p. \pageref{cont2v1}, and the
continuity of linear functional (\ref{rep4iv1}) can be proved
similarly to the case of functional (\ref{rep4v1}). } in $V.$

Consequently (see Remark 1.1 to Theorem 1.1 in \cite{BIBL20}),
there exists the unique element $\hat v\in V$ (depending on $\bar
u$) such that
\begin{equation*}
I_V(\hat v)=\inf_{v \in V}I_V(v), =\inf_{v \in V} I(\bar
u+v)=\inf_{u-\bar u \in V} I(u)= \inf_{u \in \bar u+V}
I(u)=\inf_{u \in V} I(u).
\end{equation*}
Setting $\hat u=\bar u+\hat v$ and taking into account that
$$
I_V(\hat v)=I(\bar u+\hat v)=I(\hat u),
$$
we conclude that there is one and only one element $\hat u\in V$
admitting the representation $\hat u=\bar u+\hat v$ at which the
minimum of functional $I(u)$ is attained for $u\in V$.

Therefore, for any $\tau\in R$ and $v\in V,$
\begin{equation}\label{zk43v1}
\frac{d}{dt}I(\hat u+\tau
v)\Bigl.\Bigr|_{\tau=0}=0\quad\mbox{and}\quad \frac{d}{dt}I(\hat
u+i\tau v)\Bigl.\Bigr|_{\tau=0}=0,
\end{equation}
where $i=\sqrt{-1}.$
Since $z(t;\hat u+\tau v)=z(t;\hat u)+\tau\tilde z(t;v),$ where
$\tilde z(t;v)$ is the unique solution to BVP
(\ref{gbbk21v1})--(\ref{gbbk23v1}) at $u=v$, $l_0=0,$ and $l_j=0,$
$j=\overline{1,m},$ we can use the first relationship in
(\ref{zk43v1}) to obtain
$$
0=\frac 1{2}\frac{d}{d\tau}I(\hat u+\tau v)|_{\tau=0}
$$
$$
 =\lim_{\tau\to 0}\frac
1{2\tau}\Bigl\{\Bigl[(Q^{-1}(l_0+z(\cdot;\hat u+\tau
v)),l_0+z(\cdot;\hat u+\tau v))_{L^2(a,b)}\Bigr.\Bigr.
\Bigl.-(Q^{-1}(l_0+z(\cdot;\hat u)),l_0+z(\cdot;\hat
u))_{L^2(a,b)}\Bigr]
$$
$$
+\Bigl[(Q_1^{-1}(\mathbf l+\mathbf S^+(z(\cdot;\hat u+\tau
v))),\mathbf l+\mathbf S^+(z(\cdot;\hat u+\tau v)))_{\mathbb
C^m}-\Bigr.
\Bigl.(Q_1^{-1}(\mathbf l+\mathbf S^+(z(\cdot;\hat
u))),\mathbf l+\mathbf S^+(z(\cdot;\hat u)))_{\mathbb C^m}\Bigr]
$$
$$
+\left[Q^{-1}_0(\hat u+\tau v),\hat u+\tau v)_{H_0}-(Q^{-1}_0\hat
u,\hat u)_{H_0}\right]\Bigr\}
$$
$$
\!\!\!=\!\mbox{Re}\{(Q^{-1}(l_0+z(\cdot;\hat u)),\tilde z(\cdot;v
))_{L^2(a,b)}
$$
$$
+(Q_1^{-1}(\mathbf l+\mathbf S^+(z(\cdot;\hat
u))),\mathbf S^+(\tilde z(\cdot;v)))_{\mathbb
C^m}\!+\!(Q_0^{-1}\hat u,v)_{H_0}\}.
$$
Next, since $z(t;\hat u+i\tau v)=z(t;\hat u)+i\tau\tilde z(t;v),$
we have
$$
0=\frac 1{2}\frac{d}{d\tau}I(\hat u+i\tau v)|_{\tau=0}
$$
$$
\!\!\!=\!\mbox{Im}\{(Q^{-1}(l_0+z(\cdot;\hat u)),\tilde z(\cdot;v
))_{L^2(a,b)}
$$
$$
+(Q_1^{-1}(\mathbf l+\mathbf S^+(z(\cdot;\hat u))),\mathbf
S^+(\tilde z(\cdot;v)))_{\mathbb C^m}\!+\!(Q_0^{-1}\hat
u,v)_{H_0}\},
$$
which yields
$$
(Q^{-1}(l_0+z(\cdot;\hat u)),\tilde z(\cdot;v
))_{L^2(a,b)}
$$
\begin{equation}\label{vckv1}
+\sum_{j=1}^m Q_1^{-1}(\mathbf{l}+\mathbf S^+(z(\cdot;\hat
u)))_{j}\overline{\mathbf S_j^+(\tilde z(\cdot;v))}+(Q_0^{-1}\hat
u,v)_{H_0}=0.
\end{equation}
Let $p(t)$ be a solution\footnote{Formula (\ref{sgbk30'v1})
coincides with the solvability condition for this problem by
virtue of (\ref{skk5v1}).} to the BVP
\begin{equation}\label{vsgbk23'gv1}
Lp(t)=Q^{-1}(l_0(t)+z(t;\hat u))\quad \mbox{on}\quad (a,b),
\end{equation}
\begin{equation}\label{vsllk1ykgv1}
B_{j}(p)=Q_1^{-1}(\mathbf{l}+\mathbf S^+(z(\cdot,\hat
u)))_{j},\quad j=\overline{1,m}.
\end{equation}
Then, taking into consideration Green's formula (\ref{greenv1}),
we can transform the sum of the first two terms on the left-hand
side of (\ref{vckv1}) as follows:
$$
(Q^{-1}(l_0+z(\cdot;\hat u)),\tilde z(\cdot;v
))_{L^2(a,b)}+\sum_{j=1}^m Q_1^{-1}(\mathbf{l}+\mathbf
S^+(z(\cdot;\hat u)))_{j}\overline{\mathbf S_j^+(\tilde
z(\cdot;v))}
$$
$$
=\int_a^b Lp(t)\overline{\tilde z(t;v)}\,dt+\sum_{j=1}^m B_{j}(p)
\overline{\mathbf S_j^+(\tilde z(\cdot;v))}=\int_a^b
p(t)\overline{L^+\tilde z(t;v)}\,dt
$$
$$
=-\int_a^b p(t)\overline{C^*J_{H_0}v(t)}\,dt
=-<Cp,J_{H_0}v>_{H_0\times H_0'}=-(Cp,v)_{H_0}.
$$
From the latter equality and relationship (\ref{vckv1}), it
follows that
\begin{equation}\label{gfklgv1}
(Q_0^{-1}\hat u-Cp,v)_{H_0}=0
\end{equation}
for any  $v\in V.$ Next, repeating the reasoning contained in the
proof of Theorem 2.1 on p. \pageref{reppv1}, we see that among all
solutions to problem (\ref{vsgbk23'gv1}), (\ref{vsllk1ykgv1}),
there is one and only one solution $p(t)$ for which
\begin{equation}\label{gfk2gv1}
Q_0^{-1}\hat u-Cp\in V.
\end{equation}

Setting in (\ref{gfklgv1}) $v=Q_0^{-1}\hat u-Cp$ we obtain
$Q_0^{-1}\hat u-Cp=0,$ so that $\hat u=Q_0Cp.$ Substituting this
into the equality $\int_a^b C^*J_{H_0}\hat u(t)
\overline{\varphi_0(t)}dt=0$ and denoting $z(t;\hat u)=:z(t),$ we
see that $z(t)$ and $p(t)$ satisfy system
(\ref{llkkyv1})--(\ref{sgbk25'v1}); the unique solvability of this
system follows from the uniqueness of element $\hat u.$

Show now that $\sigma(\hat u,\hat c)=l(\hat{\hat P})^{1/2};$ then
the estimate $\sigma(F)\leq l(\hat{\hat P})^{1/2}$ will be proved.
Substituting $\hat u=Q_0Cp$ into the expression for $I(\hat u)$
and taking into account that $z(t)=z(t;\hat u),$ we have
$$
I(\hat u)=\int_a^b
Q^{-1}(l_0(t)+z(t))\overline{(l_0(t)+z(t))}\,dt+(Q_1^{-1}(\mathbf
l+\mathbf S^+(z)),\mathbf l+\mathbf S^+(z))_{C^m}
$$
$$
+(Cp,Q_0Cp)_{H_0}=\int_a^bLp(t)\overline{z(t)}\,dt+
\int_a^b\overline{l_0(t)}Q^{-1}(l_0(t)+z(t))\,dt
$$
$$
+ \sum_{j=1}^mQ_1^{-1}(\mathbf l+\mathbf
S^+(z))_j(\overline{l_j+S_j^+(z)})+(Cp,Q_0Cp)_{H_0}
$$
$$
=\int_a^bLp(t)\overline{z(t)}\,dt
+\int_a^b\overline{l_0(t)}Q^{-1}(l_0(t)+z(t))\,dt
$$
$$
+\sum_{j=1}^m
B_j(p)\overline{S_j^+(z)}+\sum_{j=1}^m\overline{l_j}Q_1^{-1}(\mathbf
l+\mathbf S^+(z))_j+(Cp,Q_0Cp)_{H_0}
$$
$$
=\int_a^bp(t)\overline{L^+(t)}\,dt+\sum_{j=1}^{2n-m}S_j(p)
\overline{B_j^+(z)}+\int_a^b
\overline{l_0(t)}Q^{-1}(l_0(t)+z(t))\,dt
$$
\begin{equation}\label{w1v1}
+\sum_{j=1}^m\overline{l_j}Q_1^{-1}(\mathbf l+\mathbf
S^+(z))_j+(Cp,Q_0Cp)_{H_0}.
\end{equation}
For the first term in (\ref{w1v1}) we have
$$
\int_a^bp(t)\overline{L^+(t)}\,dt
=-\int_a^bp(t)(\overline{C^*I_{H_0}Q_0Cp)(t)})\,dt
=-<Cp,I_{H_0}Q_0Cp>_{H_0\times H_0'}.
$$
From the latter equality and (\ref{w1v1}), it follows that $I(\hat
u)=l(\hat{\hat P}),$ where
$$
\hat{\hat P}=(Q^{-1}(l_0+z),(Q_1^{-1}(\mathbf l+\mathbf
S^+(z))_1,\ldots,Q_1^{-1}(\mathbf l+\mathbf S^+(z))_m)^T).
$$
The proof of the theorem is completed.
\end{proof}

Another representation for the minimax estimate is given by the
following theorem.
\begin{pred}\label{th2v1}
The minimax estimate of $l(F)$ has the form
\begin{equation}\label{reprkv1}
\widehat{\widehat {l(F)}}=l(\hat F),
\end{equation}
where
$$
\hat F=(\hat f(\cdot),(\hat \alpha_1,\ldots, \hat
\alpha_m)^T), \,\,\hat f(t)=Q^{-1}\hat p(t)+f^{(0)}(t),\,\,$$
$$\hat \alpha_j=Q_1^{-1}\mathbf S^+(\hat p)_j+\alpha_j^{(0)},\,\,
j=\overline{1,m},
$$
and function $\hat p$ is determined from the solution to the
problem
\begin{equation}\label{ugbk42v1}
L^+\hat p(t)=C^*J_{H_0}Q_0(y-C\hat\varphi)(t)\quad
\mbox{on}\quad (a,b),
\end{equation}
\begin{equation}\label{ugbk422v1}
B_j^+(\hat p)=0,\quad j=\overline{1,2n-m},
\end{equation}
\begin{equation}\label{ugbk44v1}
\int_a^bQ^{-1}\hat p(t)\overline{\psi_i(t)}\, dt+ (Q_1^{-1}\mathbf
S^+(\hat p),\mathbf S^+(\psi_i))_{\mathbb C^m }=0, \quad
i=\overline{1,m-r},
\end{equation}
\begin{equation}\label{ugbk46v1}
L\hat \varphi(t)=Q^{-1}\hat p(t)+f^{(0)}(t)\quad \mbox{on}\quad
(a,b),
\end{equation}
\begin{equation}\label{usllk1ykv1}
B_{j}(\hat \varphi)=Q_1^{-1}\mathbf S^+(\hat
p)_{j}+\alpha_{j}^{(0)},\quad
 j=\overline{1,m},
\end{equation}
\begin{equation}\label{ugbk48v1}
\int_a^bC^*J_{H_0}Q_0\left(y-C\hat \varphi\right)(t)
\overline{\varphi_{i}(t)}\,dt=0,\quad i=\overline{1,n-r},
\end{equation}
where $\mathbf S^+(\hat p):=(S^+_1(\hat p),\ldots,$ $S^+_m(\hat
p))^T\in \mathbb C^m$ is the vector with components $S_j^+(\hat
p),$ $j=\overline{1,m}.$
Problem \eqref{ugbk42v1}--\eqref{ugbk48v1} is uniquely solvable.
\end{pred}
\begin{proof}
Introduce the problem of optimal control of the equation system
\begin{equation}\label{gbbkk21v1}
L^+\hat
p(t;u)=-(C^*J_{H_0}u)(t)+(C^*J_{H_0}Q_0(y)(t)\quad
\mbox{on}\quad (a,b),
\end{equation}
\begin{equation}\label{gbbkk22v1}
 B_j^+(\hat p(\cdot;u))=0\quad  (j=1,\ldots,m),
\end{equation}
\begin{equation}\label{gbbkk23v1}
\int_a^b Q^{-1}\hat p(t;u)\overline{\psi_i(t)} dt+
(Q_1^{-1}\mathbf S^+(\hat p(\cdot;u)),\mathbf
S^+(\psi_i))_{\mathbb C^m}=0,\quad i=\overline{1,m-r},
\end{equation}
with the cost function
\begin{multline}\label{gbkk24'v1}
I(u)=\int_a^b Q^{-1}(\hat p(t;u)+Qf^{(0)}(t))\overline{ (\hat
p(t;u)+Qf^{(0)}(t))}dt \\+(Q_1^{-1}(\mathbf S^+(\hat
p(\cdot;u))+Q_1\alpha^{(0)}), \mathbf S^+(\hat
p(\cdot;u))+Q_1\alpha^{(0)})_{\mathbb C^m}+
(Q_0^{-1}u,u)_{H_0}\!\to \inf_{u\in U},
\end{multline}
where
$$
U=\{u\in H_0: \int_a^b(C^*J_{H_0}Q_0(y(t)
-C^*J_{H_0}u(t)) \overline{\varphi_0(t)}dt=0 \}
$$
for any solutions $\varphi_0(t)$ of homogeneous problem
(\ref{skkv1}), (\ref{skk1v1}).

The form of functional $I(u)$ and reasoning contained in the proof
of Theorem \ref{th1v1} suggest the existence of unique element
$\hat u\in U$ such that
$$
I(\hat u)= \inf_{u\in U} I(u).
$$
Next, funding $\hat \varphi(t)$ as the unique solution to the
problem
$$
L\hat \varphi(t)=Q^{-1}\hat p(t;\hat u)+f^{(0)}(t)\quad
\mbox{on}\quad (a,b),
$$
$$
B_{j}(\hat \varphi)=Q_1^{-1}\mathbf S^+(\hat p(\cdot;\hat
u)_{j}+\alpha_{j}^{(0)},\quad
 j=\overline{1,m},
$$
$$
\int_a^bC^*J_{H_0}Q_0\left(y-C\hat \varphi\right)(t)
\overline{\varphi_{i}(t)}\,dt=0,\quad i=\overline{1,n-r},
$$
and repeating the proof of Theorem \ref{th1v1},  we conclude,
taking into consideration the notation  $\hat p(t)=\hat p(t;\hat
u),$ that problem (\ref{ugbk42v1})--(\ref{ugbk48v1}) is uniquely
solvable.

Let us prove the validity of the representation $\widehat{\widehat
{l(F)}}=l(\hat F).$ Substituting expressions (\ref{repk1v1}) for
$\hat u$ and $\hat c$ into (\ref{miniv1}) and using
(\ref{ugbk42v1})--(\ref{ugbk44v1}), we obtain
$$
\widehat{\widehat {l(F)}}=(y,\hat u)_{H_0}+\hat c
=(y,Q_0Cp)_{H_0}+\hat c
=\overline{(Cp,Q_0y)_{H_0}}+\hat c
$$
$$
=\overline{<Cp,J_{H_0}Q_0y>_{H_0\times H_0'}}
=\overline{(p,C^*J_{H_0}Q_0y)_{L^2(a,b)}}+\hat c
$$
$$
=\overline{\int_a^bp(t)\overline{C^*J_{H_0}Q_0y(t)}\,dt
}+\hat c =\overline{\int_a^bp(t) \overline{L^+\hat p(t)}\,dt
}+\int_a^b\overline{(l_0(t)+z(t))} f^{(0)}(t)dt
$$
\begin{equation}\label{trankv1}
 +\sum_{j=1}^m
\overline{(l_j+S_j^+(z))}\alpha_j^{(0)} +\overline{\int_a^bp(t)
\overline{C^*J_{H_0}Q_0C\hat\varphi(t)}\,dt}.
\end{equation}
Transform the sum of the first three terms on the right-hand side
of this equality using Green's formula (\ref{greenv1}) and taking
into account equalities (\ref{llkkyv1})--(\ref{sgbk25'v1}) and
(\ref{ugbk46v1})--(\ref{ugbk48v1}). As a result, we obtain
$$
\overline{\int_a^bp(t) \overline{L^+\hat p(t)}\,dt
}+\int_a^b\overline{(l_0(t)+z(t))} f^{(0)}(t)dt +\sum_{j=1}^m
\overline{(l_j+S_j^+(z))}\alpha_j^{(0)}
$$
$$
=\overline{\int_a^b Lp(t) \overline{\hat p(t)}\,dt} +\sum_{j=1}^m
\overline{B_j(p)\overline{S_j^+(\hat p(t))}}
$$
$$
+\int_a^b\overline{(l_0(t)+z(t))} f^{(0)}(t)dt +\sum_{j=1}^m
\overline{(l_j+S_j^+(z))}\alpha_j^{(0)}
$$
$$
=\overline{\int_a^b Q^{-1}(l_0(t)+z(t))\overline{\hat p(t)}\,dt }
+\sum_{j=1}^m \overline{Q_1^{-1}(\mathbf{l}+\mathbf
S^+(z))_j\overline{S_j^+(\hat p)}}
$$
$$
+\int_a^b\overline{(l_0(t)+z(t))} f^{(0)}(t)dt +\sum_{j=1}^m
\overline{(l_j+S_j^+(z))}\alpha_j^{(0)}
$$
$$
=\overline{\int_a^b Q^{-1}(l_0(t)+z(t))\overline{\hat p(t)}\,dt }
+(\overline{Q_1^{-1}(\mathbf{l}+\mathbf S^+(z),\mathbf S^+(\hat
p))_{\mathbb C^m}}
$$
$$
+\int_a^b\overline{(l_0(t)+z(t))} f^{(0)}(t)dt +\sum_{j=1}^m
\overline{(l_j+S_j^+(z))}\alpha_j^{(0)}
$$
$$
=\overline{\int_a^b (l_0(t)+z(t))\overline{Q^{-1}\hat p(t)}\,dt }
+(\overline{\mathbf{l}+\mathbf S^+(z),Q_1^{-1}\mathbf S^+(\hat
p))_{\mathbb C^m}}
$$
$$
+\int_a^b\overline{(l_0(t)+z(t))} f^{(0)}(t)dt +\sum_{j=1}^m
\overline{(l_j+S_j^+(z))}\alpha_j^{(0)}
$$
$$
=\int_a^b \overline{(l_0(t)+z(t))}Q^{-1}\hat p(t)\,dt
+\sum_{j=1}^m\overline{(l_j+ S_j^+(z))\overline{(Q_1^{-1}\mathbf
S^+(\hat p)_)j}}
$$
$$
+\int_a^b\overline{(l_0(t)+z(t))} f^{(0)}(t)dt +\sum_{j=1}^m
\overline{(l_j+S_j^+(z))}\alpha_j^{(0)}
$$
$$
=l(\hat F)+\int_a^b\overline{z(t)} Q^{-1}\hat p(t)dt
+\sum_{j=1}^m\overline{ S_j^+(z)}(Q_1^{-1}\mathbf S^+(\hat p))_j
$$
$$
+\int_a^b\overline{z(t)} f^{(0)}(t)dt +\sum_{j=1}^m
\overline{S_j^+(z)}\alpha_j^{(0)}
$$
$$
=l(\hat F)+\int_a^b\overline{z(t)} (Q^{-1}\hat p(t)+f^{(0)}(t))dt
+\sum_{j=1}^m\overline{ S_j^+(z)}(Q_1^{-1}\mathbf S^+(\hat
p))_j+\alpha_j^{(0)})
$$
$$
=l(\hat F)+\int_a^bL\hat
\varphi(t)\overline{z(t)}dt+\sum_{j=1}^mB_j(\hat
\varphi)\overline{ S_j^+(z)}=l(\hat F)+\int_a^b\hat
\varphi(t)\overline{L^+z(t)}dt
$$
$$
=l(\hat F)-\int_a^b\hat \varphi(t)\overline{C^*J_{H_0}Q_0Cp(t)}dt
=l(\hat F)-<C\hat \varphi,J_{H_0}Q_0Cp>_{H_0\times H_0'}
$$
$$
=l(\hat F)-(C\hat \varphi,Q_0Cp)_{H_0}=l(\hat F)-(Q_0C\hat
\varphi,Cp)_{H_0} =l(\hat F)-\overline{(Cp,Q_0C\hat
\varphi)_{H_0}}
$$
$$
=l(\hat F)-\overline{<Cp,J_{H_0}Q_0C\hat \varphi>_{H_0\times
H_0'}} =l(\hat F)-\overline{(p,C^*J_{H_0}Q_0C\hat
\varphi)_{L^2(a,b)}}
$$
$$
=l(\hat F)-\overline{\int_a^bp(t)\overline{C^*J_{H_0}Q_0C\hat
\varphi(t)}dt }.
$$
The resulting relationships together with equality (\ref{trankv1})
yield the sought-for representation.
\end{proof}
\begin{predlll} The function $\hat f=Q^{-1}\hat p+f^{(0)}$ and
numbers $\hat\alpha_j= Q_1^{-1}\mathbf S^+(\hat
p)_j+\alpha_j^{(0)}$ can be taken as estimates of the right-hand
sides  $f$ and $\alpha_j$ $(j=\overline{1,m})$ of equalities
\eqref{skkv1} and \eqref{skk1v1}, respectively.
\end{predlll}

\renewcommand{\refname}
{\large \bf  References}


\begin{thebibliography}{99}

\bibitem{BIBLAtk} {
Atkinson, F., {\it Discrete and Coninuous Boundary Value
Problems}, Academic Press, NY, 1964. }

\bibitem{BIBL317} Korolyuk, V.S. {\it et. al.}, {\it Handbook on Probability Theory
and Mathematical Statistics}, Moscow, Nauka, 1989.

\bibitem{BIBL20} Lions, J.,
{\it  Controle optimal de syst\'{e}mes gouvern\'{e}s par des
\'{e}quations aux d\'{e}riv\'{e}es partielles}, Dunod, Paris,
1968.

\bibitem{BIBLNaym} Naimark, M.A., {\it  Linear Differential Operators},  Moscow, Nauka,
1969.

\bibitem{BIBL202}
{ {\em Nakonechnyi, O.G.,\/} Minimax Estimates in Systems with
Distributed Parameters, Preprint 79 Acad. Sci. USSR, Inst.
Cybernetics, Kyiv,  1979, 55 p.}

\bibitem{BIBL366}  Hutson, V.,  Pym, J.,  Cloud, M., {\it Applications of Functional
Analysis and Operator Theory}, 2nd Ed., Mathematics in Science and
Engineering,  Vol. 200, Elsevier Science, 2005.

\bibitem{BIBLKUR} Kurosh A.G., {\it Course of Higher Algebra}, Moscow, Nauka, 1971.

\bibitem{BIBLFED} Fedoryuk, M.V.,
{\it  Ordinary Differential Equations}, Moscow, Nauka, 1985.

\bibitem{BIBL5a} { {\em Kirichenko, N.F., Nakonechnyi, O.G.,\/}
A Minimax Approach to Recurrent Estimation of the States of Linear
Dynamical Systems,  Kibernetika, no 4, 1977, pp. 52-55.}

\bibitem{BIBL12} Krasovskii, N.N., {\it Theory of Motion Control}, Moscow, Nauka, 1968.

\bibitem{Krein} Krein S.G., {\it Linear Equations in the Banach Space}, Moscow, Nauka,  1971.


\bibitem{BIBL13} Kurzhanskii, A.B., {\it Control and Observation under Uncertainties}, Moscow, Nauka, 1977.

\bibitem{BIBL14} Kurzhanski, A.B., Valyi I.,
{\it  Ellipsoidal Calculus for Estimation and Control},
Birkhauser, Boston, 1997.

\end{thebibliography}
\end{document}